\documentclass[twoside,11pt]{article}

\usepackage[preprint]{jmlr2e}

\usepackage{lastpage} 
\usepackage{setspace} 
\usepackage{amsmath,amssymb} 
\usepackage{physics} 
\usepackage{bm} 
\usepackage{pifont} 
\usepackage{booktabs,makecell} 
\usepackage{caption} 
\usepackage[ruled]{algorithm2e} 
\usepackage{xcolor} 

\makeatletter
\newcommand{\mylabel}[2]{%
	\begingroup
	\edef\@currentlabel{\noexpand\textnormal{#2}}%
	\label{#1}%
	\endgroup
}
\makeatother

\newcommand{\namedthmhead}[1]{\trivlist\item[\hskip\labelsep{\bfseries #1}]\itshape}

\newenvironment{asP}{\namedthmhead{Assumption (P)}}{\endtrivlist}
\newenvironment{asSg}{\namedthmhead{Assumption (Sg)}}{\endtrivlist}
\newenvironment{asSgN}{\namedthmhead{Assumption (SgN)}}{\endtrivlist}
\newenvironment{asSgNp}{\namedthmhead{Assumption (SgN$^+$)}}{\endtrivlist}
\newenvironment{asT}{\namedthmhead{Assumption (T)}}{\endtrivlist}

\makeatletter
\renewenvironment{proof}[1][Proof]{\par\addvspace{\medskipamount}\noindent{\bf #1\ }}{\hfill\BlackBox\\[2mm]}
\newenvironment{proofof}[1][]{%
	\def\@proofofarg{#1}%
	\par\addvspace{\medskipamount}\noindent{\bf Proof\ifx\@proofofarg\@empty\else\ of #1\fi\ }%
}{\hfill\BlackBox\\[2mm]}
\makeatother

\newcommand{\di}{\displaystyle} 
\newcommand{\MSE}{\mathrm{MSE}} 
\newcommand{\D}{\mathrm{d}} 
\newcommand{\Le}{\leqslant} 
\newcommand{\Ge}{\geqslant} 
\newcommand{\Prec}{\preccurlyeq} 
\newcommand{\Succ}{\succcurlyeq} 

\newcommand{\revised}[1]{#1}   
\newcommand{\reviewed}[1]{#1} 

\newcommand{\ep}{\varepsilon}
\newcommand{\Dx}{\delta{\bm X}_k}
\newcommand{\Dv}{\delta{\bm V}_k}
\newcommand{\Dmv}{\delta_m{\bm V}_k}
\newcommand{\Dhv}{\delta_h{\bm V}_k}
\newcommand{\Dy}{\delta{\bm Y}_k}

\usepackage{todonotes}

\jmlrheading{XX}{2025}{1-\pageref{LastPage}}{10/2025}{XX/2025}{25-XXXX}{Jianfeng Lu, Xuda Ye, and Zhennan Zhou}

\ShortHeadings{MSE Analysis of Stochastic Gradient Sampling Algorithms}{Lu, Ye, and Zhou}
\firstpageno{1}

\begin{document}
	
	\title{Mean Square Error Analysis of Stochastic Gradient and Variance-Reduced Sampling Algorithms}
	
	\author{\name Jianfeng Lu \email jianfeng@math.duke.edu \\
		\addr Departments of Mathematics, Physics, and Chemistry\\
		Duke University\\
		Durham, NC 27708, USA
		\AND
		\name Xuda Ye \email ye481@purdue.edu \\
		\addr Department of Mathematics\\
		Purdue University\\
		West Lafayette, IN 47907, USA
		\AND
		\name Zhennan Zhou \email zhouzhennan@westlake.edu.cn\\
		\addr Institute for Theoretical Sciences\\
		Westlake University \\
		Hangzhou, 310030, China
	}
	\editor{TBD}
	
	\maketitle
	
    \begin{abstract}%
        This paper considers mean square error (MSE) analysis for stochastic gradient sampling algorithms applied to underdamped Langevin dynamics under a global convexity assumption. A novel discrete Poisson equation framework is developed to bound the time-averaged sampling error. For the Stochastic Gradient UBU (SG-UBU) sampler, we derive an explicit MSE bound and establish that the numerical bias exhibits first-order convergence with respect to the step size $h$, with the leading error coefficient proportional to the variance of the stochastic gradient. The analysis is further extended to variance-reduced algorithms for finite-sum potentials, specifically the SVRG-UBU and SAGA-UBU methods. For these algorithms, we identify a phase transition phenomenon whereby the convergence rate of the numerical bias shifts from first to second order as the step size decreases below a critical threshold. Theoretical findings are validated by numerical experiments. In addition, \reviewed{we compare the computational cost that mini-batch SG-UBU and SVRG-UBU require to reach a prescribed accuracy, which indicates when variance reduction is worthwhile}.
    \end{abstract}

	
\section{Introduction}
\label{section: introduction}

    Sampling from complicated probability distributions is a fundamental task in computational physics and data sciences. In many problems of interest, the potential function (i.e., the negative log-density or log-likelihood) has a natural finite-sum structure. A direct example comes from Bayesian inference: the posterior distribution accumulates log-likelihood contributions from a large dataset \citep{gelman1995bayesian}.  This finite-sum structure motivates stochastic gradient sampling methods, pioneered by Stochastic Gradient Langevin Dynamics (SGLD) and related MCMC algorithms \citep{welling2011bayesian,chen2014stochastic}. At each step, a mini-batch gradient, an unbiased estimator computed from a small subset of the data, provides an approximation of the exact gradient. Because of the unbiasedness, the errors introduced into the drift term average out over long simulation horizons; consequently, as the integrator step size vanishes, the dynamics converge to the correct target distribution and yield reliable samples of the posterior. Building on these properties, SGLD has found wide applications in Bayesian deep learning and uncertainty quantification \citep{welling2011bayesian,li2016preconditioned}; scalable latent-variable models such as topic models and probabilistic matrix factorization \citep{patterson2013stochastic,chen2014stochastic}; and online or streaming Bayesian inference for large datasets \citep{ahn2012bayesian}.
	
	The necessity of stochastic gradients for sampling extends beyond the data sub-sampling context to problems where the target log-probability gradient is intrinsically stochastic. This situation frequently arises in Bayesian inference for complex generative models, particularly those defined implicitly through stochastic simulators, which render the likelihood function intractable \citep{cranmer2020frontier}. In such models, common across fields such as systems biology, physics, and econometrics, the gradient of the log-likelihood cannot be computed analytically and must instead be approximated via Monte Carlo estimation, often through repeated simulator runs. As a result, any gradient-based sampler targeting the posterior must operate with noisy gradient estimates. This challenge lies at the core of likelihood-free inference (also known as Approximate Bayesian Computation) and pseudo-marginal methods, where the central task is to perform posterior sampling in the presence of intrinsically noisy gradient information \citep{andrieu2009pseudo,meeds2014gps}.
	
We formulate the stochastic gradient sampling problem as follows. Let $\pi(\bm{x}) \propto e^{-U(\bm{x})}$ be the target distribution in $\mathbb{R}^d$. The full gradient $\nabla U(\bm{x})$ is assumed to be computationally intractable in $\mathcal O(1)$ time. Instead, we assume cheap access to a stochastic gradient $b(\bm{x},\theta)$ for some random variable $\theta \sim \mathbb{P}_\theta$, and $b(\bm{x},\theta)$ is an unbiased estimator of the full gradient:
\begin{equation*}
	\mathbb{E}_\theta\big[b(\bm{x},\theta)\big] = \nabla U(\bm{x}), \qquad \forall \bm{x} \in \mathbb{R}^d.
\end{equation*}
Two common forms of the unbiased stochastic gradient include:

\paragraph{Additive noise form.} The gradient estimator is given by
\begin{equation}
	b(\bm{x},\theta) = \nabla U(\bm{x}) + \sigma \theta,
	\qquad \theta \sim \mathcal{N}(\bm{0},\bm{I}_d),
\end{equation}
with $\sigma > 0$ being a constant noise level.

\paragraph{Finite-sum form.} This setting assumes that the potential is an average of $N$ components,
\begin{equation}
	U(\bm{x}) = \frac{1}{N} \sum_{i=1}^N U_i(\bm{x}),
	\label{finite-sum potential}
\end{equation}
where the stochastic gradient is a mini-batch estimate over a random subset $\theta \subset \{1, \dots, N\}$ of size $p$, with the batch size $p$ being a key tunable parameter:
\begin{equation}
	b(\bm{x},\theta) = \frac{1}{p} \sum_{i \in \theta} \nabla U_i(\bm{x}).
	\label{stochastic gradient}
\end{equation}
A stochastic gradient sampling algorithm generates a sequence of samples $(\bm{X}_k)_{k=0}^\infty$ by applying a numerical integrator with a fixed step size $h > 0$, driven by the stochastic gradient $b(\bm{x},\theta)$. For a sufficiently small step size, the distribution of the iterate $\bm{X}_k$ is expected to provide a close approximation to the target distribution $\pi(\bm{x})$ as $k \to \infty$.

This paper provides a rigorous analysis of the error for stochastic gradient sampling algorithms. The analysis of such methods is considerably more challenging than their full gradient counterparts, as it must account for the intricate interplay between the noise from the stochastic gradients and the systematic error from the numerical discretization. We focus on a practical error metric by first specifying a test function $f(\bm{x})$ and the target statistical average $\pi(f) := \int_{\mathbb{R}^d} f(\bm{x})\pi(\bm{x})\D\bm{x}$. For the numerical solution $(\bm{X}_k)_{k=0}^\infty$, the sampling error after $K$ steps is quantified by the mean square error (MSE):
\begin{equation}
	\mathrm{MSE}_K = \mathbb{E}\Bigg[\bigg(\frac{1}{K} \sum_{k=0}^{K-1} f(\bm{X}_k) - \pi(f)\bigg)^2\Bigg].
    \label{MSE}
\end{equation}
Unlike metrics that measure distributional discrepancy (e.g., Wasserstein distance or Kullback--Leibler (KL) divergence), the MSE characterizes the error and variance of the time average obtained from a single numerical trajectory, which more accurately reflects the practice.

The MSE can be conceptually decomposed into three primary sources: (i) the \textit{stochastic fluctuation} arising from the finite sample size, (ii) the \textit{stochastic gradient error} introduced by gradient noise, and (iii) the \textit{discretization error} originating from the numerical integrator. The stochastic fluctuation component is an intrinsic feature of Monte Carlo sampling algorithms and exhibits a conventional decay rate of $\mathcal O(1/K)$ as the number of sampling steps $K$ approaches infinity.
The discretization error can, in principle, be systematically reduced by employing higher-order integrators. For underdamped Langevin dynamics, second-order schemes like BAOAB \citep{leimkuhler2015molecular} and UBU \citep{sanz2021wasserstein}, and third-order schemes like the randomized midpoint integrator \citep{shen2019randomized} and QUICSORT \citep{scott2025underdamped}, are designed for this purpose. However, analyzing these advanced schemes in the stochastic gradient setting is non-trivial. In fact, many existing results for similar algorithms only establish a suboptimal half-order convergence rate in $h$ \citep{gouraud2025hmc, guillin2024error}. \reviewed{Our analysis uses the second-order \emph{strong} accuracy of UBU, which allows the local error to be split into a martingale difference and a remainder controlled by the second derivatives of $\phi_h$. We expect analogous results for BAOAB, but an integrator lacking this property would require higher-order expansions of $\phi_h$, and we do not treat general families of integrators here.}

To reduce the stochastic gradient error component of the MSE, variance reduction techniques come into play. By reducing the variance of the stochastic gradient, these techniques have been widely applied in optimization to achieve faster convergence \citep{bubeck2015convex,allen2018katyusha,lan2019unified,song2020variance}. The success of variance reduction in optimization motivates a natural question for sampling: Under what conditions, if any, do these techniques offer a genuine advantage? While methods like SVRG \citep{johnson2013accelerating,reddi2016stochastic,allen2016improved} and SAGA \citep{defazio2014saga,chatterji2018theory} are well-established for accelerating optimization, their role in sampling is fundamentally different. Optimization seeks an efficient path to a single minimum, whereas sampling must generate a trajectory that characterizes an entire probability distribution. This latter objective requires balancing the \textit{discretization error}, governed by the step size $h$, against the \textit{stochastic error} originating from gradient noise. Since variance reduction techniques are designed to mitigate only the stochastic component, their interaction with the discretization error is critical for the total MSE. A key contribution of this paper is to quantify this interplay and establish a practical criterion for the advantageous use of variance-reduced samplers.

\subsection{Our contributions}
This paper introduces a novel discrete Poisson equation for the MSE analysis of stochastic gradient sampling algorithms. We demonstrate the power of this approach by applying it to the Stochastic Gradient UBU (SG-UBU) sampler \citep{sanz2021wasserstein}---a second-order accurate integrator for underdamped Langevin dynamics---and its variance-reduced variants. Our main contributions are as follows:
\begin{itemize}
	\item We propose a general discrete Poisson equation framework, extending the ideas in \cite{mattingly2010convergence}, to analyze the MSE for a broad class of numerical integrators. A key advantage of this framework is its modularity: an MSE bound for any specific integrator can be derived by establishing its uniform-in-time moment bounds (stability) and its local error estimate (consistency).
	
	\item Applying this framework to SG-UBU, we prove that the numerical bias in computing the statistical average $\pi(f)$ has first-order convergence in the step size $h$ and explicitly show that its leading-order term is proportional to the stochastic gradient variance.
	
	\item For the variance-reduced variants, SVRG-UBU and SAGA-UBU, we identify a phase transition: the numerical bias shifts from first to second-order convergence as the step size $h$ decreases below a critical threshold. This accelerated convergence is a unique feature of variance reduction when combined with high-order integrators.
	
	\item We derive a practical criterion to guide the choice between standard and variance-reduced samplers, \reviewed{based on the computational cost each requires to reach a prescribed accuracy}.
\end{itemize}

\subsection{Related works}
The error analysis for stochastic gradient sampling algorithms, which typically discretize either overdamped or underdamped Langevin dynamics, is well studied in the literature. The underdamped approach is often preferred for its computational efficiency, as it exhibits faster mixing times \citep{cao2023explicit,eberle2024non,altschuler2024faster,altschuler2025shifted} and permits the use of higher-order integrators that achieve accelerated convergence with respect to step size $h$. However, this efficiency is coupled with greater analytical challenges due to the hypoellipticity of the underdamped dynamics. 

In the following, we summarize four primary analytical frameworks for stochastic gradient samplers, with results covering both overdamped and underdamped Langevin dynamics. 
These frameworks vary in their methodology and applicability. The most direct approach combines an integrator's recurrence relations with strong or weak error analysis; this is effective for simple schemes like SGLD \citep{nagapetyan2017true,dalalyan2017user,brosse2018promises,zou2019stochastic,cheng2020stochastic,chada2023unbiased,guillin2024error,shaw2025random}, but is not suitable for higher-order integrators. A second approach uses coupling arguments to establish global contractivity in a probability metric, which guarantees a unique invariant distribution but often yields suboptimal convergence rates such as $\mathcal{O}(\sqrt{h})$ \citep{chak2023reflection,leimkuhler2023contraction,schuh2024convergence,gouraud2025hmc}. Alternatively, the entropy approach analyzes the associated Fokker--Planck equation to bound the KL divergence; however, its complexity typically limits its use to simple integrators \citep{li2022sharp,kinoshita2022improved,suzuki2023convergence}. Finally, the continuous Poisson equation is a powerful tool for analyzing the MSE, but its application to underdamped systems is hampered by the significant challenge of obtaining the necessary derivative bounds on its solution, especially without assuming convexity \citep{mattingly2010convergence,ross2011fundamentals,chen2015convergence,vollmer2016exploration,teh2016consistency,xu2018global,brosse2019tamed,chen2023probability}.

Our approach is based on the discrete Poisson equation, tailored for the direct analysis of numerical integrators. Compared to the continuous Poisson equation, this approach aims at analyzing the one step transition kernel of the numerical scheme instead of the infinitesimal generator of the continuous process. The analysis results in a simple and exact decomposition of the time average error into stochastic gradient, discretization, and martingale components. This decomposition provides a direct path for a sharp MSE analysis, and as in the continuous case, requires establishing bounds on the derivatives of the solution to the Poisson equation. We establish such bounds up to the second order under a global convexity assumption on the potential function $U(\bm x)$. The idea of using the discrete Poisson equation draws strong inspiration from early studies \citep{kontoyiannis2003spectral,glynn1996liapounov,johndrow2017error}. Nevertheless, the error analyses in these previous works mainly utilize Harris' ergodic theorem \citep{meyn2012markov}, whereas our work combines the discrete Poisson equation with local strong error analysis. \reviewed{Working with the one-step transition kernel rather than the generator also lowers the regularity required: only the first and second derivatives of $\phi_h$ enter our estimates, whereas analyses based on the continuous Poisson equation typically require derivatives of $\phi_0$ up to fourth order \citep[Lemma~17]{brosse2019tamed}. Global convexity is used at a single point, in the coupling contractivity of the first variation process (Lemma~\ref{lemma: contractivity}), which in turn drives the Hessian bound of Theorem~\ref{theorem: phi h estimate}.}

Denoting by $\pi_h$ the invariant distribution of the numerical solution with the step size $h$, we may quantify the difference between $\pi_h$ and the target distribution $\pi$ using a distribution metric, such as Wasserstein or KL divergence, thereby characterizing the effects of stochastic gradient error and discretization error. For overdamped Langevin dynamics, both $\pi_h$ and $\pi$ are supported in $\mathbb{R}^d$. For underdamped Langevin dynamics, $\pi_h$ is a distribution over both position and velocity variables, and is thus supported in $\mathbb R^{2d}$. In this case, the difference is measured between $\pi_h$ and the Gaussian velocity-augmented distribution of $\pi$. By setting a test function $f$ supported in $\mathbb R^d$, the difference between $\pi_h$ and $\pi$ can also be quantified by the numerical bias
\begin{equation}
\pi_h(f) - \pi(f) = \int_{\mathbb{R}^{2d}}
f(\bm{x}) \Big(\pi_h(\bm x,\bm v) - \big(\pi\otimes\mathcal N(\bm 0,M_2^{-1}\bm I_d)\big)(\bm x,\bm v)\Big) \mathrm{d}\bm{x}\mathrm{d}\bm{v},
\label{numerical bias}
\end{equation}
whose square is the long-time limit $(K \to \infty)$ of the MSE defined in \eqref{MSE}.

To situate our contributions within the existing literature, Table~\ref{table: bias} summarizes key error analysis results for the stochastic gradient sampling algorithms in overdamped and underdamped Langevin dynamics. In the table, the second column indicates whether the global convexity of the potential $U(\bm x)$ is assumed, with $m>0$ as the convexity lower bound (see Assumption~\ref{asP}). The constant $C$ may depend on the spatial dimension $d$ or the convexity bound $m$, indicated by subscripts, while subscript $*$ indicates dependence on additional parameters for non-convex cases. 

\begin{table}[htb]
	\small
	\centering
	\renewcommand{\arraystretch}{2.22}
	\begin{tabular}{cccc}
		\toprule
		Reference & Convex & Integrator & Result \\
		\midrule
		\makecell{Theorem~9 \\ \scriptsize\citep{vollmer2016exploration}} & \textbf{--} & SGLD & $\mathcal W_1(\pi_h,\pi) \Le C_{d,*}h$ \\
		\makecell{Theorem~5 \\ \scriptsize\citep{dalalyan2017user}} & Req. & SGLD & $\mathcal W_2(\pi_h,\pi) \Le C\bigg(\dfrac{dh}{m} + \sqrt{\dfrac{dh}{m}}\bigg)$ \\
		\makecell{Corollary~3.1 \\ \scriptsize\citep{li2022sharp}} & \textbf{--} & SGLD & $\mathsf{KL}(\pi_h\|\pi) \Le C_*dh^2$ \\
		\makecell{Theorem~3.3 \\ \scriptsize\citep{chen2023probability}} & Req. & SGLD & $\mathcal W_1(\pi_h,\pi) \Le C_m d(1+\log |h|)h$ \\
		\makecell{Theorem~3.9 \\ \scriptsize\citep{shaw2025random} }& Req. & \hspace{-6pt} SGLD-RR & $\mathcal W_2(\pi_h,\pi) \Le \dfrac{Cdh}{m}$ \\
		\midrule
		\makecell{Theorem~3.8 \\ \scriptsize\citep{zou2019stochastic}} & \textbf{--} & SG-EM & $\mathcal W_2(\pi_h,\pi) \Le C_{d,*} h^{\frac34}$ \\
		\makecell{Theorem~1 \\ \scriptsize\citep{guillin2024error}} & Req. & SG-EM& $\mathcal W_1(\pi_h,\pi) \Le C_m\sqrt{dh}$ \\
		\makecell{Proposition~18 \\ \scriptsize\citep{gouraud2025hmc}} & Req. & SGgHMC & $\mathcal W_2(\pi_h,\pi) \Le C_m\sqrt{dh}$ \\
		\makecell{Theorem~\ref{theorem: SG-UBU numerical bias} \\ \scriptsize \textbf{(this paper)}} & Req. & SG-UBU & $|\pi_h(f)-\pi(f)| \Le \dfrac{Cdh}{m^2}$ \\
		\bottomrule
	\end{tabular}
	\caption{Numerical bias results for stochastic gradient sampling algorithms. The upper block collects results based on the overdamped Langevin dynamics, and the lower block those based on the underdamped Langevin dynamics. \revised{In the ``Convex'' column, Req.\ marks a result established under the global convexity of $U(\bm x)$, and \textbf{--} marks a result established without it.} The constant $C$ may depend on $d$ and $m$, as indicated by subscripts; subscript $*$ denotes dependence on additional parameters in non-convex cases.}
	\label{table: bias}
\end{table}

The underdamped results in Table~\ref{table: bias} highlight a key contribution of our work. While existing analyses for underdamped samplers often yield suboptimal convergence rates such as $\mathcal{O}(\sqrt{h})$, our analysis demonstrates that the SG-UBU integrator achieves first-order convergence of the numerical bias $\pi_h(f) - \pi(f)$. We further show that the leading-order error coefficient is determined by the stochastic gradient variance and is, remarkably, unaffected by the discretization error. 

In the case of finite-sum potentials \eqref{finite-sum potential}, error analysis for variance-reduced sampling algorithms presents additional challenges. The inherent complexity of methods such as SVRG and SAGA means that the numerical solution is typically no longer a Markov chain. This structural limitation obstructs standard analytical tools such as global contractivity proofs, and as a result, the literature mostly uses recurrence relations for error analysis \citep{nagapetyan2017true, zou2018subsampled, chatterji2018theory, baker2019control, zou2018stochastic, zou2019stochastic, chen2022approximation}. While alternative methods using the continuous Poisson equation \citep{dubey2016variance, li2019stochastic} or entropy approaches \citep{kinoshita2022improved} show theoretical promise for non-convex potentials, their extensions to the underdamped case are less explored.

Table~\ref{table: vr} summarizes existing error analysis results for these algorithms. The parameter $N$ denotes the number of components in the finite-sum potential defined in \eqref{finite-sum potential}, and $p$ is the batch size for the mini-batch gradient approximation. Due to the non-Markovian nature of these methods, the numerical bias is measured with respect to $\tilde{\pi}_h$, which represents the long-time limit of the averaged distributions over the trajectory. For a fair comparison across different works, the results cited are normalized such that the assumed upper bound on the Hessian matrix $\|\nabla^2 U(\bm x)\|$ is independent of $N$. For algorithms utilizing an anchor position, the update frequency is set to $N/p$ iterations in the mini-batch setting.

\begin{table}[htb]
	\small
	\renewcommand{\arraystretch}{2.22}
    \centerline{
	\begin{tabular}{cccc}
		\toprule
		Reference & \hspace{-0.4cm} Convex \hspace{-0.4cm} & Integrator & Result \\
		\midrule
		\makecell{Theorem~3.5 \\ \scriptsize \citep{zou2018subsampled}} &  Req.  & SVRG-LD & $\mathcal W_2(\tilde \pi_h,\pi) \Le C\bigg(
        \dfrac{dh}{m} + \sqrt{\dfrac{dh}{mp}}\min\bigg\{1,\sqrt{\dfrac{Nh}{mp}}\bigg\}\bigg)$\\
		\makecell{Theorem~4.2 \\ \scriptsize\citep{chatterji2018theory}} & Req. & SVRG-LD & $\mathcal W_2(\tilde \pi_h,\pi) \Le C\bigg(\dfrac{dh}{m}+\sqrt{\dfrac{dN}{mp^3}}h\bigg)$ \\
		\makecell{Theorem~4.1 \\ \scriptsize\citep{chatterji2018theory}} & Req. &  SAGA-LD & $\mathcal W_2(\tilde \pi_h,\pi) \Le C\bigg(\dfrac{dh}{m}+\sqrt{\dfrac{dN}{mp}}h\bigg)$ \\
		\makecell{Theorem~2.5 \\ \scriptsize\citep{chen2022approximation}} & \textbf{--} & SVRG-LD & $\mathcal W_1(\tilde \pi_h,\pi) \Le C_{d,*} h^{\frac14}$ \\
        \makecell{Theorem~1 \\ \scriptsize\hspace{-0.3cm}\citep{kinoshita2022improved}} \hspace{-0.3cm}  & \textbf{--} & SVRG-LD & $\mathsf{KL}(\tilde \pi_h\|\pi) \Le C_{d,*} \bigg(1+\dfrac{N}{p^2}\bigg)h$ \\
        \makecell{Theorem~2 \\ \scriptsize\hspace{-0.3cm}\citep{kinoshita2022improved}} \hspace{-0.3cm} & \textbf{--} & \hspace{-0.4cm}SARAH-LD \hspace{-0.4cm} & $\mathsf{KL}(\tilde \pi_h\|\pi) \Le C_{d,*} \bigg(1+\dfrac{N}{p^2}\bigg)h$ \\
		\midrule
		\makecell{Theorem~3.3 \\ \scriptsize \citep{zou2019stochastic}} &\textbf{--} & \hspace{-0.4cm} SRVR-HMC \hspace{-0.4cm} &
		$\mathcal W_2(\tilde \pi_h,\pi) \Le C_{d,*}\bigg(1+\dfrac{N}{p^2}\bigg)^{\frac14}h^{\frac34}$ \\
		\makecell{Theorem~\ref{theorem: SVRG-UBU numerical bias} \\ \scriptsize
		\textbf{(this paper)}} & Req. & \hspace{-0.4cm} SVRG-UBU \hspace{-0.4cm} & \hspace{-0.2cm} $|\tilde \pi_h(f) - \pi(f)| \Le C\bigg(\dfrac{dh}{m^2p} \min
		\bigg\{1,\dfrac{N^2h^2}{mp^2}\bigg\} + 
        \dfrac{dh^2}{m^3}\bigg)$ \hspace{-0.2cm} \\
		\makecell{Theorem~\ref{theorem: SAGA-UBU numerical bias} \\ \scriptsize \textbf{(this paper)}} &
		Req. & \hspace{-0.4cm} SAGA-UBU \hspace{-0.4cm} & \hspace{-0.2cm} $|\tilde \pi_h(f) - \pi(f)| \Le C\bigg(\dfrac{dh}{m^3}\min\big\{1,N^2h^2\big\}+ \dfrac{dh^2}{m^3}\bigg)$ \hspace{-0.2cm} \\
		\bottomrule
	\end{tabular}
    }
	\caption{Numerical bias results for variance-reduced sampling algorithms. The upper block collects results based on the overdamped Langevin dynamics, and the lower block those based on the underdamped Langevin dynamics. \revised{In the ``Convex'' column, Req.\ marks a result established under the global convexity of $U(\bm x)$, and \textbf{--} marks a result established without it.} Results that do not explicitly depend on the batch size $p$ are for the $p=1$ case. The constant $C$ may depend on $d$, as indicated by the subscript; subscript $*$ denotes dependence on additional parameters in non-convex cases.}
	\label{table: vr}
\end{table}

As shown in Table~\ref{table: vr}, our analysis of SVRG-UBU and SAGA-UBU uncovers a phase transition phenomenon for the rate of convergence. The convergence rate with respect to the step size $h$ transitions from first-order to an accelerated second-order rate as $h$ decreases below a critical threshold. This relationship appears to be a unique feature of combining variance reduction with second-order discretization schemes, a property fulfilled by the UBU-type integrators studied in this work.

The remainder of this paper is organized as follows. Section~\ref{section: pre-meth} introduces our notation and the discrete Poisson equation framework. In Section~\ref{section: MSE bias SG-UBU}, we present our main results for the SG-UBU integrator, followed by the analysis of its variance-reduced counterparts, SVRG-UBU and SAGA-UBU, in Section~\ref{section: MSE bias variance-reduced SG-UBU}. We provide numerical validation in Section~\ref{section: numerical} and conclude in Section~\ref{section: conclusion}.

\section{Preliminaries and methodology}
\label{section: pre-meth}

\subsection{Notations and assumptions}

Throughout the paper, we use $|\cdot|$ to denote the Euclidean norm for vectors. For a matrix $\bm{A} \in \mathbb{R}^{n_1 \times n_2}$, we use $\|\bm{A}\|$ for the operator 2-norm (or spectral norm), defined as $$
\|\bm{A}\| := \sup_{|\bm{x}|=1, \bm{x} \in \mathbb{R}^{n_2}} |\bm{A}\bm{x}|.
$$
This definition extends to a third-order tensor $\bm{B} \in \mathbb{R}^{n_1 \times n_2 \times n_3}$, whose norm is defined via the corresponding multilinear form:
\begin{equation*}
	\|\bm{B}\| := \sup_{|\bm{x}|=1, |\bm{y}|=1, |\bm{z}|=1} \sum_{i=1}^{n_1} \sum_{j=1}^{n_2} \sum_{k=1}^{n_3} B_{ijk} x_i y_j z_k.
\end{equation*}

We adopt standard typesetting conventions where vectors and matrices are denoted by bold fonts (e.g., $\bm{x}, \bm{A}$), while scalars are in normal font (e.g., $h, \sigma$). 
The zero matrix and the identity matrix in $\mathbb{R}^{d\times d}$ are indicated by $\bm{O}_d$ and $\bm{I}_d$, respectively. For symmetric matrices in $\mathbb R^{d\times d}$, we write $\bm A\preccurlyeq \bm B$ if $\bm B-\bm A$ is positive semidefinite. Furthermore, we distinguish continuous-time dynamics from their numerical approximations by using lowercase and uppercase letters, respectively. Specifically, the continuous solution to the underdamped Langevin dynamics is denoted by $(\bm{x}_t, \bm{v}_t)_{t \Ge 0}$. Its numerical approximation, generated by an integrator such as SG-UBU with step size $h$, is denoted by the discrete sequence $(\bm{X}_k, \bm{V}_k)_{k=0}^\infty$, where $(\bm{X}_k, \bm{V}_k)$ approximates the continuous solution at time $t=kh$.

Our analysis relies on the following assumptions regarding the potential function $U(\bm{x})$, the stochastic gradient $b(\bm{x},\theta)$, and the test function $f(\bm{x})$.

\begin{asP}
	\mylabel{asP}{(P)}
	The potential function $U(\bm{x}) \in C^3(\mathbb{R}^d)$ satisfies $\nabla U(\bm{0}) = \bm{0}$ and
	\begin{equation*}		
		m\bm{I}_d \preccurlyeq \nabla^2 U(\bm{x}) \preccurlyeq M_2\bm{I}_d ,\quad \|\nabla^3 U(\bm{x})\| \Le M_3 ,\quad \forall\bm{x}\in\mathbb{R}^d		
	\end{equation*}
	for constants $m, M_2, M_3 >0$.
\end{asP}
\begin{asSg}
	\mylabel{asSg}{(Sg)}
	The stochastic gradient $b(\bm x,\theta)$ satisfies
	$$
	\mathbb E_\theta\big[b(\bm x,\theta)\big] = \nabla U(\bm x),\quad
	\Big(\mathbb E_\theta\big[|b(\bm x,\theta) - \nabla U(\bm x)|^8\big]\Big)^{\frac14} \Le \sigma^2d, \quad
	\forall \bm x\in\mathbb R^d
	$$
	for the constant $\sigma>0$, and for almost all $\theta \sim \mathbb P_\theta$ there holds
	\begin{equation*}
		|b(\bm x,\theta)| \Le M_1 \sqrt{|\bm x|^2+d},
		\quad \forall \bm x\in\mathbb R^d	
	\end{equation*}
	for the constant $M_1>0$. Assume $\sigma\Le M_1$ for the convenience of analysis.
\end{asSg}
In the case of finite-sum potentials, Assumption~\ref{asSg} is specified as follows.
\begin{asSgN}
	\mylabel{asSgN}{(SgN)}
	For the finite-sum potential $U(\bm x)$, the gradients $\{\nabla U_i(\bm x)\}_{i=1}^N$ satisfy
	\begin{equation*}
		\bigg(\frac1N \sum_{i=1}^N |\nabla U_i(\bm x) - \nabla U(\bm x)|^8\bigg)^{\frac14} \Le \sigma^2d,\quad
		\forall \bm x\in\mathbb R^d
	\end{equation*}
	for the constant $\sigma>0$, and for all $i\in\{1,\cdots,N\}$ there holds
	\begin{equation*}
		|\nabla U_i(\bm x)| \Le M_1\sqrt{|\bm x|^2+d},\quad
		\forall \bm x\in\mathbb R^d
	\end{equation*}	
	for the constant $M_1>0$. Assume $\sigma\Le M_1$ for the convenience of analysis.
\end{asSgN}
\begin{asT}
	\mylabel{asT}{(T)}
	The test function $f(\bm x) \in C^2(\mathbb R^d)$ satisfies
	\begin{equation*}
		|\nabla f(\bm x)| \Le L_1,\quad	
		\norm{\nabla^2 f(\bm x)} \Le L_2,\quad \forall \bm x\in\mathbb R^d
	\end{equation*}
	for constants $L_1,L_2>0$.
\end{asT}
The subscripts on the constants $M_i$ and $L_i$ correspond to the order of the associated derivatives. These assumptions provide several key properties. First, Assumption~\ref{asP} guarantees $m\Le M_2$, and $U(\bm{x})$ is strongly convex and has a unique minimizer at $\bm{0}$. Second, the growth condition and unbiasedness in Assumption~\ref{asSg} together imply a linear growth bound on the true gradient, $|\nabla U(\bm{x})| \Le M_1\sqrt{|\bm{x}|^2 + d}$. Finally, the eighth moment condition in Assumption~\ref{asSg} ensures the necessary regularity for a rigorous MSE analysis.

\subsection{Underdamped Langevin dynamics and SG-UBU integrator}
\label{subsection: uld SG-UBU}

To sample from the target distribution $\pi(\bm{x}) \propto e^{-U(\bm{x})}$, we consider the discretization of the underdamped Langevin dynamics. Given the Hessian bound $\|\nabla^2 U(\bm{x})\| \Le M_2$, consider the Hamiltonian system
\begin{equation*}
	H(\bm{x},\bm{v}) = U(\bm{x}) + \frac{M_2}{2}|\bm{v}|^2, \quad \bm{x},\bm{v} \in \mathbb{R}^d,
\end{equation*}
where the auxiliary velocity variable $\bm{v}$ extends the state space to $\mathbb{R}^{2d}$.

The underdamped Langevin dynamics for $H(\bm x,\bm v)$ evolves according to the following stochastic differential equation:
\begin{equation}
	\label{ULD}
	\mathfrak{L}:\left\{
	\begin{aligned}
		\dot{\bm{x}}_t &= \bm{v}_t, \\
		\dot{\bm{v}}_t &= -\frac{1}{M_2}\nabla U(\bm{x}_t) - 2\bm{v}_t + \frac{2}{\sqrt{M_2}}\dot{\bm{B}}_t,
	\end{aligned}
	\right.
\end{equation}
where $\bm{B}_t$ is the standard Brownian motion in $\mathbb{R}^d$. We choose the damping coefficient as 2 for optimal dissipation characteristics, and the gradient scaling as $1/M_2$ to enable stable numerical integration with $\mathcal{O}(1)$ step sizes.

Next we describe the UBU-type numerical integrators for the underdamped Langevin dynamics \eqref{ULD}, including the full gradient and stochastic gradient versions. Applying the Lie--Trotter splitting method, we decompose \eqref{ULD} into two separate dynamics
\begin{equation*}
	\mathfrak{U}: 
	\left\{
	\begin{aligned}
		\dot{\bm{x}}_t &= \bm{v}_t, \\
		\dot{\bm{v}}_t &= -2\bm{v}_t + \frac{2}{\sqrt{M_2}}\dot{\bm{B}}_t,
	\end{aligned}
	\right. \quad
	\mathfrak{B}:
	\left\{
	\begin{aligned}
		\dot{\bm{x}}_t &= \bm{0}, \\
		\dot{\bm{v}}_t &= -\frac{1}{M_2}\nabla U(\bm{x}_t),
	\end{aligned}
	\right.
\end{equation*}
where $\mathfrak U$ denotes the linear part and $\mathfrak B$ represents a free particle driven by the potential gradient.
Their solution flows, denoted by $\Phi_t^{\mathfrak{U}}$ and $\Phi_t^{\mathfrak{B}}$, respectively, are given by
\begin{equation*}
	\Phi_t^{\mathfrak{U}}:
	\left\{
	\begin{aligned}
		\bm{x}_t &= \bm{x}_0 + \frac{1-e^{-2t}}{2}\bm{v}_0 + \frac{2}{\sqrt{M_2}}\int_0^t \frac{1-e^{-2(t-s)}}{2} \D\bm{B}_s,\\
		\bm{v}_t &= e^{-2t}\bm{v}_0 + \frac{2}{\sqrt{M_2}}\int_0^t e^{-2(t-s)}\D\bm{B}_s,
	\end{aligned}
	\right.
	~~
	\Phi_t^{\mathfrak{B}}:
	\left\{
	\begin{aligned}
		\bm{x}_t &= \bm{x}_0, \\
		\bm{v}_t &= \bm{v}_0 - \frac{t}{M_2}\nabla U(\bm{x}_0).
	\end{aligned}
	\right.
\end{equation*}

The Full Gradient UBU (FG-UBU) integrator implements the following composition:
\begin{equation}
	\label{FG-UBU}
	(\bm{X}_{k+1}, \bm{V}_{k+1}) = \big(\Phi_{h/2}^{\mathfrak{U}} \circ \Phi_h^{\mathfrak{B}} \circ \Phi_{h/2}^{\mathfrak{U}}\big)(\bm{X}_k, \bm{V}_k), \qquad k = 0,1,\cdots.
\end{equation}
Correspondingly, the Stochastic Gradient UBU (SG-UBU) integrator is constructed by replacing $\nabla U(\bm x)$ in FG-UBU with $b(\bm x,\theta)$. Define the perturbed solution flow corresponding to the stochastic gradient $b(\bm x,\theta)$ by
\begin{equation*}
	\Phi_t^{\mathfrak{B}}(\theta):\left\{
	\begin{aligned}
		\bm{x}_t &= \bm{x}_0,\\
		\bm{v}_t &= \bm{v}_0 - \frac{t}{M_2}b(\bm{x}_0,\theta).
	\end{aligned}
	\right.
\end{equation*}
Then the SG-UBU integrator with the step size $h$ becomes
\begin{equation}
	\label{SG-UBU}
	(\bm{X}_{k+1},\bm{V}_{k+1}) = \big(\Phi_{h/2}^{\mathfrak{U}} \circ \Phi_h^{\mathfrak{B}}(\theta_k) \circ \Phi_{h/2}^{\mathfrak{U}}\big)(\bm{X}_k,\bm{V}_k),\qquad k=0,1,\cdots,
\end{equation}
where $(\theta_k)_{k=0}^\infty$ are independent samples from $\mathbb P_\theta$. Equivalently, we can produce the numerical solution of SG-UBU $(\bm{X}_k,\bm{V}_k)_{k=0}^\infty$ through Algorithm~\ref{algorithm: SG-UBU}.
\begin{algorithm}[htb]
	\setstretch{1.12}
	\caption{Stochastic Gradient UBU (SG-UBU)}
	\textbf{Input:} initial state $(\bm{X}_0,\bm{V}_0)$, step size $h$ \\
	\textbf{Output:} numerical solution $(\bm{X}_k,\bm{V}_k)_{k=0}^\infty$ \\
	\For{$k=0,1,\cdots$}{
		Evolve $(\bm{Y}_k,\bm{V}_k^{(1)}) = \Phi_{h/2}^{\mathfrak{U}}(\bm{X}_k,\bm{V}_k)$ \\
		Sample a random index $\theta_k \sim \mathbb{P}_\theta$ \\
		Compute the stochastic gradient $\bm{b}_k = b(\bm{Y}_k,\theta_k)$ \\
		Update the velocity $\bm{V}_k^{(2)} = \bm{V}_k^{(1)} - \frac{h}{M_2}\bm{b}_k$ \\
		Evolve $(\bm{X}_{k+1},\bm{V}_{k+1}) = \Phi_{h/2}^{\mathfrak{U}}(\bm{Y}_k,\bm{V}_k^{(2)})$
	}
	\label{algorithm: SG-UBU}
\end{algorithm}
Note that in SG-UBU, the stochastic gradient is evaluated at the intermediate position $\bm{Y}_k$, which is the position component of $\Phi_{h/2}^{\mathfrak{U}}(\bm{X}_k,\bm{V}_k)$, rather than at $\bm{X}_k$ itself. Specifically, $\bm Y_k$ is explicitly expressed as
\begin{equation}
	\bm{Y}_k = \bm{X}_k + \frac{1-e^{-h}}{2}\bm{V}_k + \frac{2}{\sqrt{M_2}} \int_0^{h/2} \frac{1-e^{-(h-2s)}}{2} \D\bm{B}_{kh+s}.
	\label{Y_k expression}
\end{equation}

For the finite-sum potential defined in \eqref{finite-sum potential}, the SG-UBU integrator can be efficiently implemented using a mini-batch approach. The procedure is detailed in Algorithm~\ref{algorithm: mini-batch SG-UBU}.
\begin{algorithm}[htb]
	\setstretch{1.12}
	\caption{Mini-batch Stochastic Gradient UBU (mini-batch SG-UBU)}
	\textbf{Input:} initial state $(\bm{X}_0,\bm{V}_0)$, batch size $p$, step size $h$ \\
	\textbf{Output:} numerical solution $(\bm{X}_k,\bm{V}_k)_{k=0}^\infty$ \\
	\For{$k=0,1,\cdots$}{
		Evolve $(\bm{Y}_k,\bm{V}_k^{(1)}) = \Phi_{h/2}^{\mathfrak{U}}(\bm{X}_k,\bm{V}_k)$. \\
		Sample a random subset $\theta_k\subset \{1,\cdots,N\}$ of size $p$ \\
		Compute the mini-batch gradient $\bm{b}_k = \frac{1}{p} \sum_{i \in \theta_k} \nabla U_i(\bm Y_k)$ \\
		Update the velocity $\bm{V}_k^{(2)} = \bm{V}_k^{(1)} - \frac{h}{M_2}\bm{b}_k$ \\
		Evolve $(\bm{X}_{k+1},\bm{V}_{k+1}) = \Phi_{h/2}^{\mathfrak{U}}(\bm{Y}_k,\bm{V}_k^{(2)})$
	}
	\label{algorithm: mini-batch SG-UBU}
\end{algorithm}
The computational cost for each iteration of Algorithm~\ref{algorithm: mini-batch SG-UBU} is $\mathcal{O}(p)$, corresponding to the evaluation of $p$ individual gradients that form the mini-batch.

\subsection{The discrete Poisson equation framework}
The discrete Poisson equation provides an analytical framework for studying the time average error of numerical solutions to Langevin dynamics. For the underdamped Langevin dynamics \eqref{ULD}, the infinitesimal generator is given by
\begin{equation*}
	\mathcal{L} = \bm{v}^\top \nabla_{\bm{x}} - \bigg(\frac{1}{M_2}\nabla U(\bm{x}) + 2\bm{v}\bigg)^\top \nabla_{\bm{v}} + \frac{2}{M_2}\Delta_{\bm{v}}.
\end{equation*}
The discrete Poisson equation with respect to this $\mathcal L$ and fixed step size $h>0$ is defined as
\begin{equation}
	\frac{1 - e^{h\mathcal{L}}}{h} \phi_h(\bm{x},\bm{v}) = f(\bm{x}) - \pi(f),
	\label{discrete Poisson eq}
\end{equation}
where $\phi_h(\bm x,\bm v)$ is known as the discrete Poisson solution. 

To facilitate error analysis, we introduce a more tractable representation of the discrete Poisson solution $\phi_h(\bm x,\bm v)$. First, we define the function
\begin{equation}
	u(\bm{x},\bm{v},t) = \mathbb{E}^{(\bm{x},\bm{v})}[f(\bm{x}_t)] - \pi(f),
	\label{function u}
\end{equation}
where the superscript $(\bm{x},\bm{v})$ denotes the initial state for the exact solution $(\bm{x}_t,\bm{v}_t)_{t\Ge0}$. Using the semigroup operator, this is equivalent to
\begin{equation*}
	u(\bm{x},\bm{v},t) = (e^{t\mathcal{L}}f)(\bm{x},\bm{v}) - \pi(f),
\end{equation*}
which leads to a series representation for the discrete Poisson solution:
\begin{equation}
	\phi_h(\bm{x},\bm{v}) = h \sum_{k=0}^\infty u(\bm{x},\bm{v},kh).
	\label{discrete Poisson solution}
\end{equation}
In fact, one can conveniently verify that $\phi_h(\bm x,\bm v)$, as defined in \eqref{discrete Poisson solution}, satisfies
\begin{equation*}
    \phi_h(\bm x,\bm v) = h \sum_{k=0}^\infty \big((e^{kh\mathcal L} f)(\bm x,\bm v) - \pi(f)\big) = h(1-e^{h\mathcal L})^{-1} \big(f(\bm x) - \pi(f)\big),
\end{equation*}
which shows that $\phi_h(\bm x,\bm v)$ solves the discrete Poisson equation \eqref{discrete Poisson eq}.

As $h\to0$, the discrete Poisson solution $\phi_h(\bm x,\bm v)$ in \eqref{discrete Poisson solution} converges to the integral defining the continuous Poisson solution: \begin{equation}
\phi_0(\bm x,\bm v) = \int_0^\infty u(\bm x,\bm v,t)\D t,
\label{continuous Poisson solution}
\end{equation}
which satisfies the conventional Poisson equation $-(\mathcal{L}\phi_0)(\bm x,\bm v) = f(\bm x) - \pi(f)$.

The following theorem establishes global bounds on the gradient and Hessian matrix of the discrete Poisson solution $\phi_h(\bm x,\bm v)$.
\begin{theorem}
	\label{theorem: phi h estimate}
	Under Assumptions~\ref{asP} and \ref{asT}, let the step size $h\Le \frac14$; then the discrete Poisson solution $\phi_h(\bm x,\bm v)$ has a uniformly bounded gradient and Hessian matrix:
	\begin{align*}
		\max\big\{|\nabla_{\bm x} \phi_h|,
		|\nabla_{\bm v} \phi_h|\big\} & \Le \frac{C}{m}, \\
		\max\big\{
		\norm{\nabla_{\bm x\bm x}^2 \phi_h},
		\norm{\nabla_{\bm x\bm v}^2 \phi_h},
		\norm{\nabla_{\bm v\bm v}^2 \phi_h}
		\big\} & \Le \frac{C}{m^2},
	\end{align*}
	where the constant $C$ depends only on $M_2, M_3$ and $(L_i)_{i=1}^2$.
\end{theorem}
The proof of Theorem~\ref{theorem: phi h estimate} is given in Appendix~\ref{appendix: discrete Poisson}.

Next, we detail how the discrete Poisson equation framework is used to analyze the time average error of the numerical solutions. Given the step size $h>0$, let $\bm{Z}_k = (\bm{X}_k, \bm{V}_k)$ be the numerical solution of SG-UBU at the $k$-th step. For each $k = 0,1,2,\cdots$, we define the corresponding full-gradient update $\bar{\bm{Z}}_{k+1} = (\bar{\bm{X}}_{k+1}, \bar{\bm{V}}_{k+1})$ as
\begin{equation}
	(\bar{\bm{X}}_{k+1}, \bar{\bm{V}}_{k+1}) = 
	\big(\Phi_{h/2}^{\mathfrak{U}} \circ \Phi_h^{\mathfrak{B}} \circ \Phi_{h/2}^{\mathfrak{U}}\big)(\bm{X}_k, \bm{V}_k),
\end{equation}
and the exact one-step solution of the Langevin dynamics $\bm{Z}_k(h) = (\bm{X}_k(h), \bm{V}_k(h))$ as
\begin{equation}
	(\bm{X}_k(h), \bm{V}_k(h)) = \Phi_{h}^{\mathfrak{L}}(\bm{X}_k, \bm{V}_k),
\end{equation}
where $\Phi_h^{\mathfrak{L}}$ is the solution flow of \eqref{ULD}. This allows for a natural decomposition of the one-step evolution of the numerical solution:
\begin{equation}
	\bm{Z}_{k+1} - \bm{Z}_k = \big(\bm{Z}_{k+1} - \bar{\bm{Z}}_{k+1}\big) + \big(\bar{\bm{Z}}_{k+1} - \bm{Z}_k(h)\big) + \big(\bm{Z}_k(h) - \bm{Z}_k\big).
	\label{one-step error}
\end{equation}
The three terms on the right-hand side isolate the fundamental sources of error: the first is the unbiased \emph{stochastic gradient error}, the second is the \emph{discretization error} of the FG-UBU integrator, and the third corresponds to the evolution of the exact dynamics, which can be analyzed using its martingale properties.

Applying the discrete Poisson solution $\phi_h(\bm{x},\bm{v})$ to the one-step evolution \eqref{one-step error} yields the key error decomposition identity:
\begin{equation}
	\frac{\phi_h(\bm{Z}_{k+1})-\phi_h(\bm{Z}_k)}{h} + f(\bm{X}_k) - \pi(f) = R_k + S_k + T_k,
	\label{RST decomposition}
\end{equation}
where the error terms $R_k, S_k$ and $T_k$ are defined by
\begin{gather*}
	R_k  =\frac{\phi_h(\bm{Z}_{k+1})-\phi_h(\bar{\bm{Z}}_{k+1})}{h},\qquad S_k=\frac{\phi_h(\bar{\bm{Z}}_{k+1})-\phi_h(\bm{Z}_k(h))}{h},\\
	T_k =\frac{\phi_h(\bm{Z}_k(h))-\phi_h(\bm{Z}_k)}{h}+f(\bm{X}_k)-\pi(f).
\end{gather*}
Here, $R_k$ and $S_k$ correspond to the stochastic gradient error and the discretization error, respectively. The third term, $T_k$, forms a martingale difference sequence \revised{with respect to the filtration
\begin{equation*}
	\mathcal F_k = \sigma\big(\bm Z_0,\theta_0,\cdots,\theta_{k-1},(\bm B_t)_{0\Le t\Le kh}\big),
\end{equation*}
representing the history of the algorithm on the time interval $[0,kh]$}. This property follows directly from the discrete Poisson equation \eqref{discrete Poisson eq}, which ensures that the conditional expectation of $T_k$ is zero:
\begin{equation*}
	\mathbb{E}\big[T_k\big|\revised{\mathcal F_k}\big]=\frac{(e^{h\mathcal{L}}\phi_h)(\bm{Z}_k)-\phi_h(\bm{Z}_k)}{h}+f(\bm{X}_k)-\pi(f)=0.
\end{equation*}
\revised{The same filtration serves the other two error terms: $\bm Z_k$ is $\mathcal F_k$-measurable, while the random index $\theta_k$ and the Brownian increment on $[kh,(k+1)h]$ are both independent of $\mathcal F_k$.}

Finally, summing the identity \eqref{RST decomposition} from $k=0$ to $K-1$ establishes a fundamental expression for the time average error:
\begin{equation}
	\frac{1}{K}\sum_{k=0}^{K-1}f(\bm{X}_k)-\pi(f)=\frac{\phi_h(\bm{Z}_0)-\phi_h(\bm{Z}_K)}{Kh} + \frac{1}{K}\sum_{k=0}^{K-1}(R_k + S_k + T_k).
	\label{explicit time average}
\end{equation}
This representation allows the MSE to be bounded by analyzing the second moments of the cumulative error terms. The analysis is simplified by the orthogonality of the martingale difference sequence $(T_k)_{k=0}^{K-1}$, which implies that its variance is additive:
\begin{equation*}
	\mathbb{E}\Bigg[\bigg(\sum_{k=0}^{K-1}T_k\bigg)^2\Bigg] = \sum_{k=0}^{K-1}\mathbb{E}\big[T_k^2\big].
\end{equation*}
Therefore, the MSE estimation reduces to bounding the accumulated stochastic gradient errors ($R_k$) and discretization errors ($S_k$), while the martingale terms ($T_k$) contribute only through the sum of their individual variances.

\section{Mean square error and numerical bias of SG-UBU}
\label{section: MSE bias SG-UBU}

This section establishes a quantitative estimate for the MSE of SG-UBU by analyzing the second moment of the time average error in \eqref{explicit time average}. Our proof combines two essential components, corresponding to stability and consistency in the standard numerical analysis methodology. We first establish stability by deriving uniform-in-time moment bounds for the SG-UBU solution $(\bm X_k,\bm V_k)_{k=0}^\infty$. We then analyze consistency by estimating the moments of the local error, which is decomposed into the stochastic gradient error $\bm{Z}_{k+1} - \bar{\bm{Z}}_{k+1}$ and the discretization error $\bar{\bm{Z}}_{k+1} - \bm Z_k(h)$. This decomposition reveals that the stochastic gradient error contributes the dominant first-order term in the step size $h$, while the discretization error contributes a higher second-order term. This leads to a first-order numerical bias, whose leading-order coefficient is proportional to the variance of the stochastic gradient.

\subsection{Stability: Uniform-in-time moment bound}
We establish a uniform-in-time eighth moment bound for the SG-UBU solution $(\bm X_k,\bm V_k)_{k=0}^\infty$. This high-order moment is required to control the corresponding moments of the local error terms in our subsequent analysis. The proof hinges on a key simplification: we first analyze the related SG-BU integrator, whose update rule is given by
\begin{equation}
	\label{SG-BU}
	(\bm{X}_{k+1},\bm{V}_{k+1}) = \big(\Phi_h^{\mathfrak{U}} \circ \Phi_h^{\mathfrak{B}}(\theta_k)\big)(\bm{X}_k,\bm{V}_k),\qquad k=0,1,\cdots.
\end{equation}
Since the SG-UBU integrator \eqref{SG-UBU} differs from SG-BU only by a stable $\Phi_{h/2}^{\mathfrak{U}}$ mapping, their uniform-in-time moment bounds are equivalent \citep{schuh2024convergence}. This approach is advantageous because the SG-BU scheme evaluates the stochastic gradient at the state $\bm X_k$ itself, making the analysis more tractable than a direct treatment of SG-UBU.

To establish the uniform-in-time moment estimate for SG-BU, we employ a Lyapunov function approach. Thanks to the global convexity of the potential function $U(\bm x)$, we define the Lyapunov function based on the quadratic form in \cite{mattingly2002ergodicity}:
\begin{equation*}
	\mathcal V(\bm x,\bm v) = \big(2|\bm x|^2 + 2\bm x^\top\bm v + |\bm v|^2\big)^4,\qquad
	\bm z = (\bm x,\bm v) \in \mathbb R^{2d},
\end{equation*}
where the fourth power is chosen to control the eighth moment of the numerical solution.

For SG-BU, this choice of $\mathcal{V}$ yields the following one-step contractive property.
\begin{lemma}
	\label{lemma: SG-BU Lyapunov}
	Under Assumptions~\ref{asP} and \ref{asSg}, let the step size $h\Le \frac14$; for any $\bm z_0\in\mathbb R^{2d}$,
	\begin{equation*}
		\mathbb E\Big[\mathcal V\big((\Phi_h^{\mathfrak U} \circ \Phi_h^{\mathfrak B}(\theta)) \bm z_0\big)\Big] \Le \bigg(1-\frac{mh}{3M_2}\bigg) \mathcal V(\bm z_0) + \frac{Cd^4h}{m^3},
	\end{equation*}
	where $\theta\sim\mathbb P_\theta$ and the constant $C$ depends only on $M_1,M_2$.
\end{lemma}
This Lyapunov condition for the SG-BU integrator provides the foundation for bounding the moments of the desired SG-UBU solution.
\begin{theorem}
	\label{theorem: SG-UBU moment bound}
	Under Assumptions~\ref{asP} and \ref{asSg}, let the step size $h\Le \frac14$, and the initial state
	$$
	\bm Z_0 = (\bm X_0,\bm V_0)\text{~with~}\bm X_0 = \bm 0\text{~and~}\bm V_0 \sim \mathcal N(\bm 0,M_2^{-1}\bm I_d);
	$$
	then the SG-UBU solution $(\bm X_k,\bm V_k)_{k=0}^\infty$ satisfies
	\begin{equation*}
		\sup_{k\Ge0} \mathbb E\big[|\bm Z_k|^8\big]
		\Le \frac{Cd^4}{m^4},
	\end{equation*}
	where the constant $C$ depends only on $M_1,M_2$.
\end{theorem}
The order of parameters $d$ and $m$ in Theorem~\ref{theorem: SG-UBU moment bound} is optimal, and the equality holds if the target distribution $\pi(\bm x)$ is exactly the Gaussian distribution $\mathcal N(\bm 0,m^{-1} \bm I_d)$.

As a corollary, this uniform-in-time moment bound also applies to the numerical invariant distribution of the SG-UBU integrator, whose existence is guaranteed by the Krylov--Bogolyubov theorem \citep{hairer2006ergodic}.
\begin{lemma}
\label{lemma: SG-UBU invariant}
Under Assumptions~\ref{asP} and \ref{asSg}, let the step size $h\Le \frac14$; then SG-UBU has an invariant distribution $\pi_h(\bm x,\bm v)$ in $\mathbb R^{2d}$, which satisfies
\begin{equation*}
	\int_{\mathbb R^{2d}} |\bm z|^8 \pi_h(\bm z)
	\D\bm z \Le \frac{Cd^4}{m^4},
\end{equation*}
where the constant $C$ depends only on $M_1,M_2$.
\end{lemma}

The proofs of the results above are given in Appendix~\ref{appendix: uniform-in-time bound}.

\subsection{Consistency: Local error analysis}

Building on the one-step error decomposition in \eqref{one-step error}, we establish fourth moment bounds for both the stochastic gradient error $\bm Z_{k+1} - \bar{\bm Z}_{k+1}$ and the discretization error $\bar{\bm Z}_{k+1} - \bm Z_k(h)$. These estimates are a prerequisite for deriving the MSE of SG-UBU. The explicit integral formulas for $\bm Z_{k+1}$, $\bar {\bm Z}_{k+1}$, and $\bm Z_k(h)$ used in this analysis are detailed in Appendix~\ref{appendix: local error}.

The stochastic gradient error, $\bm{Z}_{k+1} - \bar{\bm{Z}}_{k+1}$, has the concise form:
\begin{equation}
	\left\{
	\begin{aligned}
		\bm X_{k+1} - \bar{\bm X}_{k+1} & = -\frac{h(1-e^{-h})}{2M_2}(b(\bm Y_k,\theta_k) - \nabla U(\bm Y_k)), \\
		\bm V_{k+1} - \bar{\bm V}_{k+1} & = -\frac{he^{-h}}{M_2}
		(b(\bm Y_k,\theta_k) - \nabla U(\bm Y_k)),
	\end{aligned}
	\right.
	\label{stochastic gradient error expression}
\end{equation}
where the gradient is evaluated at the intermediate position $\bm Y_k$ from \eqref{Y_k expression}, which is independent of the random variable $\theta_k$. This independence, combined with the unbiased nature of the stochastic gradient, yields the conditional unbiasedness property:
\begin{equation}
	\mathbb E\big[\bm X_{k+1} - \bar{\bm X}_{k+1}\big|\bm Y_k\big] =
	\mathbb E\big[\bm V_{k+1} - \bar{\bm V}_{k+1}\big|\bm Y_k\big] = \reviewed{\bm 0}.
	\label{martingale property}
\end{equation}
This confirms that the error sequence $\{\bm Z_{k+1} - \bar {\bm Z}_{k+1}\}_{k=0}^\infty$ is a martingale difference sequence. The following lemma provides bounds on its fourth moments.
\begin{lemma}
	\label{lemma: stochastic gradient error}
	Under Assumptions~\ref{asP} and \ref{asSg}, let the step size $h\Le \frac14$, and the initial state
	$$
	\bm Z_0 = (\bm X_0,\bm V_0)\text{~with~}\bm X_0 = \bm 0\text{~and~}\bm V_0 \sim \mathcal N(\bm 0,M_2^{-1}\bm I_d);
	$$
	then the stochastic gradient error of SG-UBU satisfies
	\begin{equation*}
		\sup_{k\Ge0}\sqrt{\mathbb E\big|\bm X_{k+1} - \bar{\bm X}_{k+1}\big|^4}
		\Le C\sigma^2dh^4,\quad
		\sup_{k\Ge0}\sqrt{\mathbb E\big|\bm V_{k+1} - \bar{\bm V}_{k+1}\big|^4}
		\Le C\sigma^2dh^2,
	\end{equation*}
	where the constant $C$ depends only on $M_1, M_2$.
\end{lemma}

The analysis of the discretization error, $\bar{\bm Z}_{k+1} - \bm Z_k(h)$, follows the methodology in Theorem~25 of \cite{sanz2021wasserstein}, which employs a martingale decomposition technique. For notational convenience, we define the position and velocity errors as
\begin{equation*}
	\Dx = \bar{\bm X}_{k+1} - \bm X_k(h),\qquad
	\Dv = \bar{\bm V}_{k+1} - \bm V_k(h).
\end{equation*}
This technique decomposes the velocity error $\Dv$ into a martingale difference component $\Dmv$ and a higher-order remainder $\Dhv$. The martingale component satisfies the conditional unbiasedness property:
\begin{equation*}
	\mathbb E\big[\Dmv\big|\bm Z_k\big] = \reviewed{\bm 0}.
\end{equation*}
The following lemma provides fourth moment bounds for each of these error components.
\begin{lemma}
\label{lemma: discretization error}
Under Assumptions \ref{asP} and \ref{asSg}, let the step size $h\Le \frac14$, and the initial state
$$
	\bm Z_0 = (\bm X_0,\bm V_0)\text{~with~}\bm X_0 = \bm 0\text{~and~}\bm V_0 \sim \mathcal N(\bm 0,M_2^{-1}\bm I_d);
$$
then the discretization error of SG-UBU satisfies
\begin{equation*}
	\sup_{k\Ge0}\sqrt{\mathbb E|\Dx|^4} \Le \frac{Cdh^6}{m},\quad
	\sup_{k\Ge0}\sqrt{\mathbb E|\Dmv|^4} \Le \reviewed{Cdh^5},\quad
	\sup_{k\Ge0}\sqrt{\mathbb E|\Dhv|^4} \Le \frac{Cd^2h^6}{m^2},
\end{equation*}
where the constant $C$ depends only on $(M_i)_{i=1}^3$.
\end{lemma}

The proofs of the results above are given in Appendix~\ref{appendix: local error}.

\subsection{Mean square error of SG-UBU}

Integrating the stability and consistency results from the preceding sections, we now establish the MSE bound for the SG-UBU integrator. The proof leverages the time average error decomposition \eqref{explicit time average}, which is built upon the discrete Poisson solution $\phi_h(\bm x,\bm v)$.
\begin{theorem}
\label{theorem: SG-UBU MSE}
Under Assumptions \ref{asP}, \ref{asSg} and \ref{asT}, let the step size $h\Le \frac14$, and
$$
\bm Z_0 = (\bm X_0,\bm V_0)\text{~with~}\bm X_0 = \bm 0\text{~and~}\bm V_0 \sim \mathcal N(\bm 0,M_2^{-1}\bm I_d);
$$
then the SG-UBU solution satisfies
\begin{equation*}
	\mathbb E \Bigg[\bigg(\frac1K\sum_{k=0}^{K-1} f(\bm X_k) - \pi(f)\bigg)^2\Bigg] \Le
	C\bigg(\frac{d}{m^3Kh} + \frac{\sigma^4d^2h^2}{m^4} + \reviewed{\frac{d^2h^4}{m^5}\bigg(1+\frac{dh}{m^2}\bigg)} \bigg),
\end{equation*}
where the constant $C$ depends only on $(M_i)_{i=1}^3$ and $(L_i)_{i=1}^2$.
\end{theorem}
\begin{proof}[Sketch of Proof]
The proof proceeds by bounding the second moments of the cumulative error terms in the time average expression \eqref{explicit time average}:
\begin{equation*}
	\mathbb E\Bigg[\bigg(\sum_{k=0}^{K-1} R_k\bigg)^2\Bigg], \quad 
	\mathbb E\Bigg[\bigg(\sum_{k=0}^{K-1} S_k\bigg)^2\Bigg], \quad 
	\mathbb E\Bigg[\bigg(\sum_{k=0}^{K-1} T_k\bigg)^2\Bigg].
\end{equation*}
The core technique is to decompose each error term into a principal martingale difference component and a higher-order remainder. The leading-order martingale components for the stochastic error $R_k$ and the discretization error $S_k$ are, respectively,
\begin{equation*}
	\hat R_k = \frac1h(\bm Z_{k+1} - \bar{\bm Z}_{k+1})^\top \nabla \phi_h(\bar{\bm Z}_{k+1}),
	\qquad
	\hat S_k = \frac1h \Dmv^\top \nabla_{\bm v} \phi_h(\bm Z_k).
\end{equation*}
This decomposition allows for a sharp estimation of the second moments. The orthogonality of the martingale sequences ensures that their cumulative variance is the sum of the individual variances, simplifying the analysis and leading to the final MSE bound.
\end{proof}
Theorem~\ref{theorem: SG-UBU MSE} decomposes the MSE into three principal components:
\begin{itemize}
	\setlength{\itemsep}{2pt}
	\item \textbf{Stochastic fluctuation:} Decays with simulation time $Kh$ at a rate of $\mathcal{O}\big(\frac{d}{m^3Kh}\big)$.
	\item \textbf{Stochastic gradient error:} The first squared-bias component, scaling as $\mathcal{O}\big(\frac{\sigma^4 d^2h^2}{m^4}\big)$; the bias itself exhibits first-order convergence with respect to the step size $h$.
	\item \textbf{Discretization error:} The second squared-bias component, scaling as $\mathcal{O}\big(\reviewed{\frac{d^2h^4}{m^5}(1\hspace{-2.5pt}+\hspace{-2.5pt}\frac{dh}{m^2})}\big)$; the bias itself exhibits second-order convergence with respect to the step size $h$.
\end{itemize}
In the SG-UBU setting, the numerical bias is dominated by the stochastic gradient error.

The MSE provides a distinct characterization of numerical accuracy compared to metrics such as the Wasserstein distance. While the Wasserstein distance measures the discrepancy between the numerical and target distributions, the MSE quantifies the error and variance of the time-averaged estimate from a single trajectory. As this aligns more closely with how sampling algorithms are used in practice, the MSE often serves as a more direct measure of the computational effort required to achieve a desired level of accuracy.

The $\mathcal{O}(d/(m^3Kh))$ term in the MSE is mainly driven by the martingale component $\sum_{k=0}^{K-1} T_k$. Its asymptotic behavior can be understood by applying the Itô formula to the discrete Poisson solution. For a sufficiently small step size $h$, we have
\begin{equation*}
	\D \phi_h(\bm Z_k(s)) = (\mathcal L\phi_h) (\bm Z_k(s)) \D s + \frac2{\sqrt{M_2}} \nabla_{\bm v} \phi_h(\bm Z_k(s)) \cdot \D \bm B_{kh+s},
\end{equation*}
and integrating over $s\in[0,h]$ yields
\begin{align*}
	\frac{\phi_h(\bm Z_k(h)) - \phi_h(\bm Z_k)}{h} 
	& \approx (\mathcal L\phi_0)(\bm Z_k) + \frac2{\sqrt{M_2}h}\int_0^h \nabla_{\bm v} \phi_0(\bm Z_k(s))\cdot\D\bm B_{kh+s} \\
	& = -f(\bm X_k) + \pi(f) + \frac2{\sqrt{M_2}h}\int_0^h \nabla_{\bm v} \phi_0(\bm Z_k(s))\cdot\D\bm B_{kh+s}.
\end{align*}
Summing from $k=0$ to $K-1$, we find the continuous-time integral limit
\begin{align*}
	\frac1K \sum_{k=0}^{K-1} T_k & = 
	\frac1K \sum_{k=0}^{K-1} \bigg(
	\frac{\phi_h(\bm Z_k(h)) - \phi_h(\bm Z_k)}{h} + 
	f(\bm X_k) - \pi(f)
	\bigg) \\
	& \approx \frac2{\sqrt{M_2}Kh} \sum_{k=0}^{K-1} 
	\int_0^h \nabla_{\bm v} \phi_0(\bm Z_k(s))\cdot\D\bm B_{kh+s} \approx \frac2{\sqrt{M_2}T} \int_0^T \nabla_{\bm v}
	\phi_0(\bm Z_0(t))\cdot\D \bm B_t,
\end{align*}
where we identify the total sampling time as $T = Kh$. Consequently, the variance from this stochastic fluctuation is principally governed by
\begin{equation*}
	\frac{4}{M_2T^2} \int_0^T \mathbb E\big[|\nabla_{\bm v} \phi_0(\bm Z_0(t))|^2\big]\D t,
\end{equation*}
which is independent of the step size $h$. This reveals a key insight: this component of the sampling error is an intrinsic feature of the underlying underdamped Langevin dynamics \eqref{ULD} and is unaffected by the numerical discretization or the variance of the stochastic gradient.

This framework also provides the MSE estimate for the mini-batch SG-UBU algorithm when applied to the finite-sum potential in \eqref{finite-sum potential}, as detailed in the following result:

\begin{theorem}
	\label{theorem: mini-batch SG-UBU MSE}
	Under Assumptions \ref{asP}, \ref{asSgN} and \ref{asT}, let the step size $h\Le \frac14$, and
	$$
	\bm Z_0 = (\bm X_0,\bm V_0)\text{~with~}\bm X_0 = \bm 0\text{~and~}\bm V_0 \sim \mathcal N(\bm 0,M_2^{-1}\bm I_d);
	$$
	then the mini-batch SG-UBU solution with the batch size $p$ satisfies
	\begin{equation*}
		\mathbb E \Bigg[\bigg(\frac1K\sum_{k=0}^{K-1} f(\bm X_k) - \pi(f)\bigg)^2\Bigg] \Le
		C\bigg(\frac{d}{m^3Kh} + \frac{\sigma^4d^2h^2}{m^4p^2} + \reviewed{\frac{d^2h^4}{m^5}\bigg(1+\frac{dh}{m^2}\bigg)} \bigg),
	\end{equation*}
	where the constant $C$ depends only on $(M_i)_{i=1}^3$ and $(L_i)_{i=1}^2$.
\end{theorem}
The MSE bound in Theorem~\ref{theorem: mini-batch SG-UBU MSE} differs from the standard SG-UBU result in Theorem~\ref{theorem: SG-UBU MSE} only in the stochastic gradient error component. This term is reduced by a factor of $p^2$, which reflects the variance reduction achieved by using a mini-batch of size $p$.

The proofs of Theorems~\ref{theorem: SG-UBU MSE} and \ref{theorem: mini-batch SG-UBU MSE} are given in Appendix~\ref{appendix: SG-UBU 1}.

\subsection{Numerical bias of SG-UBU}

The MSE bound in Theorem~\ref{theorem: SG-UBU MSE} is composed of a stochastic fluctuation term that vanishes as the number of iterations $K\to\infty$, alongside persistent terms arising from the stochastic gradient and discretization errors. These latter components constitute the square of the numerical bias defined in \eqref{numerical bias}, which bounds the long-time limit of the MSE:
$$
\pi_h(f) - \pi(f) = \int_{\mathbb R^{2d}} 
f(\bm x) \Big(\pi_h(\bm x,\bm v) - \big(\pi\otimes\mathcal N(\bm 0,M_2^{-1}\bm I_d)\big)(\bm x,\bm v)\Big)\D\bm x\D\bm v.
$$
Here, $\pi_h(\bm x,\bm v)$ is an invariant distribution of the SG-UBU algorithm for a step size $h>0$; its existence is guaranteed by Lemma~\ref{lemma: SG-UBU invariant}, and uniqueness is neither established nor needed.

Beyond Theorem~\ref{theorem: SG-UBU MSE}, the discrete Poisson solution $\phi_h(\bm x,\bm v)$ also provides a direct expression for the numerical bias. For an arbitrary initial state $\bm Z_0 = (\bm X_0,\bm V_0)$, the definition of the discrete Poisson equation in \eqref{discrete Poisson eq} implies the identity:
\begin{equation}
	\mathbb E\bigg[\frac{\phi_h(\bm Z_0) - \phi_h(\bm Z_0(h))}{h}\bigg] = f(\bm X_0) - \pi(f),
	\label{bias 1}
\end{equation}
where the expectation is taken over the Brownian motion driving the exact one-step solution $\bm Z_0(h)$. Taking a further expectation with respect to the invariant distribution $\bm Z_0 \sim \pi_h$, and using the fact that the one-step SG-UBU update
\begin{equation*}
	\bm Z_1 = \big(\Phi^{\mathfrak U}_{h/2}\circ \Phi^{\mathfrak B}_h(\theta) \circ \Phi^{\mathfrak U}_{h/2}\big) \bm Z_0
\end{equation*}
also follows the distribution $\pi_h$, allows us to replace $\mathbb E^{\pi_h}[\phi_h(\bm Z_0)]$ with $\mathbb E^{\pi_h}[\phi_h(\bm Z_1)]$. This substitution leads to an exact expression for the numerical bias:
\begin{equation}
	\pi_h(f) - \pi(f)=
	\mathbb E_\theta^{\pi_h}\bigg[\frac{\phi_h(\bm Z_1) - \phi_h(\bm Z_0(h))}{h}\bigg],
	\label{bias 2}
\end{equation}
where the total expectation is over $\bm Z_0 \sim \pi_h$, the stochastic gradient index $\theta\sim \mathbb P_\theta$, and the Brownian motion $(\bm B_t)_{t\in[0,h]}$.

The right-hand side of \eqref{bias 2} corresponds to the local error of the SG-UBU integrator, which can be decomposed into the stochastic gradient error and the discretization error:
\begin{equation}
	\pi_h(f) - \pi(f) = \mathbb E_\theta^{\pi_h}\bigg[\frac{\phi_h(\bm Z_1) - \phi_h(\bar{\bm Z}_1)}{h}\bigg] +
	\mathbb E_\theta^{\pi_h}\bigg[\frac{\phi_h(\bar{\bm Z}_1) - \phi_h(\bm Z_0(h))}{h}\bigg],
	\label{bias decomposition}
\end{equation}
where $\bar{\bm Z}_1$ is the one-step FG-UBU update
\begin{equation*}
	\bar{\bm Z}_1 = \big(\Phi^{\mathfrak U}_{h/2}\circ \Phi^{\mathfrak B}_h \circ \Phi^{\mathfrak U}_{h/2}\big) \bm Z_0.
\end{equation*}
The decomposition \eqref{bias decomposition} reveals that the stochastic gradient error term, driven by $\bm Z_1 - \bar {\bm Z}_1$, is the dominant component that dictates the first-order convergence of the numerical bias. The discretization error, driven by $\bar{\bm Z}_1 - \bm Z_0(h)$, contributes to the higher-order terms.

Formalizing this decomposition allows for a precise quantification of the numerical bias, as stated in the following theorem. An additional upper bound of $\frac{m}{d}$ is imposed on the step size $h$ to ensure stability.
\begin{theorem}
	\label{theorem: SG-UBU numerical bias}
	Under Assumptions \ref{asP}, \ref{asSg} and \ref{asT}, let the step size $h\Le \min\{\frac14,\frac{m}{d}\}$; 
	then the numerical bias $\pi_h(f)-\pi(f)$ of SG-UBU satisfies
	\begin{equation*}
		\bigg|\pi_h(f) - \pi(f) -
		\frac{h}{2M_2^2} \mathbb E_\theta^{\pi_h}\Big[
		\big(b(\bm Y_0,\theta) - \nabla U(\bm Y_0)\big)^\top \bm M_0 \big(b(\bm Y_0,\theta) - \nabla U(\bm Y_0)\big)\Big]\bigg| \Le 
        \frac{Cdh^2}{m^3},
	\end{equation*}
    where the initial state $(\bm X_0,\bm V_0) \sim \pi_h$, and $\theta\sim\mathbb P_\theta$. The intermediate position
	$\bm Y_0$ defined in \eqref{Y_k expression} is the position component of $\Phi^{\mathfrak U}_{h/2} \bm Z_0$,
	$\bm M_0$ is an averaged velocity Hessian of $\phi_h(\bm x,\bm v)$:
	\begin{equation*}
		\bm M_0 = \int_0^1\bigg(2\lambda\int_0^1\nabla_{\bm v\bm v}^2 \phi_h\big(\bar{\bm Z}_1+\lambda\varphi(\bm Z_1 - \bar{\bm Z}_1)\big)\D\varphi\bigg)\D\lambda,
	\end{equation*}
	and the constant $C$ depends only on $(M_i)_{i=1}^3$ and $(L_i)_{i=1}^2$. Consequently,
	\begin{equation*}
		|\pi_h(f) - \pi(f)| \Le  C\bigg(\frac{\sigma^2dh}{m^2} + 
        \frac{dh^2}{m^3}
        \bigg).
	\end{equation*}
\end{theorem}
Using the inequality $\frac{dh}{m} \Le 1$, the numerical bias bound in Theorem~\ref{theorem: SG-UBU numerical bias} simplifies to
\begin{equation*}
    |\pi_h(f) - \pi(f)| \Le \frac{Cdh}{m^2},
\end{equation*}
which recovers the result in Table~\ref{table: bias}. In Theorem~\ref{theorem: SG-UBU numerical bias}, the intermediate position $\bm Y_0$ is independent of the stochastic gradient index $\theta\sim \mathbb P_\theta$, and the numerical bias $\pi_h(f) - \pi(f)$ converges at a first-order rate with respect to the step size $h$. Crucially, its leading-order coefficient is shown to be proportional to the variance of the stochastic gradient.

Theorem~\ref{theorem: SG-UBU numerical bias} can be readily adapted for mini-batch SG-UBU with the batch size $p$. The use of mini-batching reduces the stochastic gradient variance, scaling the leading-order bias term by $1/p$, thus
\begin{equation*}
    |\pi_h(f) - \pi(f)| \Le C\bigg(\frac{\sigma^2dh}{m^2p} + 
        \frac{dh^2}{m^3}\bigg).
\end{equation*}

In the limit as $h\to 0$, the leading-order coefficient can be informally identified as
\begin{equation}
	\small
	\lim_{h\rightarrow 0}\frac{\pi_h(f) - \pi(f)}{h} =
	\frac1{2M_2^2} \mathbb E_\theta^\pi\Big[
	\big(b(\bm x_0,\theta) - \nabla U(\bm x_0)\big)^\top \nabla_{\bm v\bm v}^2 \phi_0(\bm x_0,\bm v_0) \big(b(\bm x_0,\theta) - \nabla U(\bm x_0)\big)\Big],
	\label{h coefficient}
\end{equation}
where the expectation is taken over the stochastic gradient index $\theta\sim \mathbb P_\theta$, and the stationary distribution
$(\bm x_0,\bm v_0) \sim \pi(\bm x)\otimes \mathcal N(\bm 0,M_2^{-1}\bm I_d)$,
which is invariant under the underdamped Langevin dynamics \eqref{ULD}. This result highlights a significant feature of employing a second-order integrator: the discretization error is suppressed to $\mathcal{O}(h^2)$, allowing the stochastic gradient error alone to determine the leading-order $\mathcal{O}(h)$ bias.

A subtle issue arises in the structure of the leading-order coefficient in \eqref{h coefficient}. The expression involves the Hessian of the continuous Poisson solution, $\nabla_{\bm v\bm v}^2 \phi_0(\bm x_0,\bm v_0)$, and our analysis in Theorem~\ref{theorem: phi h estimate} required a bound on $\nabla^2 f(\bm x)$ to control this term. This creates an apparent inconsistency, as one might intuitively expect the leading-order bias for an expectation $\pi(f)$ to depend only on the first-order properties of the test function $f$. The following lemma resolves this apparent mismatch:
\begin{lemma}
	\label{lemma: paradox}
	Under Assumption \ref{asP}, let the test function $f(\bm x)$ satisfy $|\nabla f(\bm x)|\Le L_1$; then 
	the continuous Poisson solution $\phi_0(\bm x,\bm v)$ defined in \eqref{continuous Poisson solution} satisfies
	\begin{equation*}
		\mathbb E^\pi\Big[
		\big\|\nabla_{\bm v\bm v}^2 \phi_0(\bm x_0,\bm v_0)\big\|^2\Big]
		\Le \frac{C_d}{m^3},
	\end{equation*}
    where the expectation is taken over the stationary distribution $(\bm x_0,\bm v_0) \sim \pi(\bm x)\otimes \mathcal N(\bm 0,M_2^{-1}\bm I_d)$, and the constant $C_d$ depends only on $d$, $M_2$, $M_3$, and $L_1$.
\end{lemma}
This reconciles the analytical requirements with intuition, confirming that the magnitude of the leading-order bias is controlled by the first-order information of the test function.

In the specific case where the stochastic gradient noise is Gaussian, the leading-order coefficient in the numerical bias admits a more explicit form. Suppose 
\begin{equation*}
	b(\bm x,\theta) = \nabla U(\bm x) + \sigma \theta,~~~~
	\theta \sim \mathcal N(\bm 0,\bm I_d),
\end{equation*}
where the parameter $\sigma>0$ aligns with Assumption~\ref{asSg}. In this scenario, the expectation over the Gaussian noise $\theta$ in \eqref{h coefficient} can be computed explicitly, reducing the expression to involve the trace of the Hessian, i.e., the Laplacian:
\begin{equation}
	\lim_{h\rightarrow0} \frac{\pi_h(f) - \pi(f)}{h} =
	\frac{\sigma^2}{2M_2^2} \mathbb E^{\pi}[\Delta_{\bm v}\phi_0(\bm x_0,\bm v_0)].
	\label{coeff 1}
\end{equation}
By substituting the continuous Poisson equation into \eqref{coeff 1}, we can express the Laplacian in terms of first-order derivatives of $\phi_0$, thereby eliminating the dependence on its Hessian:
\begin{equation*}
	\lim_{h\rightarrow 0}\frac{\pi_h(f) - \pi(f)}{h} =
	\frac{\sigma^2}{4M_2} \mathbb E^{\pi}\bigg[\bigg(\frac1{M_2}\nabla U(\bm x_0)+2\bm v_0\bigg)^\top \nabla_{\bm v}\phi_0(\bm x_0,\bm v_0) -\bm v_0^\top \nabla_{\bm x} \phi_0(\bm x_0,\bm v_0)\bigg].
\end{equation*}
Given the uniform bound $|\nabla\phi_0|\Le C/m$, the leading-order coefficient is bounded as follows:
\begin{equation}
	\bigg|\lim_{h\rightarrow 0}\frac{\pi_h(f) - \pi(f)}{h}\bigg|
	\Le \frac Cm \mathbb E^\pi\big[|\nabla U(\bm x_0)| + |\bm v_0|\big] \Le \frac{Cd^{\frac12}}{m^{\frac32}}.
	\label{coeff 2}
\end{equation}
This bound suggests that the numerical bias remains controlled as long as $h = \mathcal O(d^{-1/2}m^{3/2})$. This condition on the step size is notably less restrictive than the condition $h = \mathcal{O}(d^{-1}m^2)$ needed to keep the general bias bound of Theorem~\ref{theorem: SG-UBU numerical bias} of order one.

It is important to emphasize that this derivation relies on the informal limit taken in \eqref{h coefficient}, and a rigorous justification of this limit is beyond the scope of our current assumptions. Nonetheless, this analysis motivates the conjecture that the less restrictive condition $h = \mathcal{O}(d^{-1/2}m^{3/2})$ may be sufficient to ensure the stability and first-order accuracy of SG-UBU in general. A proof of this conjecture is left for future investigation.

The proofs of the results above are given in Appendix~\ref{appendix: SG-UBU 2}.

\section{Mean square error and numerical bias of variance-reduced SG-UBU}
\label{section: MSE bias variance-reduced SG-UBU}

In this section, we extend our analysis to the variance-reduced sampling algorithms designed for finite-sum potentials of the form
$$
U(\bm{x}) = \frac{1}{N}\sum_{i=1}^{N} U_i(\bm{x}).
$$
We focus on two prominent techniques adapted to the UBU framework: Stochastic Variance Reduced Gradient (SVRG) \citep{johnson2013accelerating} and SAGA\footnote{SAGA is a variant of Stochastic Averaged Gradient (SAG), but not an acronym.} \citep{defazio2014saga}. While the numerical bias of the standard SG-UBU integrator is fundamentally limited by a first-order convergence rate due to stochastic gradient noise, the SVRG-UBU and SAGA-UBU variants exhibit a more complex behavior. Our analysis reveals that these methods undergo a phase transition in their convergence rates: the numerical bias shifts from a first-order to a second-order rate as the step size $h$ decreases below a critical threshold. This accelerated convergence is a distinctive feature that emerges from the interplay between variance reduction and a second-order accurate integrator.

Our analysis of the variance-reduced methods requires additional bounds on the individual potential components $\{\nabla U_i(\bm x)\}_{i=1}^N$.
\begin{asSgNp}
	\mylabel{asSgNp}{(SgN$^+$)}
	For the finite-sum potential $U(\bm x)$, there hold
	\begin{equation*}
		\bigg(\frac1N \sum_{i=1}^N |\nabla U_i(\bm x) - \nabla U(\bm x)|^8\bigg)^{\frac14} \Le \sigma^2d,\quad
		\bigg(\frac1N \sum_{i=1}^N \big\|\nabla^2 U_i(\bm x) - \nabla^2 U(\bm x)\big\|^8\bigg)^{\frac14} \Le \sigma^2,
		\quad
		\forall \bm x\in\mathbb R^d
	\end{equation*}
	for the constant $\sigma>0$, and for all $i\in\{1,\cdots,N\}$, there hold
	\begin{equation*}
		|\nabla U_i(\bm x)| \Le M_1\sqrt{|\bm x|^2+d},\quad
		\big\|\nabla^2 U_i(\bm x)\big\| \Le M_2,\quad
		\forall \bm x\in\mathbb R^d
	\end{equation*}	
	for constants $M_1,M_2>0$. Assume $\sigma\Le \min\{M_1,M_2\}$ for the convenience of analysis.
\end{asSgNp}
Assumption~\ref{asSgNp} is a direct strengthening of Assumption~\ref{asSgN}, as it extends the regularity conditions to include bounds on the Hessians of the individual components.

\subsection{SVRG-UBU analysis}

The SVRG-UBU algorithm adapts the SVRG approach, a control variate technique widely used to accelerate the convergence of Stochastic Gradient Descent (SGD) \citep{law2000simulation, johnson2013accelerating}.
This adaptation replaces the standard mini-batch gradient estimator in the mini-batch SG-UBU with a variance-reduced counterpart.

Given a mini-batch subset $\theta \subset \{1, \dots, N\}$ of size $p$ and a fixed anchor position $\bm x_* \in \mathbb R^d$, the SVRG estimator is defined as
\begin{equation}
	b(\bm x,\bm x_*,\theta) = \frac1p 
	\sum_{i\in\theta} 
	\nabla U_i(\bm x) - \frac1p \sum_{i\in\theta} 
	\nabla U_i(\bm x_*) + \nabla U(\bm x_*).
	\label{SVRG b}
\end{equation}
It is straightforward to verify that this estimator remains unbiased for the true gradient:
\begin{equation*}
	\mathbb E_{\theta}\big[
	b(\bm x,\bm x_*,\theta)
	\big] = \nabla U(\bm x).
\end{equation*}
The key to the variance reduction lies in the structure of the gradient error. The difference between the SVRG estimator and the true gradient can be expressed as:
\begin{equation}
	b(\bm x,\bm x_*,\theta) - \nabla U(\bm x) = \bigg[\int_0^1 \bigg(\frac1p\sum_{i\in\theta}
	\nabla^2 U_i - \nabla^2 U\bigg)\big(\lambda \bm x + (1-\lambda)\bm x_*\big) \D\lambda\bigg]
	\revised{(\bm x - \bm x_*)}.
    \label{SVRG gradient error}
\end{equation}
This expression reveals that as the sampling point $\bm x$ approaches the anchor position $\bm x_*$, the gradient error diminishes.
The full gradient $\nabla U(\bm x_*)$ serves as an effective control variate, substantially reducing the variance of the estimator when $\bm x$ is in a neighborhood of $\bm x_*$.

The SVRG-UBU algorithm is implemented in epochs to manage computational overhead. At the start of each epoch, a new anchor position $\bm x_*$ is established, and the corresponding full gradient $\nabla U(\bm x_*)$ is computed. This full gradient is then reused for a fixed number of inner iterations—typically $N/p$—before being updated. Each inner iteration evaluates $p$ component gradients, and this strategy amortizes the $\mathcal{O}(N)$ cost of the full gradient calculation over the $N/p$ iterations of an epoch, ensuring that the average computational cost per step remains $\mathcal{O}(p)$, on par with the standard mini-batch SG-UBU. \revised{For convenience we assume throughout that the epoch length $q = N/p$ is an integer; in practice it can be taken as $\lceil N/p \rceil$.} The detailed procedure is presented in Algorithm~\ref{algorithm: SVRG-UBU}.
\begin{algorithm}[h]
	\setstretch{1.12}
	\caption{Stochastic Variance Reduced Gradient UBU (SVRG-UBU)}
	\textbf{Input:} initial state $(\bm X_0,\bm V_0)$, batch size $p$, step size $h$ \\
	\textbf{Output:} numerical solution $(\bm X_k,\bm V_k)_{k=0}^\infty$ \\
	\For{$k=0,1,\cdots$}{
		Evolve $(\bm Y_k,\bm V_k^{(1)}) = \Phi^{\mathfrak U}_{h/2}(\bm X_k,\bm V_k)$ \\
		\eIf{$k$ is a multiple of $\frac Np$}{
			\mbox{}\\
			Set the anchor position $\bm Y_* = \bm Y_k$ \\[2pt]
			Compute $\bm b_k = \di\nabla U(\bm Y_*) = \frac1N\sum_{i=1}^N \nabla U_i(\bm Y_*)$ \\[-2pt]
		}{
			Sample a random subset $\theta_k \subset \{1,\cdots,N\}$ of size $p$ \\[2pt]
			Compute $\di\bm b_k =
			\frac1p\sum_{i\in\theta_k} \nabla U_{\reviewed{i}}(\bm Y_k) -
			\frac1p\sum_{i\in\theta_k} \nabla U_{\reviewed{i}}(\bm Y_*) + \nabla U(\bm Y_*)$ \\[-2pt]
		}
		Update the velocity $\bm V_k^{(2)} = \bm V_k^{(1)} - \frac{h}{M_2} \bm b_k$ \\
		Evolve $(\bm X_{k+1},\bm V_{k+1}) = \Phi^{\mathfrak U}_{h/2}(\bm Y_k,\bm V_k^{(2)})$
	}
	\label{algorithm: SVRG-UBU}
\end{algorithm}

The following theorem establishes the MSE bound for the SVRG-UBU integrator, paralleling the analysis for mini-batch SG-UBU in Theorem~\ref{theorem: SG-UBU MSE}.
\begin{theorem}
	\label{theorem: SVRG-UBU MSE}
	Under Assumptions \ref{asP}, \ref{asSgNp} and \ref{asT}, let the step size $h\Le \frac14$, and
	$$
	\bm Z_0 = (\bm X_0,\bm V_0)\text{~with~}\bm X_0 = \bm 0\text{~and~}\bm V_0 \sim \mathcal N(\bm 0,M_2^{-1}\bm I_d);
	$$
	then the SVRG-UBU solution with the batch size $p$ satisfies
	\begin{equation*}
		\mathbb E \left[\bigg(\frac1K\sum_{k=0}^{K-1} f(\bm X_k) - \pi(f)\bigg)^2\right] \Le
		C\bigg(\frac{d}{m^3Kh} + \frac{\sigma^4d^2h^2}{m^4p^2} \min\bigg\{1,\frac{N^4h^4}{m^2p^4}\bigg\}+ \reviewed{\frac{d^2h^4}{m^5}\bigg(1+\frac{dh}{m^2}\bigg)} \bigg),
	\end{equation*}
	where the constant $C$ depends only on $(M_i)_{i=1}^3$ and $(L_i)_{i=1}^2$.
\end{theorem}
Theorem~\ref{theorem: SVRG-UBU MSE} reveals a phase transition in the convergence rate of the numerical bias, which depends on the step size $h$ relative to a critical threshold. For large step sizes, where $h > \sqrt{m}p/N$, the bias is dominated by the stochastic gradient variance and exhibits first-order convergence, mirroring the behavior of the mini-batch SG-UBU. Conversely, when the step size is small, $h < \sqrt{m}p/N$, the variance reduction mechanism becomes effective, suppressing the first-order error term. This allows the underlying second-order accuracy of the UBU integrator to emerge, resulting in an accelerated convergence rate. The transition between these two distinct convergence regimes occurs around the critical step size of $\sqrt{m}p/N$.

Next, we characterize the numerical bias of SVRG-UBU, adapting the approach from Theorem~\ref{theorem: SG-UBU numerical bias} to account for the algorithm's non-Markovian structure. The analysis involves two key distributions. The first is $\pi_h(\bm x,\bm v)$, the invariant distribution of the embedded Markovian subsequence $(\bm X_{kN/p},\bm V_{kN/p})_{k=0}^\infty$, whose existence is guaranteed by Lemma~\ref{lemma: SVRG-UBU invariant}.

The second, which is used to measure the bias of the full trajectory, is the averaged distribution $\tilde \pi_h(\bm x,\bm v)$. It represents the average state over a single epoch of $q = N/p$ steps, assuming the process starts from the stationary distribution $\pi_h$. For any Borel set $A\subset\mathbb R^{2d}$, this averaged distribution is precisely defined as
\begin{equation}
    \tilde \pi_h(A) = \frac1q\sum_{k=0}^{q-1} \mathbb P\big(\bm Z_k \in A:\bm Z_0 \sim \pi_h\big).
    \label{averaged q}
\end{equation}
We note that only the subsequence of SVRG-UBU iterates sampled at epoch boundaries is Markovian. Consequently,
unlike Theorem~\ref{theorem: SG-UBU numerical bias}, which represents the local error in a single SG-UBU update, SVRG-UBU must use the numerical solution $(\bm X_k,\bm V_k)_{k=0}^q$ within an epoch to characterize the numerical bias, which consists of the stochastic gradient error $\bm Z_{k+1} - \bar {\bm Z}_{k+1}$ and the discretization error $\bar {\bm Z}_{k+1} - \bm Z_k(h)$ in the total $q$ steps. Recall that $\bar {\bm Z}_{k+1}$ is the one-step FG-UBU update from $\bm Z_k$, namely,
\begin{equation*}
    \bar{\bm Z}_{k+1} = (\Phi^{\mathfrak U}_{h/2}\circ \Phi^{\mathfrak B}_h \circ \Phi^{\mathfrak U}_{h/2}) \bm Z_k.
\end{equation*}

The following theorem provides a precise characterization of the numerical bias associated with the time-averaged distribution $\tilde \pi_h$ defined in \eqref{averaged q}. 
\begin{theorem}
	\label{theorem: SVRG-UBU numerical bias}
	Under Assumptions \ref{asP}, \ref{asSgNp} and \ref{asT}, let the step size $h\Le \min\{\frac14,\frac{m}{d}\}$; then
	the numerical bias $\tilde \pi_h(f)-\pi(f)$ of SVRG-UBU with the batch size $p = N/q$ satisfies
	\begin{equation*}
		\Bigg|\tilde \pi_h(f) - \pi(f) -
		\frac{h}{2M_2^2q} \mathbb E_{\theta_0,\cdots,\theta_{q-1}}^{\pi_h}\bigg[
        \sum_{k=0}^{q-1}
		\big(\bm g(\bm Y_k,\bm Y_0,\theta_k)\big)^\top \bm M_k \big(\bm g(\bm Y_k,\bm Y_0,\theta_k)\big) \bigg]\Bigg| \Le \frac{Cdh^2}{m^3},
	\end{equation*}
    where $\bm g(\bm x,\bm x_*,\theta)$ is the stochastic gradient deviation with the subset $\theta$:
    \begin{equation*}
        \bm g(\bm x,\bm x_*,\theta) = \frac1p \sum_{i\in\theta} 
        \nabla U_i(\bm x) - \frac1{p} \sum_{i\in\theta} 
        \nabla U_i(\bm x_*) + \nabla U(\bm x_*) - \nabla U(\bm x),
    \end{equation*}
    the numerical solution $(\bm X_k,\bm V_k)_{k=0}^{q}$ in an epoch is generated by Algorithm~\ref{algorithm: SVRG-UBU}, the initial state $(\bm X_0,\bm V_0) \sim \pi_h$, and $\theta_0,\cdots,\theta_{q-1}$ are independent random subsets of $\{1,\cdots,N\}$ with size $p$. The intermediate position $\bm Y_k$ defined in \eqref{Y_k expression} is the position component of $\Phi^{\mathfrak U}_{h/2} \bm Z_k$,
	$\bm M_k$ is an averaged velocity Hessian of $\phi_h(\bm x,\bm v)$:
	\begin{equation*}
		\bm M_k = \int_0^1\bigg(2\lambda\int_0^1\nabla_{\bm v\bm v}^2 \phi_h\big(\bar{\bm Z}_{k+1}+\lambda\varphi(\bm Z_{k+1} - \bar{\bm Z}_{k+1})\big)\D\varphi\bigg)\D\lambda,
	\end{equation*}
	and the constant $C$ depends only on $(M_i)_{i=1}^3$ and $(L_i)_{i=1}^2$. Consequently,
	\begin{equation*}
		|\tilde \pi_h(f) - \pi(f)| \Le C\biggl(\frac{\sigma^2dh}{m^2p} \min
		\bigg\{1,\frac{N^2h^2}{mp^2}\bigg\} + \frac{dh^2}{m^3}
        \biggr).
	\end{equation*}
\end{theorem}
In Theorem~\ref{theorem: SVRG-UBU numerical bias}, the intermediate position $\bm Y_k$ depends on the history $\{\theta_0,\cdots,\theta_{k-1}\}$ but does not depend on the current random subset $\theta_k$.
Theorem~\ref{theorem: SVRG-UBU numerical bias} reveals that the numerical bias of SVRG-UBU is primarily determined by the average variance of the stochastic gradient over a single epoch, with residual errors of order $\mathcal O(\frac{dh^2}{m^3})$. A crucial aspect of this result is that the bias is quantified with respect to the averaged distribution $\tilde \pi_h$, rather than the invariant distribution $\pi_h$ of the anchor position subsequence. This choice is deliberate and aligns with the practical goal of measuring the time average error. Since the MSE in Theorem~\ref{theorem: SVRG-UBU MSE} is calculated over the entire numerical trajectory $(\bm X_k,\bm V_k)_{k=0}^\infty$, the averaged distribution $\tilde \pi_h$ provides the correct corresponding measure for the systematic bias.

The proofs of the results above are given in Appendix~\ref{appendix: proof SVRG-UBU}.

\subsection{SAGA-UBU analysis}

The SAGA-UBU algorithm adapts the SAGA method, which, unlike SVRG, avoids periodic full gradient computations. Instead, it maintains a table of historical gradients for each of the $N$ potential components, requiring $\mathcal{O}(dN)$ memory to store this information.

The SAGA estimator constructs a control variate from this table of stored gradients. For each component $i \in \{1, \dots, N\}$, the algorithm stores the gradient $\nabla U_i(\bm \Phi_i)$ evaluated at the historical position $\bm \Phi_i$ where it was last computed. At a new sampling position $\bm x$, the SAGA gradient estimator for a mini-batch $\theta$ of size $p$ is defined as
\begin{equation}
    b\big(\bm x,\{\bm\Phi_i\}_{i=1}^N,\theta\big) = 
    \frac1p \sum_{i\in\theta} \nabla U_i(\bm x) - \frac1p 
    \sum_{i\in\theta} \nabla U_i(\bm\Phi_i) + \frac1N 
    \sum_{i=1}^N \nabla U_i(\bm\Phi_i),
\end{equation}
where $\theta$ is a random subset of $\{1, \dots, N\}$. This estimator is unbiased for the full gradient, as can be readily verified by taking the expectation over $\theta$:
\begin{equation*}
    \mathbb E\big[b\big(\bm x,\{\bm\Phi_i\}_{i=1}^N,\theta\big)\big] = \nabla U(\bm x).
\end{equation*}
By replacing the standard stochastic gradient in the UBU framework with this SAGA estimator, we arrive at the SAGA-UBU algorithm.
\begin{algorithm}[htbp]
	\caption{SAGA-UBU}
	\label{algorithm: SAGA-UBU}
	\setstretch{1.12}
	\textbf{Input:} initial state $(\bm X_0,\bm V_0)$, batch size $p$, step size $h$ \\
	\textbf{Output:} numerical solution $(\bm X_k,\bm V_k)_{k=0}^\infty$ \\
	Evolve $(\bm Y_0,\bm V_0^{(1)}) = \Phi^{\mathfrak U}_{h/2}(\bm X_0,\bm V_0)$ \\
	Set $\bm\Phi_i = \bm Y_0$ for $i=1,\cdots,N$ \\
	Evolve  $\bm V_0^{(2)} = \bm V_0^{(1)} - \frac{h}{M_2} \nabla U(\bm Y_0)$ \\
	Evolve $(\bm X_1,\bm V_1) = \Phi^{\mathfrak U}_{h/2}(\bm Y_0,\bm V_0^{(2)})$ \\
	\For{$k=1,2,\cdots$}{
		Evolve $(\bm Y_k,\bm V_k^{(1)}) = \Phi^{\mathfrak U}_{h/2}(\bm X_k,\bm V_k)$ \\
		Sample a random subset $\theta_k\subset\{1,\cdots,N\}$ of size $p$ \\
		Compute $ \di
			\bm b_k = \frac1p\sum_{i\in\theta_k}\nabla U_i(\bm Y_k) - \frac1p \sum_{i\in\theta_k}
            \nabla U_i(\bm\Phi_i) + \frac1N \sum_{i=1}^N
			\nabla U_i(\bm\Phi_i)$ \\[2pt]
		Set $\bm \Phi_i = \bm Y_k$ for all $i\in\theta_k$ \\
		Evolve $\bm V_k^{(2)} = \bm V_k^{(1)} - \frac{h}{M_2} \bm b_k$ \\
		Evolve $(\bm X_{k+1},\bm V_{k+1}) = \Phi^{\mathfrak U}_{h/2}(\bm Y_k,\bm V_k^{(2)})$
	}
\end{algorithm}

The SAGA-UBU algorithm generates a non-Markovian solution sequence $(\bm X_k, \bm V_k)_{k=0}^\infty$ due to its reliance on a continuously updated table of historical gradients. This structure is more complex than that of SVRG-UBU, which possesses an embedded Markovian subsequence. Quantifying the stochastic gradient error for a general batch size becomes significantly more challenging in the SAGA setting. For analytical tractability, our analysis therefore focuses on the specific case where the batch size $p=1$. The following theorem establishes the MSE bound for SAGA-UBU under this condition.
\begin{theorem}
	\label{theorem: SAGA-UBU MSE}
	Under Assumptions \ref{asP}, \ref{asSgNp} and \ref{asT}, let the step size $h\Le c_0m$, and
	$$
	\bm Z_0 = (\bm X_0,\bm V_0)\text{~with~}\bm X_0 = \bm 0\text{~and~}\bm V_0 \sim \mathcal N(\bm 0,M_2^{-1}\bm I_d);
	$$
	then the SAGA-UBU solution with the batch size $1$ satisfies
	\begin{equation*}
		\mathbb E \Bigg[\bigg(\frac1K\sum_{k=0}^{K-1} f(\bm X_k) - \pi(f)\bigg)^2\Bigg] \Le
		C\bigg(\frac{d}{m^3Kh} + \frac{d^2h^2}{m^6} \min\big\{1,N^4h^4\big\}+ \reviewed{\frac{d^2h^4}{m^5}\bigg(1+\frac{dh}{m^2}\bigg)} \bigg),
	\end{equation*}
	where $c_0$ depends only on $(M_i)_{i=1}^2$, and $C$  depends only on $(M_i)_{i=1}^3$ and $(L_i)_{i=1}^2$.
\end{theorem}

\noindent
Theorem~\ref{theorem: SAGA-UBU MSE} requires an additional upper bound $c_0m$ on the step size $h$. This requirement stems from the modified Lyapunov condition satisfied by SAGA-BU (Lemma~\ref{lemma: SAGA-BU Lyapunov}), which results in a nonlinear recursive inequality. The constant $c_0$ is then specifically determined in Theorem~\ref{theorem: SAGA-UBU moment bound} to ensure the uniform-in-time moments.

SAGA-UBU (Theorem~\ref{theorem: SAGA-UBU MSE}) and SVRG-UBU (Theorem~\ref{theorem: SVRG-UBU MSE}) exhibit analogous convergence characteristics, including the phase transition concerning the step size $h$. However, the error bounds derived for SAGA-UBU are weaker than those for SVRG-UBU. This difference stems from the more stringent requirements needed in the analysis to control the moments of the SAGA-UBU solution and the variance of its stochastic gradient estimator, primarily due to its non-Markovian structure and reliance on historical gradient information.

Next, we characterize the numerical bias for SAGA-UBU. Because the SAGA-UBU solution $(\bm X_k,\bm V_k)_{k=0}^\infty$ is non-Markovian, we measure the bias using the limit distribution $\tilde\pi_h(\bm x,\bm v)$, which is defined as the subsequential weak limit of the time-averaged distributions
\begin{equation*}
		\tilde{\pi}_h^K(A) = \frac{1}{K}
		\sum_{k=0}^{K-1} \mathbb{P} \big(
		\bm{Z}_k \in A \mid \bm{X}_0 = \bm{0},\ \bm{V}_0 \sim \mathcal{N}(\bm{0}, M_2^{-1}\bm{I}_d)
		\big)
\end{equation*}
whose existence is guaranteed by Prokhorov's theorem, as stated in Lemma~\ref{lemma: SAGA-UBU invariant}. The theorem below presents the bound for the numerical bias $\tilde \pi_h(f)-\pi(f)$.
\begin{theorem}
	\label{theorem: SAGA-UBU numerical bias}
	Under Assumptions \ref{asP}, \ref{asSgNp} and \ref{asT}, let the step size $h\Le\min\{c_0m,\frac{m}{d}\}$; then
	the numerical bias $\tilde \pi_h(f)-\pi(f)$ of SAGA-UBU with the batch size $1$ satisfies
	\begin{equation*}
		|\tilde \pi_h(f) - \pi(f)| \Le
		C\bigg(\frac{dh}{m^3}\min\big\{1,N^2h^2\big\} + 
        \frac{dh^2}{m^3} \bigg),
	\end{equation*}
	where $c_0$ depends only on $(M_i)_{i=1}^2$, and $C$  depends only on $(M_i)_{i=1}^3$ and $(L_i)_{i=1}^2$.
\end{theorem}

The proofs of the results above are given in Appendix~\ref{appendix: proof SAGA-UBU}.

\section{Numerical verification}
\label{section: numerical}

In this section, we implement numerical tests to validate our theoretical results on the convergence of stochastic sampling algorithms. For this purpose, we restrict our tests to low-dimensional target distributions.

\subsection{First-order convergence of numerical bias for SG-UBU}

To verify the leading-order coefficient equality \eqref{h coefficient} for SG-UBU, we test 1D and 2D potential functions using both Gaussian noise and finite-sum stochastic gradients. Computationally, the limit in the LHS of \eqref{h coefficient} is evaluated by running SG-UBU at multiple step sizes $h$. The leading-order coefficient in the RHS is computed by numerically calculating the Hessian matrix $\nabla^2_{\bm v\bm v} \phi_0(\bm x,\bm v) \in \mathbb R^{d\times d}$, where the algorithmic details are provided in Appendix~\ref{appendix: leading-order coefficient}.

\paragraph{1D Example}
We define the potential function $U(x)$ and test function $f(x)$ as
\begin{align*}
	U(x) & = \frac{1}{2}x^2 + 0.15 \sin(1.6x-0.5) + 0.1 \sin(2.4x+0.4), \\
	f(x) & = \cos x + 0.5 \sin(2.5x) + 0.2 \sin(0.5x+0.4),
\end{align*}
which are plotted in Figure~\ref{figure: 1D_UF}.

\begin{center}
	\includegraphics[width=0.53\textwidth]{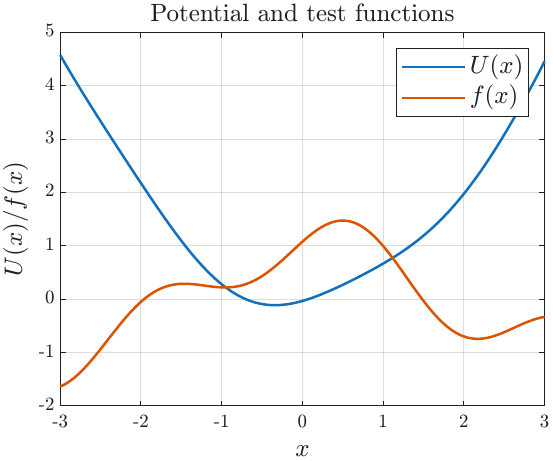}
	\captionof{figure}{Graphs of the 1D potential function $U(x)$ and the test function $f(x)$.}
	\label{figure: 1D_UF}
\end{center}

We test two stochastic gradient models: Gaussian noise and finite-sum cases. In both cases, the numerical bias $\pi_h(f) - \pi(f)$ is computed from SG-UBU simulations with varying step sizes $h$ and a total time $T=10^6$. The leading-order coefficient is computed from $5\times10^7$ independent FG-UBU trajectories using $h=2^{-10}$.

\begin{itemize}
	\item \textbf{Gaussian noise case:} The stochastic gradient is
	\begin{equation*}
		b(x,\theta) = \nabla U(x) + 3\theta, \quad \theta \sim \mathcal{N}(0,1).
	\end{equation*}
	
	\item \textbf{Finite-sum case:} The stochastic gradient is
	\begin{equation*}
		b(x,\theta) = \nabla U(x) + \nabla V_\theta(x)
	\end{equation*}
	where the random index $\theta$ is uniformly distributed on $\{1,\cdots,N\}$, and
	\begin{equation*}
		V_i(x) = a_i \sin x + b_i \cos(1.2x) + c_i \sin(2x) + d_i \cos(2.5x), \quad i=1,\cdots,N.
	\end{equation*}
	The parameters $\{a_i,b_i,c_i,d_i\}_{i=1}^N$ are drawn uniformly from $[-6,6]$ and then shifted to be mean-zero, which ensures that $b(x,\theta)$ is an unbiased estimator for $\nabla U(x)$.
\end{itemize}
Figure~\ref{figure: 1D_bias} plots the computed numerical bias $\pi_h(f) - \pi(f)$ against the step size $h$ on a log-log scale. We compare this to the theoretical linear approximation (i.e., a line with a slope given by the independently computed leading-order coefficient), which shows strong agreement.

\begin{center}
	\includegraphics[width=0.48\textwidth]{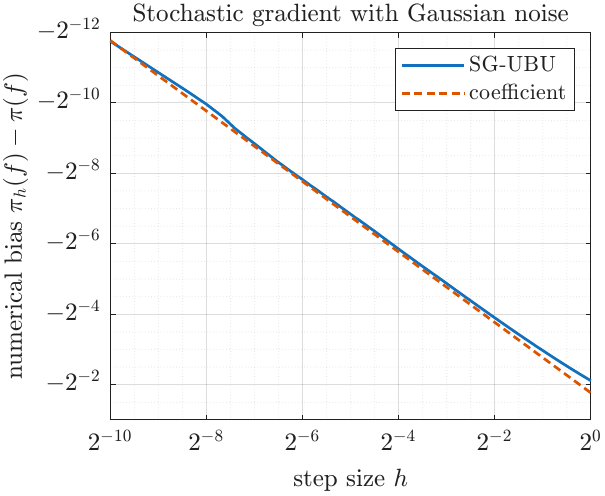}
	\includegraphics[width=0.48\textwidth]{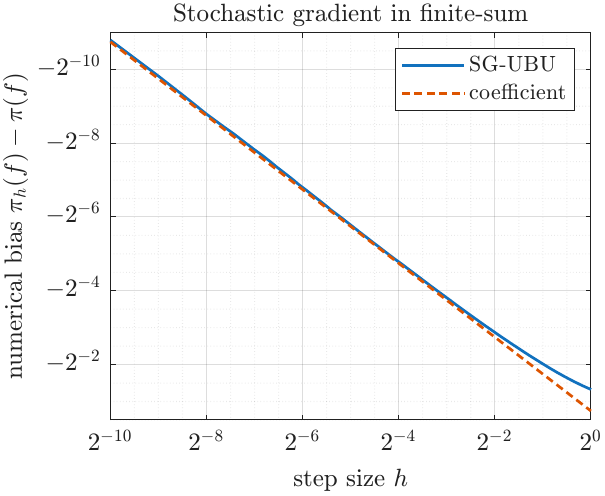}
	\captionof{figure}{Log-log plot of the SG-UBU numerical bias $\pi_h(f) - \pi(f)$ vs. step size $h$ for the 1D potential. The computed bias (blue solid line) is compared to the theoretical linear approximation (red dashed line). (Left) Gaussian noise case. (Right) Finite-sum case.}
	\label{figure: 1D_bias}
\end{center}

\paragraph{2D Example}
We define the potential function $U(x_1,x_2)$ and test function $f(x_1,x_2)$ as:
\begin{align*}
    U(x_1,x_2) & = \frac{1}{2}\big(1.4x_1^2 + 0.8x_2^2 + \sin(0.7x_1-x_2) \cos(0.4x_1+0.6x_2)\big), \\
    f(x_1,x_2) & = \cos\big(1.4x_1-1.1\sin(1.2x_2)\big),
\end{align*}
which are plotted in Figure~\ref{figure: 2D_UF}.

\begin{figure}[htb]
    \includegraphics[width=0.48\textwidth]{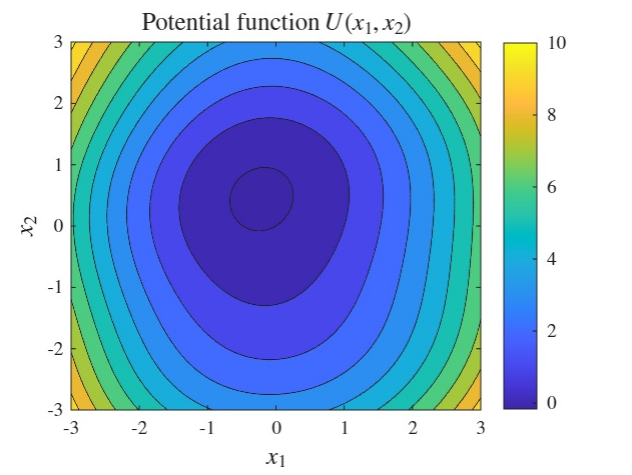}
    \includegraphics[width=0.48\textwidth]{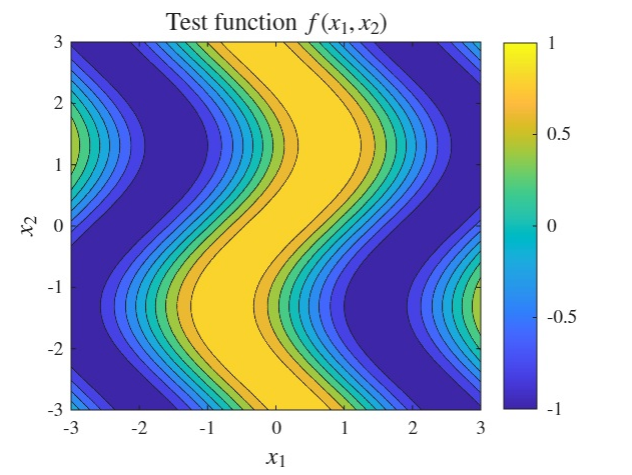}
    \caption{Graphs of the 2D potential function $U(x_1,x_2)$ and the test function $f(x_1,x_2)$.}
    \label{figure: 2D_UF}
\end{figure}

We test two stochastic gradient models: Gaussian noise and finite-sum cases. In both cases, the numerical bias $\pi_h(f) - \pi(f)$ is computed from SG-UBU simulations, and the parameter settings are the same as in the 1D example.

\begin{itemize}
    \item \textbf{Gaussian noise case:} The stochastic gradient is
    \begin{equation*}
        b(\bm{x},\theta) = \nabla U(\bm{x}) + 3\theta,\quad \theta \sim \mathcal{N}(0,\bm{I}_2).
    \end{equation*}
    
    \item \textbf{Finite-sum case:} The stochastic gradient is
    \begin{equation*}
        b(\bm{x},\theta) = \nabla U(\bm{x}) + \nabla V_\theta(\bm{x}),
    \end{equation*}
    where the random index $\theta$ is uniformly distributed on $\{1,\cdots,N\}$, and
    \begin{equation*}
        V_i(x_1,x_2) = a_i \sin(x_1+2x_2) + b_i \cos(1.2x_1-0.7x_2) + c_i e^{-\frac{1}{2}x_1^2} + d_i e^{-\frac{1}{3}x_2^2},~~i=1,\cdots,N.
    \end{equation*}
    The parameters $\{a_i,b_i,c_i,d_i\}_{i=1}^N$ are drawn uniformly from $[-8,8]$ and then shifted to be mean-zero, which ensures that $b(x,\theta)$ is an unbiased estimator for $\nabla U(x)$.
\end{itemize}
Figure~\ref{figure: 2D_bias} plots the computed numerical bias $\pi_h(f) - \pi(f)$ against the step size $h$ on a log-log scale. We compare this to the theoretical linear approximation (i.e., a line with a slope given by the independently computed leading-order coefficient), which shows strong agreement.

\begin{figure}[htb]
    \includegraphics[width=0.48\textwidth]{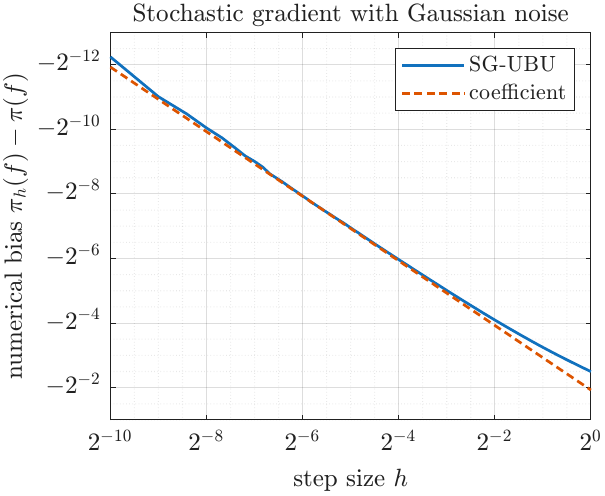}
    \includegraphics[width=0.48\textwidth]{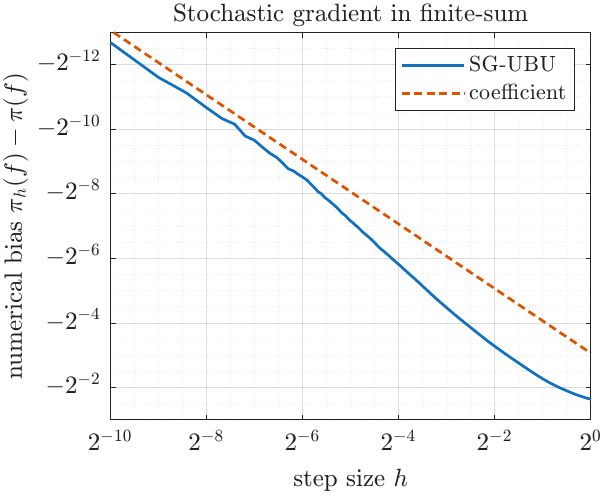}
    \caption{Log-log plot of the SG-UBU numerical bias $\pi_h(f) - \pi(f)$ vs. step size $h$ for the 2D potential. The computed bias (blue solid line) is compared to the theoretical linear approximation (red dashed line). (Left) Gaussian noise case. (Right) Finite-sum case.}
    \label{figure: 2D_bias}
\end{figure}
Results demonstrate that $\pi_h(f) - \pi(f)$ exhibits first-order convergence in the step size $h$, with the leading-order coefficient accurately predicting the linear component at moderately small step sizes. This provides numerical validation for the equality \eqref{h coefficient}. However, SG-UBU simulations in practice reveal significant nonlinear bias dependence on $h$, particularly in the finite-sum case. For the 2D example, the finite-sum bias exhibits approximate linear convergence only when $h \Le 2^{-6}$.

\subsection{Comparison of variance-reduced sampling algorithms}
\label{subsection: comparison}

We compare the numerical performance of three sampling algorithms: mini-batch SG-UBU, SVRG-UBU, and SAGA-UBU. Consider a finite-sum potential $U(\bm x)$ in $\mathbb R^{10}$ defined as
\begin{equation*}
	U(\bm x) = \frac1N \sum_{i=1}^N U_i(\bm x), \quad \text{where} \quad
	U_i(\bm x) = \frac16|\bm x|^2 + D_0( w_i^\top \bm x+b_i), \quad i=1,\cdots,N,
\end{equation*}
where the function $D_0(x) = 16 e^{-\frac12x^2} - 8\cos(x) - 4\sin(2x)$ introduces non-convexity. \revised{The potential $U(\bm x)$ is non-convex, so this test lies outside Assumption~\ref{asP} and probes whether the predicted phase transition persists beyond the regime our theorems cover.} The parameters $(w_i,b_i)_{i=1}^N$ are randomly generated as follows:
\begin{equation*}
	w_{i,j} = \frac{j^{\frac14} \xi_{i,j}}
	{\sqrt{\sum_{l=1}^{10} l^{\frac12} \xi_{i,l}^2}}, \quad
	j=1,\cdots,10, \quad \text{and} \quad b_i = \frac{\cos i + \eta_i}{10},
\end{equation*}
where $\xi_{i,j}, \eta_i \sim \mathcal N(0,1)$ are independent standard Gaussian random variables.

For convenience, we take the coefficient $M_2$ appearing in the underdamped Langevin dynamics \eqref{ULD} to be $1$, \revised{which is a scaling choice for the dynamics rather than a value of the Hessian bound in Assumption~\ref{asP}}, resulting in the stochastic differential equation
\begin{equation}
	\left\{
	\begin{aligned}
		\dot {\bm x}_t & = \bm v_t, \\
		\dot {\bm v}_t & = -\nabla U(\bm x_t) - 2\bm v_t + 2\dot{\bm B}_t.
	\end{aligned}
	\right.
	\label{ULD 2}
\end{equation}
The stochastic gradient sampling algorithms generate numerical solutions $(\bm X_k,\bm V_k)_{k=0}^\infty$ by discretizing \eqref{ULD 2} with the step size $h$. To measure the sampling error, we select the vector-valued test function $f: \mathbb R^{10} \to \mathbb R^{30}$ as
\begin{equation*}
	f(\bm x) = \Big(e^{-\frac{1}{8d}|\bm x|^2 - \frac{1}{4}x_j^2}, e^{-\frac{1}{8d}|\bm x|^2 - \frac{1}{4}(x_j+2)^2}, e^{-\frac{1}{8d}|\bm x|^2 - \frac{1}{4}(x_j+4)^2}\Big)_{j=1,\cdots, 10} \in \mathbb R^{30}.
\end{equation*}
Finally, the numerical performance of the algorithms is characterized by the time average error after $K=T/h$ steps:
\begin{equation*}
	\text{sampling error} = \bigg|\frac1K \sum_{k=0}^{K-1} f(\bm X_k) - \pi(f)\bigg|,
\end{equation*}
where $|\cdot|$ denotes the Euclidean norm in $\mathbb R^{30}$. \revised{Assumption~\ref{asT} and the theorems are stated for a scalar test function; they apply to each of the $30$ components separately, and we report the Euclidean norm over the components.} The sampling error comprises two components: the stochastic fluctuation depending on the total simulation time $T = Kh$, and the numerical bias depending on the step size $h$.

\paragraph{Sampling error dependence on $N$.}
We fix the batch size $p=1$ and test the performance of the algorithms with various $N$.
Figure~\ref{figure: variance error} shows log-log plots of the sampling error versus the step size $h$ for SG-UBU, SVRG-UBU, and SAGA-UBU. We test different dataset sizes $N\in\{10,50,100,500\}$ with a fixed simulation time $T=10^7$.

\begin{figure}[htb]
	\includegraphics[width=0.48\textwidth]{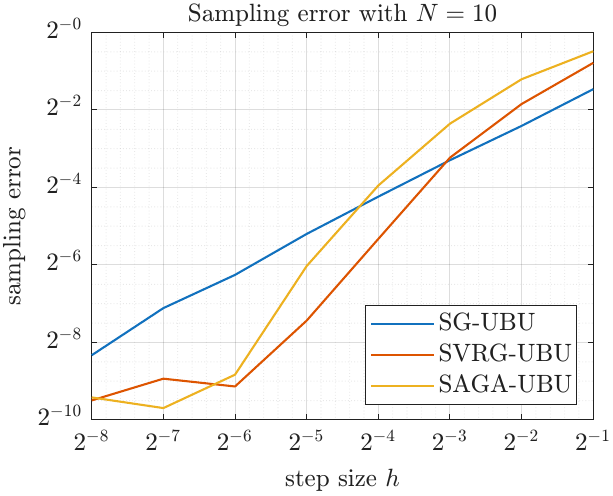}
	\includegraphics[width=0.48\textwidth]{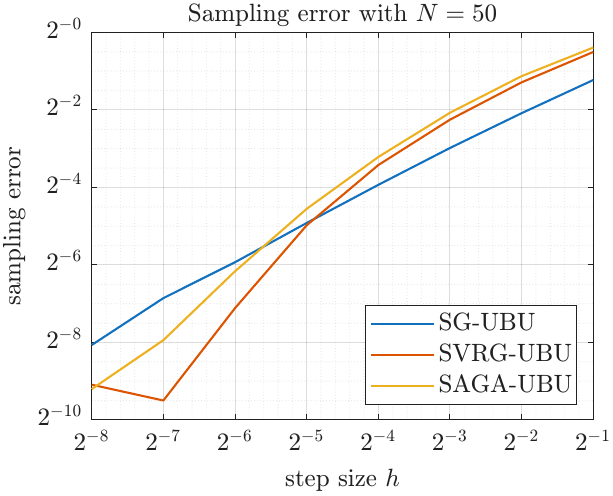} \\
	\includegraphics[width=0.48\textwidth]{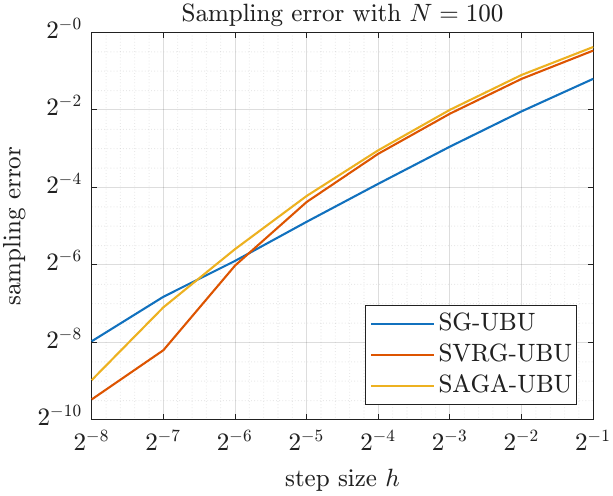}
	\includegraphics[width=0.48\textwidth]{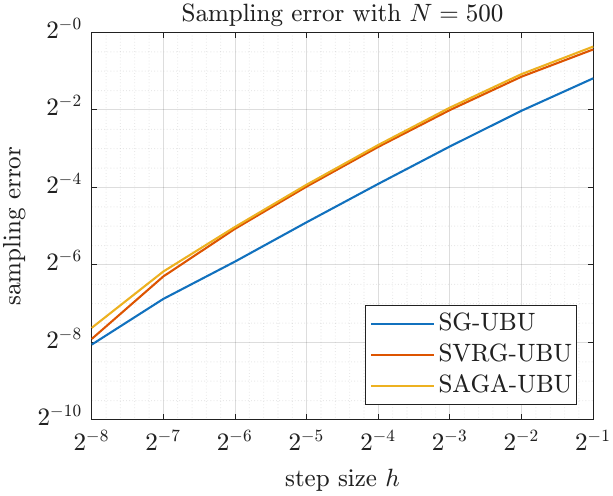}
	\caption{Log-log plot of the sampling error vs. step size $h$ for SG-UBU, SVRG-UBU, and SAGA-UBU for different numbers of components $N\in\{10,50,100,500\}$.}
	\label{figure: variance error}
\end{figure}

\revised{As predicted by Theorems~\ref{theorem: SVRG-UBU MSE} and \ref{theorem: SAGA-UBU MSE}, and observed here even outside their hypotheses,} SVRG-UBU and SAGA-UBU exhibit a phase transition in convergence order near the critical step size at $\mathcal O(1/N)$. For $N=10$, the error reduction for SVRG-UBU and SAGA-UBU saturates for $h < 2^{-6}$, likely due to stochastic fluctuations becoming dominant. For $N=500$, the variance-reduced methods can perform worse than SG-UBU in certain ranges of $h$. This occurs because a larger $N$ can increase the variance introduced by the control variate strategy itself when the anchor point is far from the current sampling point.

In summary, SVRG-UBU and SAGA-UBU achieve better error than SG-UBU only when $h$ is sufficiently small, specifically $h < \mathcal O(1/N)$. Thus, these methods may underperform SG-UBU when $N$ is large or when the target error tolerance is not sufficiently stringent.

\paragraph{Sampling error dependence on $p$.}
We fix $N=100$ and $T=10^7$ and compare the performance of mini-batch SG-UBU and SVRG-UBU for different batch sizes $p \in \{2,4,8,16\}$. Table~\ref{table: sampling error p} shows the sampling errors of these two algorithms. As predicted by theory, for a fixed $T$, the sampling error converges to a constant (dominated by stochastic fluctuations) as $h$ becomes sufficiently small, which is evident in each column.

\begin{table}[htb]
\centering
	\renewcommand{\arraystretch}{1.2}
	\begin{tabular}{c|cc|cc}
	\toprule
		& \multicolumn{2}{c|}{$p=2$}
		 & \multicolumn{2}{c}{$p=4$} \\
	$h$ & mini-batch SG-UBU & SVRG-UBU & mini-batch SG-UBU & SVRG-UBU \\
	\midrule
	$2^{-1}$ & $2.37\times10^{-1}$& $4.27\times10^{-1}$& $1.18\times10^{-1}$& $2.17\times10^{-1}$\\
	$2^{-2}$ & $1.26\times10^{-1}$& $2.28\times10^{-1}$& $6.21\times10^{-2}$& $1.06\times10^{-1}$\\
	$2^{-3}$ & $6.57\times10^{-2}$& $1.11\times10^{-1}$& $3.14\times10^{-2}$& $4.51\times10^{-2}$\\
	$2^{-4}$ & $3.37\times10^{-2}$& $4.75\times10^{-2}$& $1.60\times10^{-2}$& $1.47\times10^{-2}$\\
	$2^{-5}$ & $1.67\times10^{-2}$& $1.55\times10^{-2}$& $7.27\times10^{-3}$& $3.39\times10^{-3}$\\
	$2^{-6}$ & $8.54\times10^{-3}$& $3.35\times10^{-3}$& $3.96\times10^{-3}$& $1.36\times10^{-3}$\\
	$2^{-7}$ & $4.12\times10^{-3}$& $1.30\times10^{-3}$& $2.30\times10^{-3}$& $1.60\times10^{-3}$\\
	$2^{-8}$ & $1.90\times10^{-3}$& $1.65\times10^{-3}$& $1.24\times10^{-3}$& $1.33\times10^{-3}$\\
	\midrule
	& \multicolumn{2}{c|}{$p=8$}
	& \multicolumn{2}{c}{$p=16$}  \\
	$h$ & mini-batch SG-UBU & SVRG-UBU & mini-batch SG-UBU & SVRG-UBU \\
	\midrule
	$2^{-1}$ & $5.20\times10^{-2}$& $9.31\times10^{-2}$& $1.75\times10^{-2}$& $2.88\times10^{-2}$\\
	$2^{-2}$ & $2.81\times10^{-2}$& $4.12\times10^{-2}$& $1.12\times10^{-2}$& $1.10\times10^{-2}$\\
	$2^{-3}$ & $1.42\times10^{-2}$& $1.40\times10^{-2}$& $6.18\times10^{-3}$& $3.22\times10^{-3}$\\
	$2^{-4}$ & $6.72\times10^{-3}$& $2.87\times10^{-3}$& $3.19\times10^{-3}$& $1.26\times10^{-3}$\\
	$2^{-5}$ & $3.08\times10^{-3}$& $1.09\times10^{-3}$& $1.89\times10^{-3}$& $1.61\times10^{-3}$\\
	$2^{-6}$ & $1.66\times10^{-3}$& $1.00\times10^{-3}$& $1.64\times10^{-3}$& $1.57\times10^{-3}$\\
	$2^{-7}$ & $1.45\times10^{-3}$& $1.27\times10^{-3}$& $1.36\times10^{-3}$& $1.70\times10^{-3}$\\
	$2^{-8}$ & $1.77\times10^{-3}$& $1.36\times10^{-3}$& $1.28\times10^{-3}$& $1.77\times10^{-3}$\\
	\bottomrule
	\end{tabular}
	\caption{Sampling error vs. step size $h$ for mini-batch SG-UBU and SVRG-UBU at $N=100$ with varying batch sizes $p\in\{2,4,8,16\}$.}
	\label{table: sampling error p}
\end{table}

To quantify the relative performance, we define the sampling error ratio $R(p,h)$ as
\begin{equation*}
	R(p,h) = \frac{\text{sampling error of SVRG-UBU with batch size } p \text{ and step size } h}
	{\text{sampling error of mini-batch SG-UBU with batch size } p \text{ and step size } h}.
\end{equation*}
Figure~\ref{figure: ratio} shows the plot of $R(p,h)$ for $N=100$ and $p\in\{2,4,8,16\}$.

\begin{figure}[htb]
    \centering
    \includegraphics[width=0.6\textwidth]{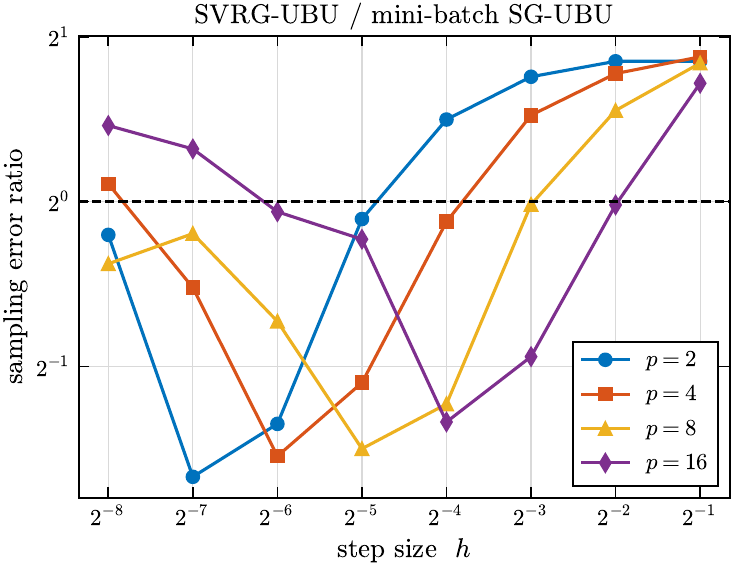}
	\caption{Sampling error ratio $R(p,h)$ (SVRG-UBU error / mini-batch SG-UBU error) vs. step size $h$ at $N = 100$ for batch sizes $p\in\{2,4,8,16\}$. Values below 1 (black dashed line) indicate that SVRG-UBU performs better than mini-batch SG-UBU.}
	\label{figure: ratio}
\end{figure} 

Figure~\ref{figure: ratio} illustrates that for each $p$, the sampling error ratio $R(p,h)$ first decreases and then increases towards $1$ as $h$ decreases. The minimum value of $R(p,h)$ occurs near the critical step size $\mathcal O(p/N)$, aligning with the prediction from Theorem~\ref{theorem: SVRG-UBU MSE}. Based on the MSE bounds within the theorem, for a fixed simulation time $T$, $R(p,h)$ can be approximated as
\begin{equation}
	R(p,h) \sim \frac{\di\frac1{\sqrt{dT}}+\frac hp\min\bigg\{1,\frac{N^2h^2}{p^2}\bigg\}+h^2}
	{\di\frac1{\sqrt{dT}}+\frac hp+h^2},
	\label{ratio approximation}
\end{equation}
where we ignore the constants and dependence on $m, \sigma$ for simplicity. When $h \Le \frac pN$, the $\min$ term becomes small, and \eqref{ratio approximation} implies $R(p,h) \Le 1$, indicating that SVRG-UBU outperforms mini-batch SG-UBU. However, when $h \Le  \frac{p}{\sqrt{dT}}$ we have
\begin{equation*}
	R(p,\reviewed{h}) \Ge \frac{\di\frac1{\sqrt{dT}}+h^2}{\di\frac2{\sqrt{dT}}+h^2} \Ge \frac12,
\end{equation*}
showing that the advantage of SVRG-UBU diminishes. \reviewed{Taken together, the two limits characterize the range of step sizes over which an advantage of SVRG-UBU is observed:}
\begin{equation}
	\frac{p}{\sqrt{dT}} \Le h\Le \frac{p}{N}.
	\label{SVRG rule}
\end{equation}

\reviewed{The two ends of \eqref{SVRG rule} are lost for different reasons. When $h \ll \frac{p}{\sqrt{dT}}$, the sampling error is dominated by the stochastic fluctuation of the finite simulation time $T$ rather than by the stochastic gradient, so removing gradient variance changes little of the total error. When $h > \frac{p}{N}$, an epoch spans the time $Nh/p$, over which the trajectory travels a considerable distance before the full gradient is recomputed; the anchor position is then a poor reference for the control variate, and the variance it removes shrinks accordingly. Accordingly, \eqref{SVRG rule} is a qualitative description of where variance reduction is effective, observed in Figure~\ref{figure: ratio}, rather than a quantitative prescription.}

\subsection{Algorithm selection by computational cost}
\label{subsection: selection}

Building upon the MSE bounds derived in Theorems~\ref{theorem: SG-UBU MSE} and \ref{theorem: SVRG-UBU MSE}, we propose the following procedure for selecting between mini-batch SG-UBU and SVRG-UBU and determining the simulation parameters. \reviewed{The two algorithms are compared at the same target accuracy: the simulation time $T$ is common to both, only the step size differs, and the choice is made by comparing the resulting computational cost. The range \eqref{SVRG rule} is used only as a reference for where variance reduction is effective.}

First, the batch size $p$ is typically chosen based on hardware constraints, such as the parallel processing capabilities of the CPU or GPU. Given $p$, the procedure is as follows:

\begin{enumerate}
    \item \textbf{Determine the common simulation time $T$:}

    For a target accuracy level $\ep > 0$ (representing the desired upper bound on the square root of the MSE), set the total simulation time $T$ to control the stochastic fluctuation component, e.g., $T = d/\ep^2$, such that the $\mathcal{O}(d/(Kh))$ term in the MSE is roughly $\ep^2$. \reviewed{This term is common to both algorithms, so $T$ is the same for both.}

    \item \textbf{Determine the step size of each algorithm:}

    \reviewed{The stochastic gradient and discretization bias terms differ between the two algorithms, so each admits its own largest step size. Take $h_{\mathrm{SG}}$ and $h_{\mathrm{SVRG}}$ to be the largest step sizes satisfying}
	\begin{equation*}
		\reviewed{\left\{
		\begin{aligned}
			& \frac{dh}{p} \Le \ep, \\
			& dh^2 \Le \ep,
		\end{aligned}
		\right.
		\qquad \text{and} \qquad
		\left\{
		\begin{aligned}
			& \frac{dh}{p} \min\bigg\{1,\frac{N^2h^2}{p^2}\bigg\} \Le \ep, \\
			& dh^2 \Le \ep,
		\end{aligned}
		\right.}
	\end{equation*}
    \reviewed{respectively. Both choices give a total error of order $\ep$, and $h_{\mathrm{SVRG}} \Ge h_{\mathrm{SG}}$, since the $\min$ term is at most one.}

    \item \textbf{Compare the computational cost and select:}

    \reviewed{With $K = T/h$ iterations, the number of component gradient evaluations is}
	\begin{equation*}
		\reviewed{\mathcal C_{\mathrm{SG}} = \frac{pT}{h_{\mathrm{SG}}},
		\qquad
		\mathcal C_{\mathrm{SVRG}} = \frac{pT}{h_{\mathrm{SVRG}}} \bigg(3 - \frac{2p}{N}\bigg),}
	\end{equation*}
    \reviewed{where the factor of roughly $3$ counts the two mini-batch gradients and the amortized full gradient of each SVRG-UBU epoch. Select the algorithm with the smaller cost, and set $K = T/h$ accordingly.}
\end{enumerate}

\reviewed{Since both algorithms are run to the same accuracy $\ep$, the selection reduces to the cost comparison in Step 3. As $\ep$ decreases, $h_{\mathrm{SVRG}}$ decreases only as $\ep^{1/2}$ while $h_{\mathrm{SG}}$ decreases as $\ep$, so $\mathcal C_{\mathrm{SVRG}}/\mathcal C_{\mathrm{SG}}$ decreases and SVRG-UBU eventually becomes the cheaper choice.}

\reviewed{We verify this procedure on a finite-sum target in $\mathbb R^8$ whose invariant expectation is known exactly. Each component potential is a symmetric two-component Gaussian mixture,}
\begin{equation*}
	\reviewed{U_i(\bm x) = -\log\Big(e^{-\frac12|\bm x-\bm a_i|^2}+e^{-\frac12|\bm x+\bm a_i|^2}\Big)
	= \frac12|\bm x|^2 - \log\cosh(\bm a_i\cdot\bm x) + \text{const},}
\end{equation*}
\reviewed{so that}
\begin{equation*}
	\reviewed{U(\bm x) = \frac12|\bm x|^2 - \frac1N \sum_{i=1}^N \log\cosh(\bm a_i\cdot\bm x),
	\qquad
	\nabla^2 U(\bm x) = \bm I_d - \frac1N \sum_{i=1}^N \mathrm{sech}^2(\bm a_i\cdot\bm x)\, \bm a_i\bm a_i^\top,}
\end{equation*}
\reviewed{up to a constant. As $0<\mathrm{sech}^2\Le1$, we have $\bm I_d - \frac1N\sum_i \bm a_i\bm a_i^\top \Prec \nabla^2 U \Prec \bm I_d$. We take $N=100$ and draw the $\bm a_i$ independently, each with direction uniform on the unit sphere of $\mathbb R^d$ and length uniform on $[1.05,1.95]$, giving $\lambda_{\max}\big(\frac1N\sum_i \bm a_i\bm a_i^\top\big) = 0.44$, so that Assumption~\ref{asP} holds. Each component $U_i$ is bimodal, with modes at $\pm\bm a_i$, but their average is strongly convex, so the target is log-concave and free of barriers.}

\reviewed{Since $U$ is even, $\pi$ is symmetric about the origin, and the exact value $\pi(f)=0$ is available for any odd test function without a reference computation; we take $f(\bm x) = \tanh(\bm w\cdot\bm x)$ with $\bm w$ a fixed unit vector. To control the stochastic gradient variance independently of the target, we add offsets $\bm c_i$ to $\nabla U_i$ alone, drawn independently from $\mathcal N(\bm 0,4\bm I_d)$ and then recentered so that $\sum_{i=1}^N \bm c_i = \bm 0$. The recentering leaves $\frac1N\sum_i \nabla U_i = \nabla U$, hence $U$ and $\pi$, exactly unchanged, while raising the deviation to $\sigma\approx5.4$. We use $p=4$ and average over $64$ independent trajectories.}

\begin{figure}[htb]
	\centering
	\includegraphics[width=0.49\textwidth]{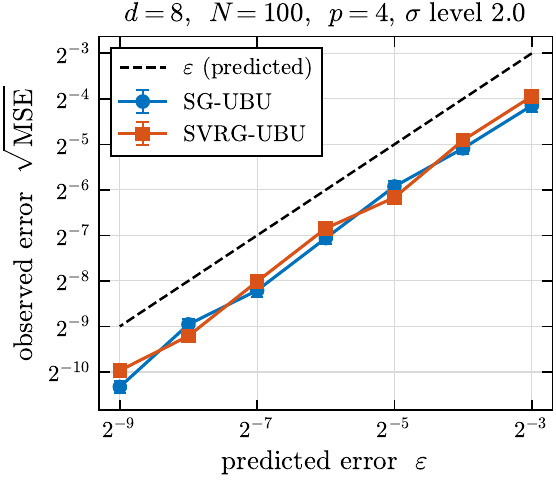}
	\includegraphics[width=0.49\textwidth]{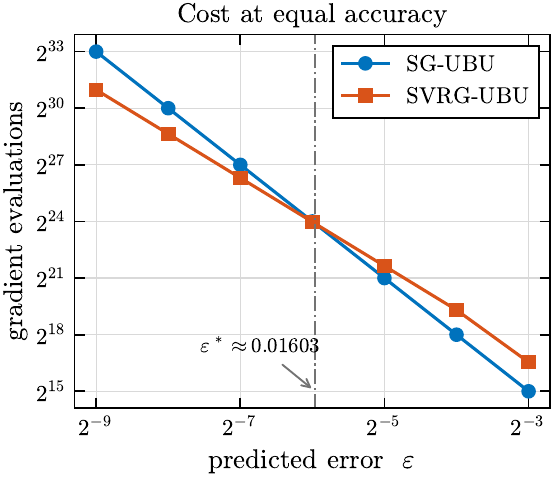}
	\caption{\reviewed{Verification of the selection procedure for the $\mathbb R^8$ Gaussian-mixture target with $N=100$, $p=4$. (Left) Observed error $\sqrt{\MSE}$ against the target accuracy $\ep$, with the diagonal $\ep$ shown dashed; error bars are the sampling uncertainty of a root mean square over $64$ trajectories. (Right) Component gradient evaluations required by each algorithm, with the step size of each set by Step 2 and the common simulation time $T=d/\ep^2$. The vertical line marks the target accuracy $\ep$ at which the two costs are equal.}}
	\label{figure: selection}
\end{figure}

\reviewed{Figure~\ref{figure: selection} (left) shows that the observed error is proportional to $\ep$ for both algorithms over the range $\ep\in[2^{-9},2^{-3}]$, with ratios $\sqrt{\MSE}/\ep$ of $0.47\pm0.05$ for mini-batch SG-UBU and $0.50\pm0.05$ for SVRG-UBU. The parameters produced by Steps 1 and 2 therefore deliver the requested accuracy for both, and the selection is governed by cost alone. Figure~\ref{figure: selection} (right) shows the two costs crossing at $\ep\approx1.6\times10^{-2}$: mini-batch SG-UBU is cheaper by a factor up to $2.9$ for larger $\ep$, while SVRG-UBU is cheaper below the crossing, by a factor $4.1$ at $\ep=2^{-9}$ and increasing. The crossing depends only on $d$, $N$ and $p$ through the two step sizes, and not on $\sigma$, which affects only whether the target accuracy is attained.}

\section{Conclusion}
\label{section: conclusion}

This paper leverages the discrete Poisson equation framework to conduct a rigorous analysis of the mean square error for three stochastic gradient sampling algorithms---SG-UBU, SVRG-UBU, and SAGA-UBU---under the assumption of a globally convex potential function. Our work advances the theoretical understanding of sampling algorithms in several key aspects. First, for the baseline SG-UBU algorithm, we establish that its numerical bias exhibits first-order convergence with respect to the step size, where the leading coefficient is proportional to the stochastic gradient variance. This result quantifies the explicit impact of gradient noise on discretization bias. Second, for the variance-reduced variants SVRG-UBU and SAGA-UBU, we uncover a distinctive phase transition phenomenon: their numerical convergence order shifts between first- and second-order regimes depending on the step size magnitude. This behavior highlights the nuanced interplay between variance reduction techniques and discretization schemes in second-order integrators. Third, we derive a practical criterion to guide algorithm selection between mini-batch SG-UBU and mini-batch SVRG-UBU. \reviewed{Given a target accuracy, the two samplers are assigned a common simulation time and their own step sizes, and the one requiring fewer component gradient evaluations is selected; the variance-reduced sampler becomes the cheaper choice once the accuracy demand is sufficiently stringent}.

The analytical value of our contributions lies in the precision and generality of the discrete Poisson framework, which enables sharp error decomposition that is difficult to achieve through traditional Wasserstein analysis alone. Specifically, our approach cleanly separates stochastic noise, discretization error, and variance reduction effects, providing a simple bias bound for SG-UBU-type methods under global convexity. For variance-reduced algorithms, the identified step size-dependent phase transition offers new insights into the design of high-precision samplers. These theoretical advances strengthen the foundation for developing efficient sampling algorithms in high-dimensional settings, particularly in Bayesian inference and other applications requiring scalable, accurate integration over complex distributions. The discrete Poisson methodology, with its flexibility in handling time average errors and numerical bias, further establishes itself as a powerful tool for analyzing stochastic sampling algorithms.

However, certain limitations in our analysis warrant mention. First, the explicit dependence of the error bounds on the dimension $d$ and the convexity parameter $m$ may be suboptimal, and our framework does not readily yield convergence rates in standard Wasserstein distances. This constraint arises from the structure of the discrete Poisson methodology itself, which fundamentally targets mean square error analysis. Additionally, the global convexity assumption is used to establish pointwise bounds on the Hessian of the Kolmogorov solution. \reviewed{Relaxing this assumption is left to future work.}

	
\acks{The work of Z. Zhou was partially supported by the National Key R\&D Program of China (Project No. 2021YFA1001200), and the National Natural Science Foundation of China (Grant No. 12171013).

The authors would like to thank Xu'an Dou at Peking University for the helpful discussions on the proof of Lemma~\ref{lemma: paradox}. }

\vskip 0.2in

\bibliography{reference.bib}


\newpage

\tableofcontents

\appendix

	\section{Estimate of discrete Poisson solution}
	\label{appendix: discrete Poisson}
	
	This section is dedicated to proving Theorem~\ref{theorem: phi h estimate}, which provides quantitative bounds on the derivatives of the discrete Poisson solution $\phi_h(\bm{x},\bm{v})$ introduced in \eqref{discrete Poisson solution}. Our analysis relies on the \emph{first variation process}\footnote{This term is also known as the \emph{tangent process} or \emph{Jacobian process} in the literature. We adopt the term \emph{first variation process}, which is common in Malliavin calculus.}, a technical tool that captures the sensitivity of the underdamped Langevin dynamics \eqref{ULD} to perturbations in its initial state $(\bm x,\bm v)$. The idea of utilizing the variation process is inspired by \cite{chak2023optimal}.
	
	\subsection{First variation process: Definition}
	
	Let $(\bm x_t,\bm v_t)_{t\Ge0}$ denote the exact solution to the underdamped Langevin dynamics \eqref{ULD} with the initial state $(\bm x_0,\bm v_0)$. The first variation process $D_{\bm v}\bm x_t \in \mathbb R^{d\times d}$ is defined as the Jacobian matrix of $\bm x_t$ with respect to the initial velocity $\bm v_0$, specifically expressed as:
	\begin{equation*}
		D_{\bm v} \bm x_t = \big[\partial_{v_j} x_i\big]_{i,j=1}^d =
		\begin{bmatrix}
			\partial_{v_1} \bm x_t & \partial_{v_2} \bm x_t & \cdots & \partial_{v_d} \bm x_t
		\end{bmatrix} \in \mathbb R^{d\times d}.
	\end{equation*}
	The companion first variation processes $D_{\bm x} \bm x_t,D_{\bm x} \bm v_t$ and $D_{\bm v} \bm v_t$ admit analogous definitions. Through differentiation of the Kolmogorov solution $u(\bm x_0,\bm v_0,t)$ in \eqref{function u}, we derive the gradient expressions for $u(\bm x_0,\bm v_0,t)$:
	\begin{equation}
		\begin{aligned}
			\nabla_{\bm x} u(\bm x_0,\bm v_0,t) & = \mathbb E^{(\bm x_0,\bm v_0)}
			\big[\big(D_{\bm x} \bm x_t\big)^\top\nabla f(\bm x_t) \big], \\
			\nabla_{\bm v} u(\bm x_0,\bm v_0,t) & = \mathbb E^{(\bm x_0,\bm v_0)}
			\big[\big(D_{\bm v} \bm x_t\big)^\top \nabla f(\bm x_t) \big].
		\end{aligned}
		\label{gradient u expression}
	\end{equation}
	This formulation reveals that the exponential decay of $\nabla u(\bm x_0,\bm v_0,t)$ originates directly from the analytical properties of the first variation processes $D_{\bm x} \bm x_t$ and $D_{\bm v} \bm x_t$.
	
	By differentiating \eqref{ULD} with respect to the initial state $(\bm x_0,\bm v_0)$, we obtain the governing ordinary differential equations for the first variation processes.
	\begin{definition}
		Let $(\bm x_t,\bm v_t)_{t\Ge0}$ be the solution to \eqref{ULD} with initial state $(\bm x_0,\bm v_0)$ and Brownian motion $(\bm B_t)_{t\Ge0}$. Define $(\bm Q_t,\bm P_t)_{t\Ge0}$ as the solution to the matrix-valued differential equation
		\begin{equation}
			\left\{
			\begin{aligned}
				\dot{\bm Q}_t & = \bm P_t, \\
				\dot{\bm P}_t & = -\frac{1}{M_2}\nabla^2 U(\bm x_t) \bm Q_t - 2 \bm P_t,
			\end{aligned}
			\right.
			\label{first variation}
		\end{equation}
		with the initial configuration $(\bm Q_0,\bm P_0) \in \mathbb R^{2d\times d}$, formally expressed as
		\begin{equation*}
			(\bm Q_t,\bm P_t)_{t\Ge0} = \mathrm{FirstVariation}\big(\bm Q_0,\bm P_0,\bm x_0,\bm v_0,(\bm B_t)_{t\Ge0}\big).
		\end{equation*}
	\end{definition}
	\noindent
	At the initial time $t=0$, the Jacobian matrix of $(\bm x_0,\bm v_0)$ with respect to $\bm x_0$ yields $(\bm I_d,\bm O_d)$, thus identifying the first variation process
	\begin{equation*}
		(D_{\bm x} \bm x_t, D_{\bm x} \bm v_t)_{t\Ge0}  =
		\mathrm{FirstVariation}\big(\bm I_d,\bm O_d,\bm x_0,\bm v_0,(\bm B_t)_{t\Ge0}\big).
	\end{equation*}
	Analogously, the velocity-related variation process satisfies
	\begin{equation*}
		(D_{\bm v} \bm x_t, D_{\bm v} \bm v_t)_{t\Ge0}  =
		\mathrm{FirstVariation}\big(\bm O_d,\bm I_d,\bm x_0,\bm v_0,(\bm B_t)_{t\Ge0}\big).
	\end{equation*}
	
	We proceed to demonstrate the exponential decay property of the first variation process $(\bm Q_t,\bm P_t)_{t\Ge0}$. Define the composite matrices
	\begin{equation*}
		\bm W_t=\begin{bmatrix}
			\bm Q_t \\ \bm P_t
		\end{bmatrix} \in \mathbb R^{2d\times d}, \qquad
		\bm S = \begin{bmatrix}
			2 \bm I_d & \bm I_d \\
			\bm I_d & \bm I_d
		\end{bmatrix} \in \mathbb R^{2d\times 2d},
	\end{equation*}
	and consider the \emph{twisted} quadratic form
	\begin{equation}
		\bm W_t^\top \bm S \bm W_t = 2\bm Q_t^\top \bm Q_t + \bm Q_t^\top \bm P_t + \bm P_t^\top \bm Q_t + \bm P_t^\top \bm P_t \in \mathbb R^{d\times d}.
		\label{twisted form W}
	\end{equation}
	The twisted quadratic form plays a fundamental role in hypocoercivity analysis \citep{villani2009hypocoercivity}, particularly for establishing ergodicity and Lyapunov conditions in the underdamped Langevin dynamics \citep{eberle2019couplings}.
	
	\subsection{First variation process: Proof of exponential decay}
	We show that the exponential decay of the quadratic form $\bm W_t^\top \bm S \bm W_t$ can be rigorously established under the global convexity of the potential function $U(\bm x)$.
	
	\begin{lemma}
	\label{lemma: first variation 1}
	Under Assumption~\ref{asP} with the initial configuration
		\begin{equation*}
			(\bm Q_0,\bm P_0)=(\bm I_d,\bm O_d)~\text{or}~(\bm O_d,\bm I_d),
		\end{equation*}
		the first variation process $\bm W_t$ satisfies the exponential decay property
		\begin{equation*}
			\bm W_t^\top \bm S \bm W_t \Prec 2e^{-\frac{mt}{M_2}} \bm I_d, \quad \forall t\Ge0.
		\end{equation*}
	\end{lemma}
	
	\begin{proof}
		\label{Proof: first variation 1}
		Define the auxiliary matrices
		\begin{equation}
			\bm H_t = \frac{1}{M_2}\nabla^2 U(\bm x_t) \in \mathbb R^{d\times d}, \qquad
			\bm A_t = \begin{bmatrix}
				\bm O_{d} & -\bm I_{d} \\
				\bm H_t & 2 \bm I_{d}
			\end{bmatrix} \in \mathbb R^{2d\times 2d},
			\label{HA expression}
		\end{equation}
		then the first variation process in \eqref{first variation} admits the matrix form $\dot{\bm W}_t = -\bm A_t \bm W_t$. Note that Assumption~\ref{asP} yields the matrix eigenvalue bound
		\begin{equation*}
			\frac{m}{M_2} \bm I_d \Prec \bm H_t \Prec \bm I_d,
		\end{equation*}
		and we obtain the following matrix inequality as the coercivity condition:
		\begin{equation}
			\bm A_t^\top \bm S + \bm S \bm A_t =
			\begin{bmatrix}
				2\bm H_t & \bm H_t \\
				\bm H_t & 2 \bm I_d
			\end{bmatrix} \Succ
			\begin{bmatrix}
				2\bm H_t & \bm H_t \\
				\bm H_t & 2 \bm H_t
			\end{bmatrix} \Succ \frac{m}{M_2} \begin{bmatrix}
				2 \bm I_d & \bm I_d \\
				\bm I_d & 2\bm I_d
			\end{bmatrix} \Succ \frac{m}{M_2} \bm S.
			\label{proof: S inequality}
		\end{equation}
		Differentiating the quadratic form $\bm W_t^\top \bm S\bm W_t$ with respect to $t$ yields
		\begin{equation*}
			\frac{\D}{\D t}(\bm W_t^\top \bm S\bm W_t) = -\bm W_t^\top (\bm A_t^\top \bm S + \bm S \bm A_t) \bm W_t \Prec -\frac{m}{M_2} (\bm W_t^\top \bm S\bm W_t),
		\end{equation*}
		then Grönwall's inequality generates the exponential decay:
		\begin{equation*}
			\bm W_t^\top \bm S \bm W_t \Prec e^{-\frac{m}{M_2}t} \bm W_0^\top \bm S \bm W_0 \Prec 2 e^{-\frac{m}{M_2}t} \bm I_d,
		\end{equation*}
		yielding the desired result.
	\end{proof}
	
Next, we derive the coupling contractivity for the underdamped Langevin dynamics \eqref{ULD}.
\begin{lemma}\label{lemma: contractivity}
	Under Assumption~\ref{asP}, let $(\bm x_t,\bm v_t)_{t\Ge0}$ and $(\bm x_t',\bm v_t')_{t\Ge0}$ be the strong solutions to \eqref{ULD} driven by the same Brownian motion $(\bm B_t)_{t\Ge0}$. Then we have almost surely
	\begin{equation*}
		|\bm x_t'-\bm x_t| \Le \sqrt{2} e^{-\frac{mt}{2M_2}} \big(|\bm x_0'-\bm x_0| + |\bm v_0'-\bm v_0|\big), \quad \forall t\Ge0.
	\end{equation*}
\end{lemma}
\begin{proof}
	Lemma~\ref{lemma: first variation 1} has established the matrix inequality
	\begin{equation}
		\bm Q_t^\top \bm Q_t \Prec \bm W_t^\top \bm S \bm W_t \Prec 2e^{-\frac{mt}{M_2}} \bm I_d, \quad \forall t\Ge0,
		\label{proof: QQ decay}
	\end{equation}
	implying $\|\bm Q_t\| \Le \sqrt{2} e^{-\frac{mt}{2M_2}}$ for all $t\Ge0$. This consequently establishes
	\begin{equation}
		\max\big\{\|D_{\bm x}\bm x_t\|,\|D_{\bm v}\bm x_t\|\big\} \Le \sqrt{2} e^{-\frac{mt}{2M_2}}, \quad \forall t\Ge0.
		\label{proof: D decay}
	\end{equation}
	Consider coupled solutions $(\bm x_t,\bm v_t)$ and $(\bm x_t',\bm v_t')$ driven by the same Brownian motion $(\bm B_t)_{t\Ge0}$. For $s\in[0,1]$, define their interpolation as
	\begin{equation*}
		(\bm x_0^s,\bm v_0^s) = (1-s)(\bm x_0,\bm v_0) + s(\bm x_0',\bm v_0')= (\bm x_0,\bm v_0) + s(\bm x_0'-\bm x_0,\bm v_0'-\bm v_0),
	\end{equation*}
	and let the process $(\bm x_t^s,\bm v_t^s)_{t\Ge0}$ solve the underdamped Langevin dynamics \eqref{ULD} with the initial state $(\bm x_0^s,\bm v_0^s)$. 
	Through the chain rule differentiation of $\bm x_t^s$ with respect to the parameter $s$, we derive the sensitivity equation
	\begin{equation*}
		\frac{\D}{\D s} \bm x_t^s = D_{\bm x}\bm x_t^s (\bm x_0'-\bm x_0) + D_{\bm v}\bm x_t^s (\bm v_0'-\bm v_0).
	\end{equation*}
	Integration over $s\in[0,1]$ provides the error decomposition
	\begin{equation}
		\bm x_t' - \bm x_t = \int_0^1 \Big(D_{\bm x}\bm x_t^s (\bm x_0'-\bm x_0) + D_{\bm v} \bm x_t^s (\bm v_0'-\bm v_0)\Big)\D s.
		\label{proof: x'-x}
	\end{equation}
	Finally, applying \eqref{proof: D decay} and \eqref{proof: x'-x} generates the contractivity bound:
	\begin{equation*}
		|\bm x_t'-\bm x_t| \Le \sqrt{2} e^{-\frac{mt}{2M_2}} \int_0^1 \big(|\bm x_0'-\bm x_0| + |\bm v_0'-\bm v_0|\big) \D s = \sqrt{2} e^{-\frac{mt}{2M_2}} \big(|\bm x_0'-\bm x_0| + |\bm v_0'-\bm v_0|\big),
	\end{equation*}
	which completes the proof of Lemma~\ref{lemma: contractivity}.
\end{proof}

\reviewed{The companion bound on $|\bm v_t'-\bm v_t|$ is not needed. The coupling enters the subsequent estimates only through the differences of $\nabla^2 U$ and $\nabla f$ evaluated at $\bm x_t$ and $\bm x_t'$, which the position bound above already controls. An analogous exponential bound for $|\bm v_t'-\bm v_t|$ can be derived, and we omit it in Lemma~\ref{lemma: contractivity} for convenience.}

While the exponential decay of the first variation process $\bm W_t$ is sufficient for bounding the gradient of $u(\bm x,\bm v,t)$, estimating its Hessian requires analyzing the second-order sensitivity of the underdamped Langevin dynamics \eqref{ULD}.

\subsection{First variation process: Proof of coupling contractivity}
We now prove the contractivity of the first variation process $\bm W_t$ with respect to perturbations in the initial state $(\bm x_0,\bm v_0)$. This property provides the necessary second-order information for analyzing the underdamped Langevin dynamics \eqref{ULD}.

\begin{lemma}\label{lemma: first variation 2}
	Under Assumption~\ref{asP} with the initial configuration
	\begin{equation*}
		(\bm Q_0,\bm P_0)=(\bm I_d,\bm O_d)~\text{or}~(\bm O_d,\bm I_d),
	\end{equation*}
	for the first variation processes
	\begin{equation*}
		\begin{aligned}
			\bm W_t &= \mathrm{FirstVariation}\big(\bm Q_0,\bm P_0,\bm x_0,\bm v_0,(\bm B_t)_{t\Ge0}\big), \\
			\bm W_t' &= \mathrm{FirstVariation}\big(\bm Q_0,\bm P_0,\bm x_0',\bm v_0',(\bm B_t)_{t\Ge0}\big),
		\end{aligned}
	\end{equation*}
	their discrepancy satisfies almost surely the contractive bound
	\begin{equation*}
		(\bm W_t'-\bm W_t)^\top \bm S (\bm W_t'-\bm W_t) \Prec \frac{32M_3^2}{m^2} e^{-\frac{mt}{4M_2}} \big(|\bm x_0'-\bm x_0| + |\bm v_0'-\bm v_0|\big)^2 \bm I_d, \quad \forall t\Ge0.
	\end{equation*}
\end{lemma}
\begin{proof} 
	Building upon the analytical framework established in Lemma~\ref{lemma: first variation 1}, define
	\begin{equation*}
		\bm F_t = (\bm W_t' - \bm W_t)^\top \bm S (\bm W_t' - \bm W_t) \in \mathbb R^{d\times d}.
	\end{equation*}
	Computing the time derivative of $\bm F_t$ with respect to $t$ yields the evolution equation:
	\begin{align}
		\frac{\D}{\D t} \bm F_t & = -(\bm W_t'-\bm W_t)^\top \bm S (\bm A_t' \bm W_t'-\bm A_t \bm W_t) + \text{transpose} \notag \\
		& = -(\bm W_t'-\bm W_t)^\top \bm S \bm A_t(\bm W_t' - \bm W_t) 
		- (\bm W_t'-\bm W_t)^\top \bm S (\bm A_t'-\bm A_t) \bm W_t' + \text{transpose} \notag \\
		& = -\bm F_t^{(1)} - \bm F_t^{(2)},
		\label{proof: F12}
	\end{align}
	with constituent matrices defined as
	\begin{align*}
		\bm F_t^{(1)} & = (\bm W_t'-\bm W_t)^\top (\bm A_t^\top \bm S+\bm S \bm A_t)(\bm W_t' - \bm W_t), \\
		\bm F_t^{(2)} & = (\bm W_t'-\bm W_t)^\top \bm S (\bm A_t'-\bm A_t) \bm W_t' + \text{transpose}.
	\end{align*}
	
	Applying the matrix inequality \eqref{proof: S inequality} directly to $\bm F_t^{(1)}$ yields the coercive bound
	\begin{equation}
		\bm F_t^{(1)} \Succ \frac{m}{M_2} (\bm W_t'-\bm W_t)^\top \bm S (\bm W_t'-\bm W_t) = \frac{m}{M_2} \bm F_t.
		\label{proof: F1 estimate}
	\end{equation}
	For $\bm F_t^{(2)}$, we first decompose it using the expressions of $\bm A_t$ and $\bm A_t'$ in \eqref{HA expression},
	\begin{equation*}
		\bm F_t^{(2)} = (\bm Q_t'-\bm Q_t + \bm P_t'-\bm P_t)^\top (\bm H_t' -\bm H_t) \bm Q_t' + \text{transpose}.
	\end{equation*}
	Applying Young's inequality with parameter $k>0$ generates
	\begin{equation}
		\bm F_t^{(2)} \Succ -\|\bm H_t'-\bm H_t\| \bigg(k (\bm Q_t'-\bm Q_t + \bm P_t'-\bm P_t)^\top (\bm Q_t'-\bm Q_t + \bm P_t'-\bm P_t) + \frac{1}{k} (\bm Q_t')^\top \bm Q_t'\bigg),
		\label{proof: F2 estimate}
	\end{equation}
	where Lemma~\ref{lemma: contractivity} provides the distance bound of $\bm H_t$ and $\bm H_t'$:
	\begin{equation*}
		\|\bm H_t'-\bm H_t\| \Le \frac{M_3}{M_2} |\bm x_t'-\bm x_t| \Le \frac{\sqrt{2} M_3}{M_2} e^{-\frac{mt}{2M_2}} \big(|\bm x_0'-\bm x_0| + |\bm v_0'-\bm v_0|\big).
	\end{equation*}
	Furthermore, the quadratic term in \eqref{proof: F2 estimate} admits the upper bound
	\begin{align*}
		(\bm Q_t'-\bm Q_t + \bm P_t'-\bm P_t)^\top (\bm Q_t'-\bm Q_t + \bm P_t'-\bm P_t) & \Prec (\bm W_t'-\bm W_t)^\top \bm S (\bm W_t'-\bm W_t) = \bm F_t,
	\end{align*}
	and \eqref{proof: QQ decay} implies the inequality
	\begin{equation*}
		(\bm Q_t')^\top \bm Q_t' \Prec 2 e^{-\frac{mt}{M_2}} \bm I_d.
	\end{equation*}
	Substituting the estimates above into \eqref{proof: F2 estimate} yields the estimate of $\bm F_t^{(2)}$:
	\begin{equation}
		\bm F_t^{(2)} \Succ -\frac{\sqrt{2} M_3}{M_2} e^{-\frac{mt}{2M_2}} \big(|\bm x_0'-\bm x_0| + |\bm v_0'-\bm v_0|\big) \bigg(k\bm F_t + \frac{2}{k} e^{-\frac{mt}{M_2}} \bm I_d\bigg).
		\label{proof: F2 estimate 2}
	\end{equation}
	
	We synthesize the estimates of $\bm F_t^{(1)}$ and $\bm F_t^{(2)}$ from \eqref{proof: F1 estimate} and \eqref{proof: F2 estimate 2} to derive
	\begin{equation}
		\frac{\D}{\D t} \bm F_t \Prec -\frac{m}{M_2} \bm F_t + \frac{\sqrt{2} M_3}{M_2} e^{-\frac{mt}{2M_2}} \delta \bigg(k\bm F_t + \frac{2}{k} e^{-\frac{mt}{M_2}} \bm I_d\bigg),
		\label{proof: F inequality}
	\end{equation}
	where $\delta = |\bm x_0'-\bm x_0| + |\bm v_0'-\bm v_0|$ denotes the initial state discrepancy parameter. Given $\delta>0$, we choose $k$ to be moderately small to ensure $\bm F_t$ in \eqref{proof: F inequality} has a negative coefficient:
	\begin{equation*}
		\frac{m}{2M_2} = \frac{\sqrt{2} M_3}{M_2} e^{-\frac{mt}{2M_2}} \delta k ~\Longrightarrow~ k = \frac{m}{2\sqrt{2}M_3\delta} e^{\frac{mt}{2M_2}}.
	\end{equation*}
	Substituting this $k$ value into \eqref{proof: F inequality} yields the refined differential inequality
	\begin{equation*}
		\frac{\D}{\D t} \bm F_t \Prec -\frac{m}{2M_2} \bm F_t + \frac{8(M_3\delta)^2}{M_2 m} e^{-\frac{2mt}{M_2}} \bm I_d, \qquad \forall t\Ge0.
	\end{equation*}
	
	To establish the exponential decay of $\bm F_t$, we introduce the exponentially weighted matrix $\bm G_t = e^{\frac{mt}{4M_2}} \bm F_t$, which satisfies the evolution inequality
	\begin{align*}
		\frac{\D}{\D t} \bm G_t & = \frac{m}{4M_2} e^{\frac{mt}{4M_2}} \bm F_t + e^{\frac{mt}{4M_2}} \frac{\D}{\D t}\bm F_t \\
		& \Prec e^{\frac{mt}{4M_2}} \bigg(-\frac{m}{4M_2} \bm F_t + \frac{8(M_3\delta)^2}{M_2 m} e^{-\frac{2mt}{M_2}} \bm I_d\bigg) \\
		& \Prec -\frac{m}{4M_2} \bm G_t + \frac{8(M_3\delta)^2}{M_2 m} \bm I_d.
	\end{align*}
	\reviewed{Since the two processes start from the same initial data, $\bm W_0' = \bm W_0$ and hence $\bm F_0 = \bm G_0 = \bm O_d$.} This Grönwall's inequality thus generates the upper bound
	\begin{equation*}
		\bm G_t \Prec \frac{32 M_3^2}{m^2} \delta^2 \bm I_d ~\Longrightarrow~
		\bm F_t \Prec \frac{32 M_3^2}{m^2} e^{-\frac{mt}{4M_2}} \delta^2 \bm I_d,
	\end{equation*}
	thereby completing the proof of Lemma~\ref{lemma: first variation 2}.
\end{proof}

Applying Lemma~\ref{lemma: first variation 2}, we can now provide a quantitative estimate of the gradient and Hessian for the Kolmogorov solution $u(\bm x,\bm v,t)$ defined in \eqref{function u}.
\begin{theorem}
	\label{theorem: u estimate}
	Under Assumptions~\ref{asP} and \ref{asT}, the Kolmogorov solution $u(\bm x,\bm v,t)$ satisfies
	\begin{align*}
		\revised{|u(\bm x,\bm v,t)|} \revised{{}\Le C e^{-\frac{mt}{2M_2}}
			\bigg(|\bm x|+|\bm v|+\sqrt{\frac dm}\bigg),} & \\
		\max\big\{|\nabla_{\bm x} u(\bm x,\bm v,t)|, 
		|\nabla_{\bm v} u(\bm x,\bm v,t)|\big\} & \Le C e^{-\frac{mt}{2M_2}}, \\
		\max\big\{
		\norm{\nabla_{\bm x\bm x}^2 u(\bm x,\bm v,t)},
		\norm{\nabla_{\bm x\bm v}^2 u(\bm x,\bm v,t)},
		\norm{\nabla_{\bm v\bm v}^2 u(\bm x,\bm v,t)}
		\big\} & \Le \frac{C}{m} e^{-\frac{mt}{8M_2}},
	\end{align*}
	where the constant $C$ depends only on $M_2,M_3$ and $(L_i)_{i=1}^2$.
\end{theorem}
\begin{proof}
	First applying Lemma~\ref{lemma: first variation 1} with the decay estimate \eqref{proof: D decay}, we establish the exponential decay of the first variation processes:
	\begin{equation*}
		\max\big\{\norm{D_{\bm x}\bm x_t},\norm{D_{\bm v}\bm x_t}\big\} \Le \sqrt{2} e^{-\frac{mt}{2M_2}},\qquad \forall t \Ge 0.
	\end{equation*}
	Combining this with the gradient representation $\nabla_{\bm x} u(\bm x_0,\bm v_0,t)$ in \eqref{gradient u expression}, we deduce
	\begin{equation*}
		|\nabla_{\bm x} u(\bm x_0,\bm v_0,t)| \Le 
		\mathbb{E}^{(\bm x_0,\bm v_0)} \big[ \norm{D_{\bm x}\bm x_t} |\nabla f(\bm x_t)| \big] \Le C e^{-\frac{mt}{2M_2}}.
	\end{equation*}
	\revised{The same argument bounds $\nabla_{\bm v} u(\bm x,\bm v,t)$: by \eqref{gradient u expression} it is represented through $D_{\bm v}\bm x_t$ in place of $D_{\bm x}\bm x_t$, and the decay estimate covers both. This gives} the \revised{gradient bound of} Theorem~\ref{theorem: u estimate}.
	
	To establish the Hessian estimates, consider two distinct initial states $(\bm x_0,\bm v_0)$ and $(\bm x_0',\bm v_0')$ and define the corresponding solutions $(\bm x_t,\bm v_t)_{t\Ge0}$ and $(\bm x_t',\bm v_t')_{t\Ge0}$ to the underdamped Langevin dynamics \eqref{ULD} driven by the same Brownian motion $(\bm B_t)_{t\Ge0}$. Then
	\begin{equation}
		\nabla_{\bm x} u(\bm x_0',\bm v_0',t) - \nabla_{\bm x} u(\bm x_0,\bm v_0,t) = \mathbb{E}\Big[ (D_{\bm x} \bm x_t')^\top \nabla f(\bm x_t') - (D_{\bm x} \bm x_t)^\top \nabla f(\bm x_t) \Big] = I_1 + I_2,
		\label{proof: grad_u difference}
	\end{equation}
	where the expectation is taken over $(\bm B_t)_{t\Ge0}$, and the constituent terms decompose as
	\begin{equation*}
		I_1 = \mathbb{E}\Big[(D_{\bm x} \bm x_t' - D_{\bm x} \bm x_t)^\top \nabla f(\bm x_t')\Big], \qquad
		I_2 = \mathbb{E}\Big[ (D_{\bm x} \bm x_t)^\top (\nabla f(\bm x_t') - \nabla f(\bm x_t)) \Big].
	\end{equation*}
	For $I_1$ estimation, Lemma~\ref{lemma: first variation 2} yields the exponential decay property
	\begin{equation*}
		\norm{D_{\bm x}\bm x_t'- D_{\bm x}\bm x_t} \Le \frac{C}{m} e^{-\frac{mt}{8M_2}} \big(|\bm x_0'-\bm x_0|+ |\bm v_0'-\bm v_0|\big),
	\end{equation*}
	which leads to the inequality
	\begin{equation}
		|I_1| \Le \mathbb{E} \Big[\norm{D_{\bm x}\bm x_t'- D_{\bm x}\bm x_t} |\nabla f(\bm x_t')|\Big] \Le \frac{C}{m} e^{-\frac{mt}{8M_2}} \big(|\bm x_0'-\bm x_0|+ |\bm v_0'-\bm v_0|\big).
		\label{proof: I1 estimate}
	\end{equation}
	The $I_2$ term is controlled through Lemma~\ref{lemma: contractivity} via the chain of inequalities
	\begin{align}
		|I_2| & \Le \mathbb{E} \Big[
		\norm{D_{\bm x} \bm x_t} |\nabla f(\bm x_t') - \nabla f(\bm x_t)| \Big] \Le C\cdot \mathbb{E} \Big[\norm{D_{\bm x}\bm x_t} |\bm x_t' -\bm x_t|\Big]  \notag \\
		& \Le  C \cdot e^{-\frac{mt}{2M_2}} \cdot e^{-\frac{mt}{2M_2}} \big(|\bm x_0'-\bm x_0|+ |\bm v_0'-\bm v_0|\big) \notag \\
		& \Le C e^{-\frac{mt}{M_2}} \big(|\bm x_0'-\bm x_0|+ |\bm v_0'-\bm v_0|\big).
		\label{proof: I2 estimate}
	\end{align}
	
	Synthesizing the estimates in \eqref{proof: grad_u difference}\eqref{proof: I1 estimate}\eqref{proof: I2 estimate}, we establish the composite bound
	\begin{equation*}
		\big|\nabla_{\bm x} u(\bm x_0',\bm v_0',t) - \nabla_{\bm x} u(\bm x_0,\bm v_0,t)\big| \Le \frac{C}{m} e^{-\frac{mt}{8M_2}} \big(|\bm x_0'-\bm x_0|+ |\bm v_0'-\bm v_0|\big).
	\end{equation*}
	The arbitrariness of initial states $(\bm x_0,\bm v_0)$ and $(\bm x_0',\bm v_0')$ implies the Hessian norm control:
	\begin{equation*}
		\norm{\nabla_{\bm x\bm x}^2 u(\bm x,\bm v,t)} \Le \frac{C}{m} e^{-\frac{mt}{8M_2}}.
	\end{equation*}
	\revised{The mixed derivative $\nabla_{\bm x\bm v}^2 u(\bm x,\bm v,t)$ and the velocity Hessian $\nabla_{\bm v\bm v}^2 u(\bm x,\bm v,t)$ follow from the same decomposition with $D_{\bm v}\bm x_t$ in place of $D_{\bm x}\bm x_t$, which \eqref{proof: D decay} bounds identically.} 

	\revised{It remains to bound $u(\bm x,\bm v,t)$ itself. Write
	\begin{equation*}
		\pi(\bm x,\bm v) = \pi(\bm x) \bigg(\frac{M_2}{2\pi}\bigg)^{\frac d2} e^{-\frac{M_2}2|\bm v|^2} \propto 
		\exp\bigg(-\frac{M_2}2|\bm v|^2 - U(\bm x)\bigg)
	\end{equation*}
	for the invariant density of \eqref{ULD} on $\mathbb R^{2d}$, whose position marginal is $\pi(\bm x)$. Then
	\begin{align*}
		|u(\bm x,\bm v,t)| & = \bigg|\int_{\mathbb R^{2d}} \big(u(\bm x,\bm v,t) - u(\bm x',\bm v',t)\big) \pi(\bm x',\bm v') \D\bm x'\D\bm v'\bigg| \\
		& \Le C e^{-\frac{mt}{2M_2}} \int_{\mathbb R^{2d}} \big(|\bm x-\bm x'|+|\bm v-\bm v'|\big) \pi(\bm x',\bm v') \D\bm x'\D\bm v' \\
		& \Le C e^{-\frac{mt}{2M_2}} \bigg(|\bm x|+|\bm v| + \int_{\mathbb R^{2d}} \big(|\bm x'|+|\bm v'|\big) \pi(\bm x',\bm v') \D\bm x'\D\bm v'\bigg) \\
		& \Le C e^{-\frac{mt}{2M_2}} \bigg(|\bm x|+|\bm v|+\sqrt{\frac dm}\bigg),
	\end{align*}
	where the last inequality uses $\int_{\mathbb R^{2d}}|\bm x|\pi(\bm x)\D\bm x\Le\sqrt{d/m}$.}
\end{proof}

\revised{Theorem~\ref{theorem: u estimate} shows that the series defining the discrete Poisson solution}
\begin{equation*}
	\revised{\phi_h(\bm x,\bm v) = h\sum_{k=0}^\infty u(\bm x,\bm v,kh)}
\end{equation*}
\revised{and its derivatives $\nabla\phi_h,\nabla^2\phi_h$ converge absolutely at every $(\bm x,\bm v)$. Summing the resulting geometric series under the condition $h\Le\frac14$, and using $\frac{h}{1-e^{-ah}}\Le\frac2a$ for $ah\Le1$ with $a=\frac m{2M_2}$ and $a=\frac m{8M_2}$ respectively, gives the bounds $\frac Cm$ and $\frac C{m^2}$ of Theorem~\ref{theorem: phi h estimate}, concluding its proof.}

\section{Uniform-in-time moment bound for SG-UBU}
\label{appendix: uniform-in-time bound}

For convenience, introduce the notations
\begin{equation*}
	\bm z = \begin{bmatrix}
		\bm x \\ \bm v
	\end{bmatrix} \in \mathbb R^{2d},\quad
	\bm Z_k = \begin{bmatrix}
		\bm X_k \\ \bm V_k
	\end{bmatrix} \in \mathbb R^{2d},\quad
	\bm S = \begin{bmatrix}
		2\bm I_d & \bm I_d \\
		\bm I_d & \bm I_d
	\end{bmatrix} \in \mathbb R^{2d\times 2d},
\end{equation*}
then the Lyapunov function $\mathcal V(\bm x,\bm v) = \big(2|\bm x|^2 + 2\bm x^\top \bm v + |\bm v|^2\big)^4$ can be compactly written as
\begin{equation}
	\mathcal V(\bm z) = (\bm z^\top \bm S\bm z)^4.
\end{equation}
\subsection{Lyapunov condition for FG-BU integrator}
To begin with, we verify the Lyapunov condition for the FG-BU integrator.
\begin{lemma}
	\label{lemma: FG-BU Lyapunov}
	Under Assumption~\ref{asP}, let the step size $h\Le \frac14$; for any $\bm z_0\in\mathbb R^{2d}$,
	\begin{equation*}
		\mathbb E\Big[\mathcal V\big((\Phi_h^{\mathfrak U} \circ \Phi_h^{\mathfrak B}\big) \bm z_0\big)\Big] \Le \bigg(1-\frac{mh}{2M_2}\bigg) \mathcal V(\bm z_0) + \frac{Cd^4h}{m^3},
	\end{equation*}
	where the constant $C$ depends only on $M_1,M_2$.
\end{lemma}
\begin{proof}
From the expressions of the flow solutions $\Phi_h^{\mathfrak U}$ and $\Phi_h^{\mathfrak B}$ in Section~\ref{subsection: uld SG-UBU}, the one-step FG-BU update $\bar{\bm z}_1 = (\bar{\bm x}_1,\bar{\bm v}_1) \in \mathbb R^{2d}$ can be explicitly represented as 
\begin{equation*}
	\left\{
	\begin{aligned}
		\bar{\bm x}_1 & = \bm x_0 + \frac{1-e^{-2h}}2 \bm v_0 - \frac{h(1-e^{-2h})}{2M_2} \nabla U(\bm x_0) + \frac{2}{\sqrt{M_2}}
		\int_0^h \frac{1-e^{-2(h-s)}}2 \D \bm B_s, \\
		\bar{\bm v}_1 & = e^{-2h} \bm v_0 - \frac{he^{-2h}}{M_2} \nabla U(\bm x_0)
		+ \frac2{\sqrt{M_2}} \int_0^h
		e^{-2(h-s)}\D \bm B_s.
	\end{aligned}
	\right.
\end{equation*}
Using $\nabla U(\bm 0) = \bm 0$ and the mean value theorem, it holds
\begin{equation*}
	\frac1{M_2}\nabla U(\bm x_0) = \bm H \bm x_0,~~~~\text{where}~
	\bm H = \frac1{M_2}\int_0^1 \nabla^2 U(s\bm x_0)\D s \in \mathbb R^{d\times d}
\end{equation*}
is positive definite with $\frac{m}{M_2} \bm I_d \Prec \bm H \Prec \bm I_d$. Thus the FG-BU update $\bar{\bm z}_1 = (\bar{\bm x}_1,\bar{\bm v}_1)$ reads
\begin{equation}
	\bar{\bm z}_1 = \bm A \bm z_0 + \sqrt{h}\bm w,
	\label{proof: z1 expression}
\end{equation}
where we employ the notations
\begin{equation*}
	\bm A = \begin{bmatrix}
		\bm I_d - \dfrac{h(1-e^{-2h})}{2} \bm H & \dfrac{1-e^{-2h}}{2} \bm I_d \vspace{6pt} \\
		-h e^{-2h} \bm H & e^{-2h} \bm I_d
	\end{bmatrix} ,\quad
	\bm w = \frac2{\sqrt{M_2 h}}\begin{bmatrix}
		\di\int_0^h \dfrac{1-e^{-2(h-s)}}2 \D \bm B_s \\
		\di\int_0^h e^{-2(h-s)}\D \bm B_s
	\end{bmatrix}.
\end{equation*}
It is clear that $\bm w$ is a Gaussian random variable with
\begin{equation*}
	\mathbb E[\bm w] = \bm 0, \qquad
	\mathbb E\big[|\bm w|^8\big] \Le Cd^4,
\end{equation*}
where the constant $C$ depends only on $M_1,M_2$.

Next, we prove that for any step size $h\Le \frac14$, the following matrix inequality holds:
\begin{equation}
	\bm A^\top \bm S \bm A \Prec \bigg(1-\frac{mh}{2M_2}\bigg) \bm S.
	\label{proof: ASA contractivity}
\end{equation}
The inequality \eqref{proof: ASA contractivity} serves as the coercivity condition for the FG-BU integrator. Direct calculation yields the upper bound of $\bm A^\top \bm S \bm A$:
\begin{align*}
	\bm A^\top \bm S \bm A & = 
	\begin{bmatrix}
		2\bm I_d-2h\bm H + \dfrac{1+e^{-4h}}2 h^2 \bm H^2 & \bm I_d - \dfrac{1+e^{-4h}}2 h\bm H 
		\vspace{6pt} \\
		\bm I_d - \dfrac{1+e^{-4h}}2 h\bm H & \dfrac{1+e^{-4h}}2  \bm I_d
	\end{bmatrix} \\
	& \Prec \begin{bmatrix}
		2\bm I_d - \dfrac74 h\bm H & \bm I_d - \dfrac{1+e^{-4h}}2 h\bm H \vspace{6pt} \\
		\bm I_d - \dfrac{1+e^{-4h}}2 h\bm H & \dfrac{1+e^{-4h}}2 \bm I_d
	\end{bmatrix}. \tag{using $h\Le\frac14$ and $\bm H\Prec \bm I_d$}
\end{align*}
Thus using the expression of $\bm S \in \mathbb R^{2d\times 2d}$, there holds
\begin{align*}
	\bm S - \bm A^\top \bm S \bm A & \Succ
	\begin{bmatrix}
		\dfrac74 h \bm H & \dfrac{1+e^{-4h}}2 h\bm H \vspace{6pt} \\
		\dfrac{1+e^{-4h}}2 h\bm H & \dfrac{1-e^{-4h}}2 \bm I_d
	\end{bmatrix} \\
	& \Succ
	\begin{bmatrix}
		\dfrac74 h \bm H & \dfrac{1+e^{-4h}}2 h\bm H \vspace{6pt} \\
		\dfrac{1+e^{-4h}}2 h\bm H & h \bm H
	\end{bmatrix} \tag{using $\dfrac{1-e^{-4h}}2 \Ge h$ and $\bm I_d \Succ \bm H$} \\
	& \Succ \dfrac{mh}{M_2} \begin{bmatrix}
		\dfrac74 \bm I_d & \dfrac{1+e^{-4h}}2 \bm I_d \vspace{6pt} \\
		\dfrac{1+e^{-4h}}2 \bm I_d & \bm I_d
	\end{bmatrix} \Succ \dfrac{mh}{2M_2} \bm S \tag{using $\bm H \Succ \dfrac{m}{M_2}I_d$},
\end{align*}
where the last inequality follows from
\begin{equation*}
	\begin{bmatrix}
		\dfrac74 & \dfrac{1+e^{-4h}}2 \vspace{4pt}\\
		\dfrac{1+e^{-4h}}2 & 1
	\end{bmatrix} \Succ \frac12 \begin{bmatrix}
		2 & 1 \\
		1 & 1
	\end{bmatrix}~ \Longleftrightarrow~ \begin{bmatrix}
		\dfrac32 & e^{-4h} \vspace{4pt} \\
		e^{-4h} & 1
	\end{bmatrix} \Succ 0~ \Longleftrightarrow~ \frac32 - e^{-8h}\Ge0.
\end{equation*}
This completes the proof of inequality \eqref{proof: ASA contractivity}.

The positive definite matrix $\bm S$ admits a unique symmetric square root
\begin{equation*}
	\sqrt{\bm S} =
	\frac1{\sqrt{5}} \begin{bmatrix}
		3 \bm I_d & \bm I_d \\
		\bm I_d & 2\bm I_d
	\end{bmatrix} \in \mathbb R^{2d\times 2d}.
\end{equation*}
This gives the representation $\bm z^\top \bm S \bm z = \big|\sqrt{\bm S}\bm z\big|^2$ for $\bm z\in\mathbb R^{2d}$. Using the identity $\bar{\bm z}_1 = \bm A \bm z_0 + \sqrt{h}\bm w$, we obtain the moment bound for the one-step FG-BU update $\bm z_1$:
\begin{align*}
	\mathbb E\big|\sqrt{\bm S} \bar{\bm z}_1\big|^8 & =
	\mathbb E\big|\sqrt{\bm S} \bm A\bm z_0 + \sqrt{h \bm S} \bm w\big|^8 \\
	& \Le \big|\sqrt{\bm S} \bm A\bm z_0\big|^8 + 28
	\big|\sqrt{\bm S} \bm A \bm z_0\big|^6 h\cdot \mathbb E\big|\sqrt{\bm S}\bm w\big|^2 + 70\big|\sqrt{\bm S} \bm A\bm z_0\big|^4 h^2 \cdot \mathbb E\big|\sqrt{\bm S}\bm w\big|^4  \\
	& ~~~~+ 28
	\big|\sqrt{\bm S} \reviewed{\bm A}\bm z_0\big|^2 h^3\cdot \mathbb E\big|\sqrt{\bm S}\bm w\big|^6 + h^4\cdot \mathbb E\big|\sqrt{\bm S}\bm w\big|^8 \\
	& \Le \big|\sqrt{\bm S} \bm A\bm z_0\big|^8 + Cdh
	\big|\sqrt{\bm S} \bm A\bm z_0\big|^6 + Cd^2h^2\big|\sqrt{\bm S} \bm A\bm z_0\big|^4 +
	Cd^3h^3\big|\sqrt{\bm S} \bm A\bm z_0\big|^2 +
	Cd^4h^4 \\
	& \Le \big|\sqrt{\bm S} \bm A\bm z_0\big|^8 + \frac{mh}{2M_2} \big|\sqrt{\bm S} \bm A\bm z_0\big|^8 +
	\frac{Cd^4h}{m^3},
\end{align*}
where we have utilized the binomial theorem, the vanishing of odd powers of $\bm w$ in the expansion, \revised{and the Cauchy--Schwarz inequality to bound the surviving inner products by $\big|\sqrt{\bm S}\bm A\bm z_0\big|\big|\sqrt{\bm S}\bm w\big|$}. Then the contractivity inequality \eqref{proof: ASA contractivity} yields the moment estimate
\begin{align*}
	\mathbb E\big[(\bar{\bm z}_1^\top \bm S \bar{\bm z}_1)^4\big] & \Le 
	\bigg(1+\frac{mh}{2M_2}\bigg) \mathbb E
	\big[(\bm z_0^\top \bm A^\top\bm S \bm A\bm z_0)^4\big] + \frac{Cd^4h}{m^3} \\
	& \Le
	\bigg(1+\frac{mh}{2M_2}\bigg)
	\bigg(1-\frac{mh}{2M_2}\bigg)^4
	(\bm z_0^\top \bm S \bm z_0)^4 + \frac{Cd^4h}{m^3} \\
	& \Le \bigg(1-\frac{mh}{2M_2}\bigg) (\bm z_0^\top \bm S \bm z_0)^4 + \frac{Cd^4h}{m^3},
\end{align*}
concluding the proof of the Lyapunov condition for the FG-BU integrator.
\end{proof}

Notably, Lemma~\ref{lemma: FG-BU Lyapunov} establishes the Lyapunov condition directly for the FG-BU integrator. While a corresponding condition can be derived for the continuous underdamped Langevin dynamics \eqref{ULD}, our analytical framework bypasses its continuous counterpart, operating entirely at the discrete level.

\subsection{Lyapunov condition for SG-BU integrator}
Based on Lemma~\ref{lemma: FG-BU Lyapunov} and the fact that SG-BU employs an unbiased stochastic gradient, we can conveniently establish the Lyapunov condition for the SG-BU integrator. This approach also applies to variance-reduced stochastic gradient sampling algorithms, as the unbiasedness property still holds (see Lemmas~\ref{lemma: SVRG-BU Lyapunov} and \ref{lemma: SAGA-BU Lyapunov}).

\begin{proofof}[Lemma~\ref{lemma: SG-BU Lyapunov}]
Following the proof of Lemma~\ref{lemma: FG-BU Lyapunov}, we denote $\bar{\bm z}_1 = (\Phi^{\mathfrak U}_h \circ \Phi^{\mathfrak B}_h)\bm z_0$ and $\bm z_1 = (\Phi^{\mathfrak U}_h \circ \Phi^{\mathfrak B}_h(\theta))\bm z_0$ as the one-step updates of FG-BU and SG-BU. Then
\begin{equation*}  
	\left\{  
	\begin{aligned}  
		\bm x_1 & = \bm x_0 + \frac{1-e^{-2h}}2 \bm v_0 - \frac{h(1-e^{-2h})}{2M_2} b(\bm x_0,\theta) + \frac{2}{\sqrt{M_2}}  
		\int_0^h \frac{1-e^{-2(h-s)}}2 \D \bm B_s, \\  
		\bm v_1 & = e^{-2h} \bm v_0 - \frac{he^{-2h}}{M_2} b(\bm x_0,\theta)  
		+ \frac2{\sqrt{M_2}} \int_0^h  
		e^{-2(h-s)}\D \bm B_s.  
	\end{aligned}  
	\right.  
\end{equation*}
Thus we can write the stochastic gradient error as
\begin{equation*}  
	\bm x_1 - \bar{\bm x}_1 = h \bm\Delta_{\bm x},
    ~~~~
	\bm v_1 - \bar{\bm v}_1 = h \bm\Delta_{\bm v},  
\end{equation*}  
where the error terms are given by
\begin{align*}  
	\bm\Delta_{\bm x} = -\frac{1-e^{-2h}}{2M_2} (b(\bm x_0,\theta) - \nabla U(\bm x_0)), ~~~~  
	\bm\Delta_{\bm v} = -\frac{e^{-2h}}{M_2} (b(\bm x_0,\theta) - \nabla U(\bm x_0)).  
\end{align*}  
\revised{The proof below uses only the unbiasedness and the eighth moment bound of $b(\bm x,\theta)-\nabla U(\bm x)$ in Assumption~\ref{asSg}, not the almost-sure growth bound $|b(\bm x,\theta)| \Le M_1\sqrt{|\bm x|^2+d}$.} From Assumption~\ref{asSg} we obtain the unbiased condition and the moment bound
\begin{equation*}  
	\mathbb E[\bm\Delta_{\bm x}] = \mathbb E[\bm\Delta_{\bm v}] = \bm 0,\quad 
	\mathbb E\big[|\bm\Delta_{\bm x}|^8\big]
	 \Le Cd^4h^8,\quad\mathbb E\big[|\bm\Delta_{\bm v}|^8\big] \Le Cd^4.  
\end{equation*}  
Define $\bm\Delta_{\bm z} = (\bm\Delta_{\bm x},\bm\Delta_{\bm v})\in\mathbb R^{2d}$, giving $\bm z_1 = \bar{\bm z}_1 + h\bm\Delta_{\bm z}$ with  
\begin{equation*}  
	\mathbb E[\bm\Delta_{\bm z}] = \bm 0, \quad
	\mathbb E\big[|\bm\Delta_{\bm z}|^8\big] \Le Cd^4.  
\end{equation*} 
Exploiting the Lyapunov condition of the FG-BU integrator in Lemma~\ref{lemma: FG-BU Lyapunov}, we have
\begin{align*}  
	\mathbb E\big|\sqrt{\bm S} \bm z_1\big|^8 & =  
	\mathbb E\big|\sqrt{\bm S} \bar{\bm z}_1 + h \sqrt{\bm S} \bm\Delta_{\bm z}\big|^8 \\  
	& \Le \mathbb E\big|\sqrt{\bm S} \bar{\bm z}_1\big|^8 +
	28h^2\mathbb E\Big(\big|\sqrt{\bm S} \bar{\bm z}_1\big|^6\big|\sqrt{\bm S}\bm \Delta_{\bm z}\big|^2\Big) \\
	& \hspace{1cm} + \revised{56h^3\mathbb E\Big(\big|\sqrt{\bm S} \bar{\bm z}_1\big|^5\big|\sqrt{\bm S}\bm \Delta_{\bm z}\big|^3\Big)} +
	70h^4\mathbb E\Big(\big|\sqrt{\bm S} \bar{\bm z}_1\big|^4\big|\sqrt{\bm S}\bm \Delta_{\bm z}\big|^4\Big) \\
	& \hspace{1cm} + \revised{56h^5\mathbb E\Big(\big|\sqrt{\bm S} \bar{\bm z}_1\big|^3\big|\sqrt{\bm S}\bm \Delta_{\bm z}\big|^5\Big)} +
	28h^6\mathbb E\Big(\big|\sqrt{\bm S} \bar{\bm z}_1\big|^2\big|\sqrt{\bm S}\bm \Delta_{\bm z}\big|^6\Big) \\
	& \hspace{1cm} + \revised{8h^7\mathbb E\Big(\big|\sqrt{\bm S} \bar{\bm z}_1\big|\big|\sqrt{\bm S}\bm \Delta_{\bm z}\big|^7\Big)} +
	h^8 \mathbb E\big|\sqrt{\bm S} \bm \Delta_{\bm z}\big|^8 \\
	& \Le \mathbb E\big|\sqrt{\bm S} \bar{\bm z}_1\big|^8 + \frac{mh}{12M_2}  
	\mathbb E\big|\sqrt{\bm S} \bar{\bm z}_1\big|^8 + \frac{Ch}{m^3} \mathbb E\big[|\bm\Delta_{\bm z}|^8\big] \\  
	& \Le \bigg(1+\frac{mh}{12M_2}\bigg)\bigg(  
	\bigg(1-\frac{mh}{2M_2}\bigg)  
	\mathbb E\big|\sqrt{\bm S} \bm z_0\big|^8 + \frac{Cd^4h}{m^3} \bigg) + \frac{Ch}{m^3} \mathbb E\big[|\bm\Delta_{\bm z}|^8\big]\\  
	& \Le \bigg(1-\frac{mh}{3M_2}\bigg)  
	\mathbb E\big|\sqrt{\bm S} \bm z_0\big|^8 +  
	\frac{Cd^4h}{m^3}.  
\end{align*}  
\revised{The term linear in $\bm\Delta_{\bm z}$ is dropped from the binomial expansion because $\mathbb E[\bm\Delta_{\bm z}]=\bm 0$, and the remaining terms are absorbed into the first and last ones by Young's inequality.}
Recalling the Lyapunov function $\mathcal{V}(\bm z) = (\bm z^\top \bm S \bm z)^4$, the inequality above implies
\begin{equation*}  
	\mathbb E\big[\mathcal{V}(\bm z_1)\big]  
	\Le \bigg(1-\frac{mh}{3M_2}\bigg)  
	\mathcal{V}(\bm z_0) +  
	\frac{Cd^4h}{m^3},  
\end{equation*} 
which completes the proof of Lemma~\ref{lemma: SG-BU Lyapunov}.
\end{proofof}

\revised{As noted at the start of its proof, Lemma~\ref{lemma: SG-BU Lyapunov} applies to any unbiased stochastic gradient whose error satisfies the eighth moment bound of Assumption~\ref{asSg}, and it is in this form that it is invoked for the variance-reduced integrators.}

\subsection{Uniform-in-time moment bound of SG-UBU integrator}
Next, it is convenient to derive the uniform-in-time moment bound for the SG-UBU integrator by simply applying an additional $\Phi_{h/2}^{\mathfrak U}$ flow on the SG-BU integrator.
\begin{proofof}[Theorem~\ref{theorem: SG-UBU moment bound}]
	For the step size $h\Le \frac14$, the following bounds hold for any $\bm z_0\in\mathbb R^{2d}$:
	\begin{equation*}
		\mathbb E\big[\mathcal V\big(\Phi^{\mathfrak U}_{h/2}\bm z_0\big)\big] \Le C
		\mathcal V(\bm z_0) + Cd^4,~~~~
		\mathbb E\big[\mathcal V\big(\Phi^{\mathfrak B}_{h}(\theta)\bm z_0\big)\big] \Le C
		\mathcal V(\bm z_0) + Cd^4,
	\end{equation*}
	where $C$ depends only on $M_1,M_2$. 
	These inequalities can be easily found from the explicit expressions of the solution flows $\Phi_{h/2}^{\mathfrak U}$ and $\Phi_h^{\mathfrak B}(\theta)$. Then for $\bm X_0 = \bm 0$ and $\bm V_0 \sim \mathcal N(\bm 0,M_2^{-1}\bm I_d)$, we have the moment bound 
	\begin{equation*}
		\mathbb E\big[\mathcal V\big(\Phi^{\mathfrak U}_{h/2} \bm Z_0\big)\big] \Le \frac{Cd^4}{m^4}.
	\end{equation*}
	For $k\Ge 2$, the SG-UBU solution is written as
	\begin{equation*}
		\bm Z_k = \Phi^{\mathfrak U}_{h/2}\circ \Phi^{\mathfrak B}_h(\theta_{k-1})\circ \bigg(\prod_{l=0}^{k-2} \Phi^{\mathfrak U}_{h}\circ
		\Phi^{\mathfrak B}_h(\theta_l) \bigg) \circ \Phi^{\mathfrak U}_{h/2} \bm Z_0.
	\end{equation*}
	The Lyapunov condition of SG-BU provided in Lemma~\ref{lemma: SG-BU Lyapunov} thus yields the bound for $k\Ge2$, while $k=0$ and $k=1$ are covered by the two one-step inequalities above; hence
	\begin{equation*}
		\sup_{k\Ge0} \mathbb E\big[\mathcal V(\bm Z_k)\big]
		\Le \frac{Cd^4}{m^4},
	\end{equation*}
	which completes the proof of Theorem~\ref{theorem: SG-UBU moment bound}.
\end{proofof}

Finally, we employ Theorem~\ref{theorem: SG-UBU moment bound} to show the existence and moment bound for the invariant distribution of SG-UBU.
\begin{proofof}[Lemma~\ref{lemma: SG-UBU invariant}]
Employing the Krylov--Bogolyubov theorem (Theorem~4.21, \cite{hairer2006ergodic}), when the step size $h\Le \frac14$, the Lyapunov condition in Lemma~\ref{lemma: SG-BU Lyapunov} implies that SG-BU has an invariant distribution $\hat \pi_h(\bm x,\bm v)$ satisfying
\begin{equation}
	\int_{\mathbb R^{2d}} |\bm z|^8 \hat \pi_h(\bm z)\D\bm z \Le \frac{Cd^4}{m^4}.
	\label{proof: SG-BU bound}
\end{equation}
Define distribution semigroups for $\Phi^{\mathfrak U}_t$ and $\Phi^{\mathfrak B}_t(\theta)$:
\begin{align*}
	P_t^{\mathfrak U} \mu & := \text{distribution of } \Phi^{\mathfrak U}_t \bm z_0 ~~ \text{with} ~ \bm z_0\sim \mu, \\
	P_t^{\mathfrak B} \mu & := \mathbb E_{\theta}\Big[\text{distribution of } \Phi^{\mathfrak B}_t(\theta) \bm z_0 ~ \text{with} ~ \bm z_0 \sim \mu\Big].
\end{align*}
Then the invariant distribution $\hat \pi_h$ of SG-BU satisfies:
\begin{equation*}
	P_h^{\mathfrak U} P_h^{\mathfrak B} \hat \pi_h = \hat \pi_h.
\end{equation*}
Now define the distribution $\pi_h$ in $\mathbb R^{2d}$ as
\begin{equation*}
	\pi_h = P_{h/2}^{\mathfrak U} P_h^{\mathfrak B} \hat \pi_h.
\end{equation*}
This $\pi_h$ satisfies the same moment bound \eqref{proof: SG-BU bound}, and the identity
\begin{equation*}
	\pi_h = P_{h/2}^{\mathfrak U} P_h^{\mathfrak B} P_h^{\mathfrak U} P_h^{\mathfrak B} \hat \pi_h = P_{h/2}^{\mathfrak U} P_h^{\mathfrak B} P_{h/2}^{\mathfrak U} \pi_h,
\end{equation*}
confirming that $\pi_h$ is an invariant distribution of SG-UBU.
\end{proofof}
It should be noted that while Lemma~\ref{lemma: SG-UBU invariant} establishes the existence of an invariant distribution, it does not guarantee uniqueness, a property that is not required for the subsequent analysis.

\section{Local error analysis for SG-UBU}
\label{appendix: local error}

This section provides fourth moment estimates for the stochastic gradient error, $\bm{Z}_{k+1} - \bar{\bm{Z}}_{k+1}$, and the discretization error, $\bar{\bm{Z}}_{k+1} - \bm{Z}_k(h)$. The analysis requires comparing the explicit formulas for the one-step evolution of the exact solution, the full gradient update, and the stochastic gradient update, which we present below for clarity.

\paragraph{Exact solution in one-step: $\bm Z_k(h)$}
\begin{equation*}
	\small
	\left\{
	\begin{aligned}
		\bm X_k(h) & = \bm X_k + \frac{1-e^{-2h}}2 \bm V_k - \frac1{M_2}\int_0^h \frac{1-e^{-2(h-s)}}{2} \nabla U(\bm X_k(s))\D s \\
		& \hspace{1cm} + \frac2{\sqrt{M_2}}\int_0^h \frac{1-e^{-2(h-s)}}{2} \D\bm B_{kh+s}, \\
		\bm V_k(h) & = e^{-2h} \bm V_k - \frac1{M_2} \int_0^h e^{-2(h-s)}\nabla U(\bm X_k(s))\D s + \frac2{\sqrt{M_2}} \int_0^h e^{-2(h-s)} \D \bm B_{kh+s}.
	\end{aligned}
	\right.
\end{equation*}
\paragraph{One-step of FG-UBU update: $\bar{\bm Z}_{k+1}$}
\begin{equation*}
	\small
	\left\{
	\begin{aligned}
		\bar{\bm X}_{k+1} & = \bm X_k + \frac{1-e^{-2h}}2 \bm V_k - \frac{h(1-e^{-h})}{2M_2} \nabla U(\bm Y_k) +
		\frac2{\sqrt{M_2}}\int_0^h \frac{1-e^{-2(h-s)}}{2} \D\bm B_{kh+s}, \\
		\bar{\bm V}_{k+1} & = e^{-2h} \bm V_k - \frac{he^{-h}}{M_2} \nabla U(\bm Y_k) + \frac2{\sqrt{M_2}} \int_0^h e^{-2(h-s)} \D \bm B_{kh+s}.
	\end{aligned}
	\right.
\end{equation*}
\paragraph{One-step of SG-UBU update: $\bm Z_{k+1}$}
\begin{equation*}
	\small
	\left\{
	\begin{aligned}
		\bm X_{k+1} & = \bm X_k + \frac{1-e^{-2h}}2 \bm V_k - \frac{h(1-e^{-h})}{2M_2} b(\bm Y_k,\theta_k) +
		\frac2{\sqrt{M_2}}\int_0^h \frac{1-e^{-2(h-s)}}{2} \D\bm B_{kh+s}, \\
		\bm V_{k+1} & = e^{-2h} \bm V_k - \frac{he^{-h}}{M_2} b(\bm Y_k,\theta_k) + \frac2{\sqrt{M_2}} \int_0^h e^{-2(h-s)} \D \bm B_{kh+s}.
	\end{aligned}
	\right.
\end{equation*}
These expressions are derived by applying the $\Phi_t^{\mathfrak U}$ and $\Phi_t^{\mathfrak B}$ solution flows in the UBU integrator, and can also be found directly in Example~8 of \cite{sanz2021wasserstein}.
Recall that $\bm Y_k$ is the intermediate position defined in \eqref{Y_k expression}, and the stochastic gradient indices $(\theta_k)_{k=0}^\infty$ are generated independently from $\mathbb P_\theta$.

\subsection{Estimate of stochastic gradient error}
With the explicit expression for the stochastic gradient error established in \eqref{stochastic gradient error expression}, we now prove the corresponding moment bounds in Lemma~\ref{lemma: stochastic gradient error}.
\begin{proofof}[Lemma~\ref{lemma: stochastic gradient error}]
	Under Assumption~\ref{asSg}, the gradient expression \eqref{stochastic gradient error expression} gives
	\begin{equation*}
		\mathbb E\big|\bm X_{k+1} - \bar{\bm X}_{k+1}\big|^4 \Le Ch^8
		\mathbb E\big|b(\bm Y_k,\theta_k) - \nabla U(\bm Y_k)\big|^4 \Le
		C\sigma^4d^2h^8,
	\end{equation*}
	yielding the position stochastic gradient error bound
	\begin{equation*}
		\sqrt{\mathbb E\big|\bm X_{k+1} - \bar{\bm X}_{k+1}\big|^4}
		\Le C\sigma^2dh^4.
	\end{equation*}
	Similarly, the velocity error bound follows
	\begin{equation*}
		\sqrt{\mathbb E\big|\bm V_{k+1} - \bar{\bm V}_{k+1}\big|^4}
		\Le C\sigma^2dh^2,
	\end{equation*}
	which completes the proof.
\end{proofof}
Notably, the position and velocity components of the stochastic gradient error converge at different rates with respect to the step size $h$. This discrepancy in order is a direct consequence of their respective coefficients in the explicit form \eqref{stochastic gradient error expression}.

\subsection{Estimate of discretization error}
Our analysis of the fourth moment of the discretization error relies on the explicit expressions for the local errors $\Dx$ and $\Dv$ provided in Equations~(40) and~(42) of \cite{sanz2021wasserstein}. 
The proof relies on a delicate decomposition of the discretization error using integration by parts.
\begin{proofof}[Lemma~\ref{lemma: discretization error}]
	Utilizing the explicit expressions of $\bar{\bm Z}_{k+1}$ and $\bm Z_k(h)$, 
	the discretization errors $\Dx = \bar{\bm X}_{k+1} - \bm X_k(h)$ and $\Dv = \bar{\bm V}_{k+1} - \bm V_k(h)$ are written as
	\begin{subequations}
	\begin{align}
			\Dx & = \frac{h(1-e^{-h})}{2M_2}\big(\nabla U(\bm X_k(h/2))-\nabla U(\bm Y_k)\big) + \bm I_6 + \bm I_7, \\
			\Dv & = \frac{he^{\reviewed{-h}}}{M_2} \big(\nabla U(\bm X_k(h/2) ) - \nabla U(\bm Y_k)\big) + \bm I_1 + \bm I_2 + \bm I_3 + \bm I_4 + \bm I_5, 
	\end{align}
	\label{proof: XV expression}
	\end{subequations}
	with the error terms $\bm I_1$ to $\bm I_7$ given by
	\begin{align*}
		\bm I_1 & = \frac4{M_2}\int_{h/2}^h \D s \int_{h/2}^s \D s'
		\int_{h-s'}^{s'} e^{-2(h-s'')} \nabla U(\bm X_k(s''))\D s'', \\
		\bm I_2 & = \frac2{M_2}\int_{h/2}^h \D s \int_{h/2}^s \D s'
		\int_{h-s'}^{s'} e^{-2(h-s'')} \nabla^2 U(\bm X_k(s'')) \bm V_k(s'')\D s'', \\
		\bm I_3 & = \frac1{M_2}\int_{h/2}^h \D s \int_{h/2}^s \D s'
		\int_{h-s'}^{s'} e^{-2(h-s'')} \bm V_k(s'')^\top\nabla^3 U(\bm X_k(s'')) \bm V_k(s'')\D s'', \\
		\bm I_4 & = -\frac1{M_2^2}\int_{h/2}^h \D s \int_{h/2}^s \D s'
		\int_{h-s'}^{s'} e^{-2(h-s'')} \nabla^2 U(\bm X_k(s'')) \nabla U(\bm X_k(s''))\D s'',
	\end{align*}
	\begin{align*}
		\bm I_5 & = \frac2{M_2^{3/2}} \int_{h/2}^h \D s \int_{h/2}^s \D s'
		\int_{h-s'}^{s'} e^{-2(h-s'')}
		\reviewed{\nabla^2 U(\bm X_k(s''))}\D \bm B_{kh+s''}, \\
		\bm I_6 & = \frac1{M_2}\int_0^h \D s \int_{h/2}^s \frac{1-e^{-2(h-s')}}2 \nabla^2 U(\bm X_k(s')) \bm V_k(s') \D s', \\
		\bm I_7 & = -\frac1{M_2}\int_0^h \D s \int_{h/2}^s e^{-2(h-s')}
		\nabla U(\bm X_k(s'))\D s'.
	\end{align*}
	For convenience, we define the deviation of $\bm Y_k$ by
	\begin{equation*}
		\Dy = \bm Y_k - \bm X_k(h/2) = \frac1{M_2}\int_0^{h/2} \frac{1-e^{-(h-2s)}}2 \nabla U(\bm X_k(s))\D s,
	\end{equation*}
	then its norm $|\Dy|$ admits the upper bound
	\begin{equation*}
		|\Dy| \Le Ch\int_0^{h/2}|\nabla U(\bm X_k(s))|\D s.
	\end{equation*}
	From the uniform-in-time moment bound in Theorem~\ref{theorem: SG-UBU moment bound} and H\"older's inequality we derive
	\begin{equation*}
		\mathbb E|\Dy|^4 \Le Ch^7\int_0^{h/2} \mathbb E|\nabla U(\bm X_k(s))|^4 \D s \Le Ch^7 \int_0^{h/2} \mathbb E\big(|\bm X_k(s)|^2+d\big)^2 \D s \Le \frac{Cd^2h^8}{m^2}.
	\end{equation*}
	Applying the same approach, we establish fourth moment bounds of $\bm I_1$ to $\bm I_4$:
	\begin{align*}
		\mathbb E|\bm I_1|^4 & \Le Ch^9\int_{h/2}^h\D s
		\int_{h/2}^s \D s'
		\int_{h-s'}^{s'}
		\mathbb E|\nabla U(\bm X_k(s''))|^4 \D s'' \Le
		\frac{Cd^2 h^{12}}{m^2}, \\
		\mathbb E|\bm I_2|^4 & \Le Ch^9\int_{h/2}^h\D s
		\int_{h/2}^s \D s'
		\int_{h-s'}^{s'} \mathbb E|\bm V_k(s'')|^4\D s'' \Le
		\frac{Cd^2 h^{12}}{m^2}, \\
		\mathbb E|\bm I_3|^4 & \Le Ch^9 \int_{h/2}^h\D s
		\int_{h/2}^s \D s'
		\int_{h-s'}^{s'} \mathbb E|\bm V_k(s'')|^8 \D s'' \Le \frac{Cd^4h^{12}}{m^4}, \\
		\mathbb E|\bm I_4|^4 & \Le Ch^9\int_{h/2}^h\D s
		\int_{h/2}^s \D s'
		\int_{h-s'}^{s'} \mathbb E\big|\nabla U(\bm X_k(s''))\big|^4\D s'' \Le \frac{Cd^2 h^{12}}{m^2}.
	\end{align*}
	For the stochastic integral $\bm I_5$, its second moment is controlled via
	\begin{align*}
		\mathbb E|\bm I_5|^2 & \Le C h^2
		\int_{h/2}^h \D s \int_{h/2}^s
		\mathbb E\bigg|
		\int_{h-s'}^{s'} e^{-2(h-s'')} \reviewed{\nabla^2 U(\bm X_k(s''))} \D \bm B_{kh+s''}
		\bigg|^2 \D s' \\
		& \Le C h^2 \int_{h/2}^h \D s \int_{h/2}^s \D s'
		\int_{h-s'}^{s'} \reviewed{\mathbb E\big\|\nabla^2 U(\bm X_k(s''))\big\|_{\mathrm F}^2}\D s'' \\
		& \reviewed{{}\Le Ch^2 \int_{h/2}^h \D s \int_{h/2}^s \D s'
		\int_{h-s'}^{s'}  M_2^2 d\, \D s'' \Le Cdh^5,}
	\end{align*}
	\reviewed{where the It\^o isometry produces the Frobenius norm of the Hessian, bounded by $\|\nabla^2 U\|_{\mathrm F}^2 \Le d\,\|\nabla^2 U\|^2 \Le M_2^2 d$ under Assumption~\ref{asP}.}
	Writing $\bm J(s') = \int_{h-s'}^{s'} e^{-2(h-s'')}\nabla^2 U(\bm X_k(s''))\D\bm B_{kh+s''}$, the Burkholder--Davis--Gundy inequality \citep{karatzas1991brownian} gives
	\begin{equation*}
		\mathbb E|\bm J(s')|^4 \Le C\,\mathbb E\bigg(\int_{h-s'}^{s'}\big\|\nabla^2 U(\bm X_k(s''))\big\|_{\mathrm F}^2\D s''\bigg)^2 \Le C\big(M_2^2dh\big)^2,
	\end{equation*}
	and Cauchy's inequality on the outer double integral yields the fourth moment
	\begin{equation}
		\mathbb E|\bm I_5|^4 \Le Ch^6 \int_{h/2}^h \D s \int_{h/2}^s \mathbb E|\bm J(s')|^4 \D s' \Le
		\reviewed{Cd^2h^{10}}.
	\label{proof: local 1}
	\end{equation}
	Therefore, in the discretization error $\Dv = \bar{\bm V}_{k+1} - \bm V_k(h)$, we select the martingale component $\Dmv = \bm I_5$ with the high-order term $\Dhv$ as
	\begin{equation*}
		\Dhv =  \frac{he^{\reviewed{-h}}}{M_2} \big(\nabla U(\bm X_k(h/2)) - \nabla U(\bm Y_k)\big) + \bm I_1 + \bm I_2 + \bm I_3 + \bm I_4.
	\end{equation*}
	The fourth moment of $\Dhv$ then satisfies
	\begin{equation}
		\mathbb E|\Dhv|^4 \Le Ch^4\mathbb E|\Dy|^4 
		+ C\sum_{i=1}^4 \mathbb E|\bm I_i|^4 \Le \frac{Cd^4h^{12}}{m^4}.
		\label{proof: local 2}
	\end{equation}
	
	Finally, we bound the fourth moment of $\bm I_6$ and $\bm I_7$. Observing the inequality
	\begin{equation*}
		|\bm I_6| \Le Ch\int_0^h \D s \bigg|\int_{h/2}^s |\bm V_k(s')|\D s'\bigg|,
	\end{equation*}
	then H\"older's inequality yields the moment bound
	\begin{equation*}
		\mathbb E|\bm I_6|^4 \Le Ch^{\reviewed{10}} \int_0^h \D s \bigg|\int_{h/2}^s \mathbb E|\bm V_k(s')|^4\D s'\bigg| \Le \frac{Cd^2h^{12}}{m^2}.
	\end{equation*}
	Next, from the chain rule
	\begin{equation*}
		\frac{\D}{\D s} \Big(e^{2s}\nabla U(\bm X_k(s))\Big) = e^{2s} \Big(
		2\nabla U(\bm X_k(s)) + \nabla^2 U(\bm X_k(s))\bm V_k(s)\Big),
	\end{equation*}	
	we deduce the equality
	\begin{equation*}
	\bm I_7 = -\frac1{M_2} 
	\int_{h/2}^h \D s \int_{h/2}^s \D s' \int_{h-s'}^{s'} \frac{\D}{\D s''}
	\Big(e^{-2(h-s'')} \nabla U(\bm X_k(s''))\Big)\D s'' = -\frac12(\bm I_1+\bm I_2).
	\end{equation*}
	Hence the fourth moment of $\bm I_7$ is bounded by
	\begin{equation*}
		\mathbb E|\bm I_7|^4 \Le \frac{Cd^2h^{12}}{m^2}.
	\end{equation*}
	In total, the moment bound in the position variable reads
	\begin{equation}
		\mathbb E|\Dx|^4 \Le C\big( h^8\,\mathbb E|\Dy|^4 + \mathbb E|\bm I_6|^4 + \mathbb E|\bm I_7|^4 \big) \Le \frac{Cd^2h^{12}}{m^2}.
		\label{proof: local 3}
	\end{equation}
	
	Inequalities \eqref{proof: local 1}--\eqref{proof: local 3} complete the proof of Lemma~\ref{lemma: discretization error}.
\end{proofof}
We note that the estimate for the term $\bm I_3$ in the preceding proof requires both the bound on the third derivative, $\nabla^3 U(\bm x)$, and the uniform eighth moment bound of the numerical solution $(\bm X_k,\bm V_k)_{k=0}^\infty$. For this reason, Assumption~\ref{asSg} requires an eighth moment bound on the stochastic gradient deviation $b(\bm x,\theta) - \nabla U(\bm x)$.

\subsection{Estimate of long-time displacement}
At the $K$-th step, the state displacement $\bm Z_K - \bm Z_0$ can be estimated as follows. 
\begin{lemma}
	\label{lemma: displacement}
	Under Assumptions \ref{asP} and \ref{asSg}, let the step size $h\Le \frac14$, and the initial state
	$$
	\bm Z_0 = (\bm X_0,\bm V_0)\text{~with~}\bm X_0 = \bm 0\text{~and~}\bm V_0 \sim \mathcal N(\bm 0,M_2^{-1}\bm I_d);
	$$
	then the SG-UBU solution $(\bm X_k,\bm V_k)_{k=0}^\infty$ satisfies
	\begin{equation*}
		\mathbb E|\bm Z_K - \bm Z_0|^2 \Le \frac{CdKh}{m},
	\end{equation*}
	where the constant $C$ depends only on $M_1,M_2$.
\end{lemma}
\begin{proof}
We write the one-step SG-UBU update $\bm Z_{k+1}$ as
\begin{subequations}
	\begin{align}
		\bm X_{k+1} - \bm X_k & = \frac{1-e^{-2h}}2 \bm V_k
		- \frac{h(1-e^{-h})}{2M_2} b(\bm Y_k,\theta_k) + \sqrt{h} \bm M_k, \\
		\bm V_{k+1} - \bm V_k & = -(1-e^{-2h}) \bm V_k -
		\frac{he^{-h}}{M_2} b(\bm Y_k,\theta_k) +
		\sqrt{h} \bm N_k,
	\end{align}
	\label{XV form}
\end{subequations}
where the martingales $\bm M_k$ and $\bm N_k$ are defined as
\begin{equation*}
	\bm M_k = \frac2{\sqrt{M_2h}} \int_0^h \frac{1-e^{-2(h-s)}}2 \D \bm B_{kh+s},~~~~
	\bm N_k = \frac2{\sqrt{M_2h}} \int_0^h e^{-2(h-s)}\D\bm B_{kh+s},
\end{equation*}
and their second moments satisfy
\begin{equation*}
	\mathbb E|\bm M_k|^2 \Le Cd,~~~~
	\mathbb E|\bm N_k|^2 \Le Cd.
\end{equation*}
Applying Cauchy's inequality to \eqref{XV form} yields
\begin{align*}
	\mathbb E \bigg|\bm X_K - \bm X_0 - \sqrt{h} & \sum_{k=0}^{K-1}
	\bm M_k\bigg|^2  = \mathbb E\bigg|
	\sum_{k=0}^{K-1} \big(
	\bm X_{k+1} - \bm X_k - \sqrt{h} \bm M_k
	\big)
	\bigg|^2 \\
	& \Le K\sum_{k=0}^{K-1} \mathbb E
	\bigg|\frac{1-e^{-2h}}2 \bm V_k
	- \frac{h(1-e^{-h})}{2M_2} b(\bm Y_k,\theta_k)\bigg|^2 \Le \frac{CdK^2h^2}{m},
\end{align*}
and similarly,
\begin{align*}
	\mathbb E\bigg|\bm V_K - \bm V_0 - \sqrt{h} & \sum_{k=0}^{K-1}
	\bm N_k\bigg|^2 = \mathbb E\bigg|
	\sum_{k=0}^{K-1} \big(
	\bm V_{k+1} - \bm V_k - \sqrt{h} \bm N_k
	\big)
	\bigg|^2 \\
	& \Le K\sum_{k=0}^{K-1} \mathbb E
	\bigg|-(1-e^{-2h}) \bm V_k
	- \frac{h e^{-h}}{M_2} b(\bm Y_k,\theta_k)\bigg|^2 \Le \frac{CdK^2h^2}{m}.
\end{align*}
Since the $\bm M_k$ are independent and centered, we obtain the estimate
\begin{align*}
	\mathbb E|\bm X_K - \bm X_0|^2 & \Le 2h \mathbb E
	\bigg|\sum_{k=0}^{K-1} \bm M_k\bigg|^2 + \frac{CdK^2h^2}{m} \\
	& = 2h \sum_{k=0}^{K-1} \mathbb E|\bm M_k|^2 + \frac{CdK^2h^2}{m} \Le \frac{Cd}{m} (Kh+K^2h^2).
\end{align*}
Similarly, for $\bm V_k - \bm V_0$ we have the same inequality
\begin{equation}
	\mathbb E|\bm Z_K - \bm Z_0|^2 \Le \frac{Cd}{m}(Kh+K^2h^2).
	\label{displacement: 1}
\end{equation}
On the other hand, from the uniform-in-time moments in Theorem~\ref{theorem: SG-UBU moment bound}, we have
\begin{equation}
	\mathbb E|\bm Z_K - \bm Z_0|^2 \Le 2\mathbb E
	\big(|\bm Z_K|^2 + |\bm Z_0|^2\big)
	\Le \frac{Cd}{m}.
	\label{displacement: 2}
\end{equation}
Integrating \eqref{displacement: 1}, \eqref{displacement: 2}, and the algebraic inequality $\min\{x+x^2,1\} \Le 2x$ we obtain
\begin{equation*}
	\mathbb E|\bm Z_K - \bm Z_0|^2 \Le \frac{CdKh}{m},
\end{equation*}
which completes the proof of Lemma~\ref{lemma: displacement}.
\end{proof}

\section{Proofs for MSE and numerical bias of SG-UBU}
\label{appendix: proof SG-UBU}
Recall that the time average error of SG-UBU is expressed as in \eqref{explicit time average}:
\begin{equation*}
	\frac{1}{K}\sum_{k=0}^{K-1}f(\bm{X}_k)-\pi(f)=\frac{\phi_h(\bm{Z}_0)-\phi_h(\bm{Z}_K)}{Kh} + \frac{1}{K}\sum_{k=0}^{K-1}(R_k + S_k + T_k),
\end{equation*}
where the error terms $R_k$, $S_k$, and $T_k$ are given by
\begin{gather*}
	R_k  =\frac{\phi_h(\bm{Z}_{k+1})-\phi_h(\bar{\bm{Z}}_{k+1})}{h},\qquad S_k=\frac{\phi_h(\bar{\bm{Z}}_{k+1})-\phi_h(\bm{Z}_k(h))}{h},\\
	T_k =\frac{\phi_h(\bm{Z}_k(h))-\phi_h(\bm{Z}_k)}{h}+f(\bm{X}_k)-\pi(f).
\end{gather*}
\subsection{Mean square error of SG-UBU}
\label{appendix: SG-UBU 1}

\begin{proofof}[Theorem~\ref{theorem: SG-UBU MSE}]
	According to \eqref{explicit time average}, we only need to estimate the second moments:
	\begin{equation*}
		\mathbb E\Bigg[\bigg(\sum_{k=0}^{K-1} R_k\bigg)^2\Bigg],\quad
		\mathbb E\Bigg[\bigg(\sum_{k=0}^{K-1} S_k\bigg)^2\Bigg],\quad
		\mathbb E\Bigg[\bigg(\sum_{k=0}^{K-1} T_k\bigg)^2\Bigg].
	\end{equation*}
	
	First, the stochastic gradient error $R_k$ can be written as
	\begin{equation*}
		R_k = \frac1{h}\big(\phi_h(\bm Z_{k+1}) - \phi_h(\bar{\bm Z}_{k+1})\big) = \frac1{h}
		(\bm Z_{k+1} - \bar{\bm Z}_{k+1}) ^\top
		\int_0^1 \nabla \phi_h\big(\lambda \bm Z_{k+1} + (1-\lambda) \bar{\bm Z}_{k+1}\big) \D\lambda,
	\end{equation*}
	we decompose $R_k = R_{k,1} + R_{k,2}$, where $R_{k,1}$ is a martingale and $R_{k,2}$ is a high-order remainder:
	\begin{align*}
		R_{k,1} & = \frac1h (\bm Z_{k+1} - \bar {\bm Z}_{k+1}) ^\top \nabla \phi_h(\bar{\bm Z}_{k+1}), \\
		R_{k,2} & = \frac1h (\bm Z_{k+1} - \bar {\bm Z}_{k+1})
		^\top \int_0^1 \Big(\nabla \phi_h\big(\lambda \bm Z_{k+1} + (1-\lambda) \bar{\bm Z}_{k+1}\big)-\nabla\phi_h(\bar {\bm Z}_{k+1})\Big)\D\lambda.
	\end{align*}
	The bounds of $\nabla\phi_h(\bm x,\bm v)$ and $\nabla^2\phi_h(\bm x,\bm v)$ in 
	Theorem~\ref{theorem: phi h estimate} imply the inequalities
	\begin{equation*}
		|R_{k,1}| \Le \frac C{mh} |\bm Z_{k+1} - \bar{\bm Z}_{k+1}|,\qquad
		|R_{k,2}| \Le \frac C{m^2h} |\bm Z_{k+1} - \bar{\bm Z}_{k+1}|^2.
	\end{equation*}
	Hence utilizing the stochastic gradient error estimate in Lemma~\ref{lemma: stochastic gradient error}, we obtain
	\begin{align*}
		\mathbb E[R_{k,1}^2] & \Le \frac{C}{m^2h^2} \mathbb E|\bm Z_{k+1} - \bar{\bm Z}_{k+1}|^2 \Le \frac{C\sigma^2d}{m^2}, \\
		\mathbb E[R_{k,2}^2] & \Le \frac{C}{m^4h^2} \mathbb E|\bm Z_{k+1} - \bar{\bm Z}_{k+1}|^4 \Le \frac{C\sigma^4d^2h^2}{m^4}.
	\end{align*}
	Since $R_{k,1}$ is a martingale difference with \revised{$\mathbb E[R_{k,1}|\mathcal G_k] = 0$, where
	\begin{equation*}
		\mathcal G_k = \mathcal F_k \vee \sigma\big((\bm B_t)_{kh\Le t\Le (k+1)h}\big)
	\end{equation*}
	enlarges $\mathcal F_k$ by the Brownian increment over the $k$-th step, so that $\bar{\bm Z}_{k+1}$ is $\mathcal G_k$-measurable while $\theta_k$ stays independent of $\mathcal G_k$, and $R_{j,1}$ is $\mathcal G_k$-measurable for $j<k$},
	\begin{equation*}
		\mathbb E\Bigg[\bigg(\sum_{k=0}^{K-1} R_{k,1}\bigg)^2\Bigg]
		= \sum_{k=0}^{K-1} \mathbb E[R_{k,1}^2] \Le \frac{C\sigma^2dK}{m^2}.
	\end{equation*}
	On the other hand, by Cauchy's inequality
	\begin{equation*}
		\mathbb E\Bigg[\bigg(\sum_{k=0}^{K-1} R_{k,2}\bigg)^2\Bigg]
		\Le K \sum_{k=0}^{K-1} \mathbb E[R_{k,2}^2] \Le
		\frac{C\sigma^4d^2K^2h^2}{m^4}.
	\end{equation*}
	Combining the estimates of $\sum_{k=0}^{K-1} R_{k,1}$ and $\sum_{k=0}^{K-1} R_{k,2}$ yields
	\begin{equation}
		\frac1{K^2}\mathbb E\Bigg[\bigg(\sum_{k=0}^{K-1} R_k\bigg)^2\Bigg] \Le C\bigg(\frac{\sigma^2d}{m^2K} + \frac{\sigma^4d^2h^2}{m^4}\bigg).
		\label{estimate: R}
	\end{equation}
	
	The discretization error $S_k$ analysis parallels the approach for $R_k$. Adopting the analytical framework in Lemma~\ref{lemma: discretization error}, we decompose the local error $\bar{\bm Z}_{k+1} - \bm Z_k(h)$ as
	\begin{equation*}
		\bar{\bm X}_{k+1} - \bm X_k(h) = \Dx,~~~~
		\bar{\bm V}_{k+1} - \bm V_k(h) = \Dmv + \Dhv,
	\end{equation*}
	where $\Dmv$ is a martingale and $\Dhv$ is a higher-order term. Now $S_k$ decomposes as
	\begin{align*}
		S_{k,1} & = \frac1h\Big(\Dx^\top \nabla_{\bm x} \phi_h(\bm Z_k) + \Dhv^\top \nabla_{\bm v} \phi_h(\bm Z_k)\Big), \\
		S_{k,2} & = \frac1h\Dmv^\top \nabla_{\bm v} \phi_h(\bm Z_k), \\
		S_{k,3} & = \frac1h(\bar{\bm Z}_{k+1} - \bm Z_k(h))^\top \int_0^1 \Big(\nabla \phi_h\big(\lambda \bar{\bm Z}_{k+1} + (1-\lambda) \bm Z_k(h)\big) - \nabla\phi_h(\bm Z_k)\Big)\D\lambda.
	\end{align*}
	Theorem~\ref{theorem: phi h estimate} yields the bounds of $S_{k,1}$ and $S_{k,2}$:
	\begin{equation*}
		|S_{k,1}| \Le \frac{C}{mh} \big(|\Dx|+|\Dhv|\big),\qquad
		|S_{k,2}| \Le \frac{C}{mh} |\Dmv|.
	\end{equation*}
	Hence leveraging the discretization error estimates in Lemma~\ref{lemma: discretization error} produces
	\begin{equation}
		\mathbb E[S_{k,1}^2] \Le \frac{Cd^2h^4}{m^4},\qquad
		\mathbb E[S_{k,2}^2] \Le \reviewed{\frac{Cdh^3}{m^2}}.
        \label{S 1}
	\end{equation}
	The martingale property of $S_{k,2}$ gives
	\begin{equation}
		\mathbb E\Bigg[\bigg(\sum_{k=0}^{K-1} S_{k,2}\bigg)^2\Bigg] = \sum_{k=0}^{K-1} \mathbb E[S_{k,2}^2] \Le \reviewed{\frac{CdKh^3}{m^2}}.
		\label{proof: S1 estimate}
	\end{equation}
	By Cauchy's inequality, $\sum_{k=0}^{K-1} S_{k,1}$ has the moment bound
	\begin{equation}
		\mathbb E\Bigg[\bigg(\sum_{k=0}^{K-1} S_{k,1}\bigg)^2\Bigg] \Le K \sum_{k=0}^{K-1} \mathbb E[S_{k,1}^2] \Le \frac{Cd^2K^2h^4}{m^4}.
		\label{proof: S2 estimate}
	\end{equation}
	Exploiting the expressions of $\bar{\bm Z}_{k+1}$ and $\bm Z_k(h)$ in Appendix~\ref{appendix: local error}, it is easy to derive
	\begin{equation*}
		\mathbb E|\bar{\bm Z}_{k+1} - \bm Z_k|^4 \Le \frac{Cd^2h^2}{m^2},~~~~
		\mathbb E|\bm Z_k(h)-\bm Z_k|^4\Le \frac{Cd^2h^2}{m^2}.
	\end{equation*}
	Applying Theorem~\ref{theorem: phi h estimate} we obtain
	\begin{equation*}
		|S_{k,3}| \Le \frac{C}{m^2h}
		\sqrt{|\bar{\bm Z}_{k+1} - \bm Z_k(h)|^2\big(
			|\bar{\bm Z}_{k+1} - \bm Z_k|^2 + 
			|\bm Z_k(h) - \bm Z_k|^2\big)},
	\end{equation*}
	and thus by the discretization error moment estimate in Lemma~\ref{lemma: discretization error},
	\begin{align}
		\mathbb E[S_{k,3}^2] & \Le
		\frac{C}{m^4h^2}\sqrt{\mathbb E|\bar{\bm Z}_{k+1} - \bm Z_k(h)|^4 \mathbb E\big(|\bar{\bm Z}_{k+1} - \bm Z_k|^4 + 
			|\bm Z_k(h) - \bm Z_k|^4\big)} \notag\\
		& \Le \frac{C}{m^4h^2}\bigg(\frac{d^2h^6}{m^2} + \reviewed{dh^5}\bigg)\frac{dh}{m} = C\bigg(\reviewed{\frac{d^2h^4}{m^5}} + \frac{d^3h^5}{m^7}\bigg).
        \label{S split}
	\end{align}
	Then by Cauchy's inequality:
	\begin{equation}
		\mathbb E\Bigg[\bigg(\sum_{k=0}^{K-1} S_{k,3}\bigg)^2\Bigg] \Le C\bigg(\reviewed{\frac{d^2K^2h^4}{m^5}} + \frac{d^3K^2h^5}{m^7}\bigg).
		\label{proof: S3 estimate}
	\end{equation}
	Synthesizing the inequalities \eqref{proof: S1 estimate}, \eqref{proof: S2 estimate}, and \eqref{proof: S3 estimate} gives
	\begin{equation}
		\frac1{K^2}\mathbb E\Bigg[\bigg(\sum_{k=0}^{K-1} S_k\bigg)^2\Bigg] \Le C\bigg(\reviewed{\frac{dh^3}{m^2K}} + \reviewed{\frac{d^2h^4}{m^5}} + \frac{d^3h^5}{m^7}\bigg).
		\label{estimate: S}
	\end{equation}
	
	Finally, decompose the martingale $T_k$ as $T_k = T_{k,1} + T_{k,2}$ with
	\begin{equation*}
		T_{k,1} = \frac1h \big(\phi_h(\bm Z_k(h)) - \phi_h(\bm Z_k)\big),\qquad
		T_{k,2} = f(\bm X_k) - \pi(f).
	\end{equation*}
	Applying Theorem~\ref{theorem: phi h estimate} produces the bound
	\begin{equation*}
		|T_{k,1}| \Le \frac{C}{mh} |\bm Z_k(h) - \bm Z_k|~
		\Longrightarrow~\mathbb E[T_{k,1}^2] \Le \frac{C}{m^2h^2}
		\mathbb E|\bm Z_k(h) - \bm Z_k|^2 \Le
		\frac{Cd}{m^3h}.
	\end{equation*}
	For $T_{k,2}$, take $\bm X_*\sim \pi$ independent of $\bm X_k$, then $\mathbb E|\bm X_*|^2 \Le \frac{Cd}m$ and
	\begin{equation*}
		\mathbb E[T_{k,2}^2] \Le \mathbb E|f(\bm X_k)-f(\bm X_*)|^2 \Le
		C\mathbb E\big(|\bm X_k|^2 + |\bm X_*|^2\big) \Le
		\frac{Cd}{m}.
	\end{equation*}
	Combining these moment estimates, we derive
	\begin{equation*}
		\mathbb E[T_k^2] \Le C\big(\mathbb E [T_{k,1}^2]+\mathbb E[T_{k,2}^2]\big)
		\Le C\bigg(\frac{d}{m^3h} + \frac{d}{m}\bigg) \Le \frac{Cd}{m^3h}.
	\end{equation*}
	Leveraging the martingale property of $T_k$, we establish the time-averaged moment bound
	\begin{equation}
		\frac1{K^2}\mathbb E\Bigg[ \bigg(\sum_{k=0}^{K-1}
		T_k \bigg)^2\Bigg] \Le \frac{Cd}{m^3Kh}.
		\label{estimate: T}
	\end{equation}
	
	Combining the estimates \eqref{estimate: R}, \eqref{estimate: S}, and \eqref{estimate: T}, we derive
	\begin{equation}
		\mathbb E\Bigg[\bigg(\frac1K\sum_{k=0}^{K-1} (R_k + S_k + T_k)\bigg)^2\Bigg]
		\Le C\bigg(\frac{d}{m^3Kh} + \frac{\sigma^4d^2h^2}{m^4} + \reviewed{\frac{d^2h^4}{m^5}} + \frac{d^3h^5}{m^7} \bigg).
		\label{final estimate: 1}
	\end{equation}
	On the other hand, Lemma~\ref{lemma: displacement} provides the displacement moment estimate:
	\begin{equation}
		\mathbb E\Bigg[\bigg(\frac{\phi_h(\bm Z_0) - \phi_h(\bm Z_K)}{Kh}\bigg)^2\Bigg] \Le
		\frac{C}{m^2K^2h^2}\mathbb E|\bm Z_0 - \bm Z_K|^2 \Le
		\frac{Cd}{m^3Kh}.
		\label{final estimate: 2}
	\end{equation}
	Integrating \eqref{final estimate: 1}, \eqref{final estimate: 2} with the time average error expression \eqref{explicit time average} yields
	\begin{equation*}
		\mathbb E \Bigg[\bigg(\frac1K\sum_{k=0}^{K-1} f(\bm X_k) - \pi(f)\bigg)^2\Bigg] \Le
		C\bigg(\frac{d}{m^3Kh} + \frac{\sigma^4d^2h^2}{m^4} + \reviewed{\frac{d^2h^4}{m^5}} + \frac{d^3h^5}{m^7} \bigg),
	\end{equation*}
	completing the proof of Theorem~\ref{theorem: SG-UBU MSE}.
\end{proofof}

The MSE analysis for mini-batch SG-UBU requires a quantitative bound on the high-order moments of the stochastic gradient. The following lemma establishes the key estimate.
\begin{lemma}
	\label{lemma: eighth moment}
	Let $x_1,\dots,x_N \in \mathbb{R}^d$ be vectors with bounded eighth moments:
	\begin{equation*}
		\frac1N \sum_{i=1}^N |x_i|^8 \Le 1.
	\end{equation*}
	Denote $\bar x_N = \frac1N \sum_{i=1}^N x_i \in \mathbb R^d$. Let $\theta$ be a random subset of $\{1,\dots,N\}$ of size $p$. Then
	\begin{equation*}
		\mathbb E_{\theta}\Bigg[\bigg|\frac1p\sum_{i\in \theta} x_i - \bar x_N\bigg|^8\Bigg] \Le
		\frac{C}{p^4},
	\end{equation*}
	where the constant $C$ is absolute.
\end{lemma}
The proof mainly relies on Hoeffding's inequality \citep{hoeffding1963probability} and Rosenthal's inequality \citep{petrov2012sums,chen2020rosenthal}, which are stated below.
\begin{theorem}[Hoeffding]
Let $\Pi = \{x_1, \dots, x_N\}$ be a finite set. Let $(X_1, \dots, X_p)$ be a sample drawn without replacement from $\Pi$, and let $(Y_1, \dots, Y_p)$ be a sample drawn with replacement from $\Pi$. If $f$ is a continuous convex function on $\mathbb R^d$,
$$
\mathbb E\bigg[f\bigg(\sum_{j=1}^p X_j\bigg)\bigg] \Le \mathbb E\bigg[f\bigg(\sum_{j=1}^p Y_j\bigg)\bigg].
$$
\end{theorem}
\begin{theorem}[Rosenthal]
	Let $Z_1, \dots, Z_p$ be independent mean-zero random vectors in $\mathbb R^{d}$. For $k \Ge 2$, there exists a constant $C_k$ depending only on $k$ such that
	$$
	\mathbb E\Bigg[\bigg|\sum_{j=1}^p Z_j\bigg|^k\Bigg] \Le C_k \max\Bigg(\sum_{j=1}^p \mathbb E\big[|Z_j|^k\big], \bigg(\sum_{j=1}^p \mathbb E\big[|Z_j|^2\big]\bigg)^{\frac k2} \Bigg).
	$$
\end{theorem}

\begin{proofof}[Lemma~\ref{lemma: eighth moment}]
Since $|\bar x_N|^8 \Le \frac1N\sum_{i=1}^N|x_i|^8 \Le 1$, replacing $x_i$ by $x_i - \bar x_N$ preserves the hypothesis up to the absolute factor $2^8$, and the constant $C$ below is absolute; we may therefore assume $\bar x_N = 0$.
Let $X_1,\cdots,X_p$ denote independent samples from $\{\reviewed{x_1,\cdots,x_N}\}$ (with replacement), then by Hoeffding's inequality we have
\begin{equation}
	\mathbb E_{\theta}\Bigg[\bigg|\frac1p\sum_{i\in \theta} x_i\bigg|^8\Bigg] \Le 
	\mathbb E\Bigg[\bigg|\frac1p \sum_{j=1}^p X_j\bigg|^8\Bigg].
	\label{eighth: 1}
\end{equation}
Since $X_1,\cdots,X_p$ are independent, applying Rosenthal's inequality with $k=8$,
\begin{align*}
\mathbb E\Bigg[\bigg|\sum_{j=1}^p X_j\bigg|^8\Bigg] & \Le C \max\Bigg(\sum_{j=1}^p \mathbb E\big[|X_j|^8\big], \bigg(\sum_{j=1}^p \mathbb E\big[|X_j|^2\big]\bigg)^4 \Bigg) \\
& \Le C\bigg(p\,\mathbb E\big[|X_1|^8\big] + p^4 \Big(\mathbb E\big[|X_1|^2\big]\Big)^4\bigg) \\
& \Le Cp^4 \mathbb E\big[|X_1|^8\big] = \frac{Cp^4}{N}
\sum_{i=1}^N |x_i|^8 \Le Cp^4.
\end{align*}
Hence we derive the upper bound
\begin{equation}
	\mathbb E\Bigg[\bigg|\frac1p \sum_{j=1}^p X_j\bigg|^8\Bigg] \Le \frac{C}{p^4}.
	\label{eighth: 2}
\end{equation}
Combining the estimates \eqref{eighth: 1} and \eqref{eighth: 2} yields the result of Lemma~\ref{lemma: eighth moment}.
\end{proofof}

By Lemma~\ref{lemma: eighth moment}, Assumption~\ref{asSgN} implies the following eighth moment control of the mini-batch gradient error: for any $\bm x\in\mathbb R^d$ and random subset $\theta\subset\{1,\dots,N\}$ with size $p$,
\begin{equation*}
	\Bigg(\mathbb E_{\theta}\Bigg[\bigg|\frac1{p}\sum_{i\in\theta}\nabla U_i(\bm x)-\nabla U(\bm x)\bigg|^8\Bigg]\Bigg)^{\tfrac14}
	\Le \frac{C\sigma^2 d}{p}.
\end{equation*}
Hence the stochastic gradient in mini-batch SG-UBU satisfies all items of Assumption~\ref{asSg} with $\sigma^2$ replaced by $C\sigma^2/p$. Therefore, Theorem~\ref{theorem: mini-batch SG-UBU MSE} follows directly from Theorem~\ref{theorem: SG-UBU MSE}.

\subsection{Numerical bias of SG-UBU}
\label{appendix: SG-UBU 2}

\begin{proofof}[Theorem~\ref{theorem: SG-UBU numerical bias}]
To streamline the notation, expectations are taken with respect to the numerical invariant distribution $\pi_h$ and the stochastic gradient index $\theta$, which we omit in the proof. From Lemma~\ref{lemma: SG-UBU invariant},  $\bm Z_0 = (\bm X_0,\bm V_0)$ drawn from $\pi_h$ satisfies the moment bound:
\begin{equation*}
	\mathbb E\big[|\bm Z_0|^8\big] \Le \frac{Cd^4}{m^4}.
\end{equation*}
According to \eqref{bias decomposition}, 
the numerical bias is the sum of the expected stochastic gradient error, $\mathbb E[R_0]$, and the expected discretization error, $\mathbb E[S_0]$, where
\begin{equation*}
	R_0 = \frac{\phi_h(\bm Z_1) - \phi_h(\bar{\bm Z_1})}{h},\qquad
	S_0 = \frac{\phi_h(\bar{\bm Z}_1) - \phi_h(\bm Z_0(h))}{h}.
\end{equation*}
\revised{Since $\pi_h$ satisfies the same uniform-in-time moment bound by Lemma~\ref{lemma: SG-UBU invariant}, the local error estimates of Lemma~\ref{lemma: stochastic gradient error} and Lemma~\ref{lemma: discretization error}, together with \eqref{S 1} and \eqref{S split}, remain valid here: all of them depend on the state only through its eighth moment.}

We begin with the stochastic error term $\mathbb E[R_0]$. A Taylor expansion of $\phi_h(\bm Z_1)$ around $\bar{\bm Z}_1$ using the mean value theorem gives
\begin{align*}
	\mathbb E[R_0] & =  \frac1h \mathbb E\bigg[(\bm Z_1 - \bar {\bm Z}_1)^\top \int_0^1 \nabla \phi_h\big(\lambda \bm Z_1 + (1-\lambda) \bar{\bm Z}_1\big)\D\lambda \bigg].
\end{align*}
The key insight is that the first-order term of this expansion vanishes in expectation due to the martingale property of the stochastic gradient error established in \eqref{martingale property}:
\begin{equation*}
	\mathbb E\big[(\bm Z_1 - \bar{\bm Z}_1)^\top \nabla \phi_h(\bar {\bm Z}_1)\big] = 0.
\end{equation*}
This simplification leaves only the second-order terms:
\begin{align*}
	\mathbb E[R_0] & = \frac1h \mathbb E\bigg[(\bm Z_1 - \bar {\bm Z}_1)^\top \int_0^1 \Big(\nabla \phi_h\big(\lambda \bm Z_1 + (1-\lambda) \bar{\bm Z}_1\big)-\nabla\phi_h(\bar {\bm Z}_1)\Big)\D\lambda \bigg] \\
	& = \frac1h \mathbb E\bigg[(\bm Z_1 - \bar {\bm Z}_1)^\top  \int_0^1\bigg(\lambda\int_0^1\nabla^2 \phi_h\big(\bar{\bm Z}_1+\lambda\varphi(\bm Z_1 - \bar{\bm Z}_1)\big)\D\varphi\bigg)\D\lambda\,
	(\bm Z_1-\bar{\bm Z_1}) \bigg].
\end{align*}
Recall that the components of the local error $\bm Z_1 - \bar{\bm Z}_1$ converge at different rates. Lemma~\ref{lemma: stochastic gradient error} provides the moment bounds:
\begin{equation*}
	\sqrt{\mathbb E\big|\bm X_1 - \bar{\bm X}_1\big|^4}
	\Le C\sigma^2dh^4,\qquad
	\sqrt{\mathbb E\big|\bm V_1 - \bar{\bm V}_1\big|^4}
	\Le C\sigma^2dh^2.
\end{equation*}
Since the position error $\bm X_1 - \bar{\bm X}_1$ is of a higher order in $h$ than the velocity error, its contribution to the bias is smaller. Using the Hessian bound $\|\nabla^2\phi_h\| \Le C/m^2$ from Theorem~\ref{theorem: phi h estimate}, we can bound these higher-order terms:
\begin{align*}
	& \frac1h \mathbb E\bigg[(\bm X_1 - \bar {\bm X}_1)
	^\top \int_0^1\bigg(\lambda\int_0^1\nabla_{\bm x\bm v}^2 \phi_h\big(\bar{\bm Z}_1+\lambda\varphi(\bm Z_1 - \bar{\bm Z}_1)\big)\D\varphi\bigg)\D\lambda\,
	(\bm V_1-\bar{\bm V}_1) \bigg] \Le \frac{Cdh^2}{m^2}, \\
	& \frac1h \mathbb E\bigg[(\bm X_1 - \bar {\bm X}_1)
	^\top \int_0^1\bigg(\lambda\int_0^1\nabla_{\bm x\bm x}^2 \phi_h\big(\bar{\bm Z}_1+\lambda\varphi(\bm Z_1 - \bar{\bm Z}_1)\big)\D\varphi\bigg)\D\lambda\,
	(\bm X_1-\bar{\bm X}_1) \bigg] \Le \frac{Cdh^3}{m^2}.
\end{align*}
The dominant contribution to the bias therefore comes from the velocity error component. Using the definition of $\bm M_0$ from the theorem statement, we arrive at the estimate:
\begin{equation}
	\bigg|\mathbb E[R_0] - \frac1{2h} \mathbb E\Big[(\bm V_1-\bar{\bm V}_1)^\top
	\bm M_0 (\bm V_1-\bar{\bm V}_1)\Big]\bigg| \Le  \frac{Cdh^2}{m^2}.
	\label{proof: R_0 estimate}
\end{equation}
Finally, we substitute the explicit expression for the velocity error from \eqref{stochastic gradient error expression}:
\begin{equation*}
	\bm V_1 - \bar{\bm V}_1 = -\frac{he^{-h}}{M_2}
	\big(b(\bm Y_0,\theta) - \nabla U(\bm Y_0)\big).
\end{equation*}
This substitution into \eqref{proof: R_0 estimate}, along with the linear approximation $e^{-2h} = 1 + \mathcal O(h)$, yields the leading-order term for the stochastic error's contribution to the bias:
\begin{equation*}
	\bigg|\mathbb E[R_0] - \frac{h}{2M_2^2} \mathbb E
	\Big[\big(b(\bm Y_0,\theta) - \nabla U(\bm Y_0)\big)^\top \bm M_0 \big(b(\bm Y_0,\theta) - \nabla U(\bm Y_0)\big)\Big]\bigg| \Le \frac{Cdh^2}{m^2}.
\end{equation*}

For the discretization error component, $\mathbb E[S_0]$, we directly follow the decomposition in Theorem~\ref{theorem: SG-UBU MSE} and write $S_0 = S_{0,1} + S_{0,2} + S_{0,3}$, where
\begin{align*}
    S_{0,1} & = \frac1h\Big(\delta\bm X_0^\top \nabla_{\bm x} \phi_h(\bm Z_0) + \delta_h \bm V_0^\top \nabla_{\bm v} \phi_h(\bm Z_0)\Big), \\
		S_{0,2} & = \frac1h\delta_m\bm V_0^\top \nabla_{\bm v} \phi_h(\bm Z_0), \\
		S_{0,3} & = \frac1h(\bar{\bm Z}_{1} - \bm Z_0(h))^\top \int_0^1 \Big(\nabla \phi_h\big(\lambda \bar{\bm Z}_{1} + (1-\lambda) \bm Z_0(h)\big) - \nabla\phi_h(\bm Z_0)\Big)\D\lambda.
\end{align*}
The martingale property of $\delta_m\bm V_0$ stated in Lemma~\ref{lemma: discretization error} implies $\mathbb E[S_{0,2}] = 0$, and thus the expectation simplifies to $\mathbb E[S_0] = \mathbb E[S_{0,1}] + \mathbb E[S_{0,3}]$. Utilizing the moment estimates of $S_{0,1}$ and $S_{0,3}$ in \eqref{S 1} and \eqref{S split}, and the step size constraint $h\Le \frac{m}{d}$, we obtain
\begin{equation*}
    \mathbb E|S_{0,1}| \Le \frac{Cdh^2}{m^2},\qquad 
    \mathbb E|S_{0,3}| \Le \frac{Cdh^2}{m^3}.
\end{equation*}
Combining these estimates provides a bound on the expected total local error:
\begin{equation}
	\big|\mathbb E[S_0]\big| \Le \mathbb E|S_{0,1}| + \mathbb E|S_{0,3}| \Le 
    \frac{Cdh^2}{m^3}.
    \label{proof: S_0 estimate}
\end{equation}

Synthesizing the estimates for the stochastic error in \eqref{proof: R_0 estimate} and the discretization error in \eqref{proof: S_0 estimate}, we arrive at a complete characterization of the numerical bias:
\begin{equation*}
	\bigg|\pi_h(f) - \pi(f) - \frac{h}{2M_2^2} \mathbb E\Big[\big(b(\bm Y_0,\theta) - \nabla U(\bm Y_0)\big)^\top \bm M_0 \big(b(\bm Y_0,\theta) - \nabla U(\bm Y_0)\big)\Big]\bigg| \Le \frac{Cdh^2}{m^3}.
\end{equation*}
This expression isolates the leading-order term of the bias, which is proportional to the variance of the stochastic gradient. By applying the established bounds $\|\bm M_0\| \Le C/m^2$ and $\mathbb E\big[|b(\bm Y_0,\theta) - \nabla U(\bm Y_0)|^2\big] \Le C\sigma^2 d$, we can bound this leading term, confirming that the overall numerical bias exhibits first-order convergence:
\begin{equation*}
	|\pi_h(f) - \pi(f)| \Le C\bigg(\frac{\sigma^2dh}{m^2} + \frac{dh^2}{m^3}\bigg).
\end{equation*}
This completes the proof of Theorem~\ref{theorem: SG-UBU numerical bias}.
\end{proofof}

\begin{proofof}[Lemma~\ref{lemma: paradox}]
We begin by recalling the continuous Poisson equation in \eqref{continuous Poisson solution}:
\begin{equation}
	\bm v^\top \nabla_{\bm x}\phi_0(\bm x,\bm v) - \bigg(\frac1{M_2}\nabla U(\bm x)+2\bm v\bigg)^\top
	\nabla_{\bm v} \phi_0(\bm x,\bm v) + \frac2{M_2} \Delta_{\bm v} \phi_0(\bm x,\bm v) = \pi(f) - f(\bm x).
	\label{recall continuous Poisson}
\end{equation}
Theorem~\ref{theorem: u estimate} provides an exponential decay estimate for the  Kolmogorov solution $u(\bm x,\bm v,t)$:
\begin{equation*}
	\max\big\{
	|\nabla_{\bm x} u(\bm x,\bm v,t)|,|\nabla_{\bm v} u(\bm x,\bm v,t)|
	\big\} \Le C L_1 e^{-\frac{mt}{2M_2}}.
\end{equation*}
By integrating this inequality with respect to $t$ over $(0, \infty)$ and using the integral representation $\phi_0(\bm x,\bm v) = \int_0^\infty u(\bm x,\bm v,t) \D t$, we obtain uniform bounds on the derivatives of $\phi_0(\bm x,\bm v)$:
\begin{equation*}
	\max\big\{
	|\nabla_{\bm x} \phi_0(\bm x,\bm v)|,|\nabla_{\bm v} \phi_0(\bm x,\bm v)|
	\big\} \Le \frac{CL_1}{m}.
\end{equation*}
With these derivative bounds established, we can rearrange the Poisson equation \eqref{recall continuous Poisson} to isolate the Laplacian term. Since the potential gradient $\nabla U(\bm x)$ and the test function $f(\bm x)$ both exhibit at most linear growth in $|\bm x|$, all terms in the equation other than the Laplacian grow at most linearly in $|\bm x|$ and $|\bm v|$. This allows us to establish a growth bound on the Laplacian of the velocity component of $\phi_0$:
\begin{equation}
	|\Delta_{\bm v} \phi_0(\bm x,\bm v)| \Le
	\frac{C}{m}\sqrt{|\bm x|^2+|\bm v|^2+d},\qquad
	\forall \bm x,\bm v\in\mathbb R^d.
	\label{Laplace bound}
\end{equation}

To complete the proof, we need to bound the second moment of the Hessian of the Poisson solution, averaged over the equilibrium velocity distribution. Thus we define the quantity $\mathcal Q(\bm x)$ as this velocity-averaged integral:
\begin{equation*}
	\mathcal Q(\bm x) = \bigg(\frac{M_2}{2\pi}\bigg)^{\frac d2}\int_{\mathbb R^d} \big\|\nabla_{\bm v\bm v}^2 \phi_0(\bm x,\bm v)\big\|^2 e^{-\frac{M_2}2|\bm v|^2}\D\bm v.
\end{equation*}
The goal is to show that the expectation of $\mathcal Q(\bm x)$ with respect to the target distribution $\pi(\bm x)$ is appropriately bounded:
\begin{equation}
	\mathbb E^{\pi} \big[\mathcal Q(\bm x)\big] \Le \frac{C_d}{m^3}.
	\label{Qx bound}
\end{equation}
Our strategy is to leverage local elliptic regularity estimates. We partition the velocity space $\mathbb R^d$ into unit hypercubes centered at integer vectors $\bm j = (j_1,\cdots,j_d) \in \mathbb Z^d$:
\begin{align*}
	B_{\bm j} & = \Big(j_1-\frac12,j_1+\frac12\Big) \times \cdots \times \Big(j_d-\frac12,j_d+\frac12\Big), \\
	\tilde B_{\bm j} & = (j_1-1,j_1+1) \times \cdots \times (j_d-1,j_d+1).
\end{align*}
Standard local $L^2$-estimates for elliptic equations allow us to bound the integral of the Hessian over a block $B_{\bm j}$ by the integral of the Laplacian and the function itself over the larger block $\tilde B_{\bm j}$:
\begin{align*}
	\int_{B_{\bm j}} \big\|\nabla_{\bm v\bm v}^2 \phi_0(\bm x,\bm v)\big\|^2
	\D\bm v & \Le C_d \int_{\tilde B_{\bm j}} \Big(|\Delta_{\bm v}
	\phi_0(\bm x,\bm v)|^2 + |\phi_0(\bm x,\bm v)|^2\Big)
	\D\bm v \\
	& \Le \frac{C_d}{m^2} \int_{\tilde B_{\bm j}} (|\bm x|^2+|\bm v|^2+d)\D\bm v,
\end{align*}
where we have used the growth bound \eqref{Laplace bound} for the Laplacian and a similar bound for the solution $\phi_0$.

For any $\bm v \in \tilde B_{\bm j}$, its squared norm is related to the squared norm of the center point $\bm j$. This relationship allows us to bound the Gaussian weight within each block. For $\bm v \in B_{\bm j}$, we can bound the exponential term from below in terms of $|\bm j|^2$, which simplifies the integral. Combining these estimates, we can bound the contribution to $\mathcal Q(\bm x)$ from each block $B_{\bm j}$:
\begin{align*}
	& \quad \bigg(\frac{M_2}{2\pi}\bigg)^{\frac d2}
	\int_{B_{\bm j}} \big\|\nabla_{\bm v\bm v}^2 \phi_0(\bm x,\bm v)\big\|^2 e^{-\frac{M_2}2|\bm v|^2}
	\D\bm v \\
	& \Le C_d \exp\bigg(-\frac{M_2}4\sum_{i=1}^dj_i^2\bigg) \int_{B_{\bm j}} \big\|\nabla_{\bm v\bm v}^2 \phi_0(\bm x,\bm v)\big\|^2
	\D\bm v \\
	& \Le \frac{C_d}{m^2} \exp\bigg(-\frac{M_2}4\sum_{i=1}^dj_i^2\bigg)
	\int_{\tilde B_{\bm j}} (|\bm x|^2+|\bm v|^2+d)\D\bm v \\
	& \Le \frac{C_d}{m^2} \exp\bigg(-\frac{M_2}4\sum_{i=1}^dj_i^2\bigg)
	\bigg(|\bm x|^2 + \sum_{i=1}^d j_i^2 + C_d \bigg) \\
	& \Le \frac{C_d}{m^2}
	\exp\bigg(-\frac{M_2}8\sum_{i=1}^dj_i^2\bigg) (|\bm x|^2+1).
\end{align*}
Summing over all blocks $\bm j\in\mathbb Z^d$ yields a bound on $\mathcal Q(\bm x)$. The sum over the exponential terms is a convergent geometric series, resulting in an upper bound for $\mathcal Q(\bm x)$ that exhibits at most quadratic growth in $|\bm x|$:
\begin{align*}
	\mathcal Q(\bm x) & = \sum_{\bm j\in\mathbb Z^d}
	\bigg(\frac{M_2}{2\pi}\bigg)^{\frac d2}
	\int_{B_{\bm j}} \big\|\nabla_{\bm v\bm v}^2 \phi_0(\bm x,\bm v)\big\|^2 e^{-\frac{M_2}2|\bm v|^2}
	\D\bm v \\
	& \Le \frac{C_d}{m^2} (|\bm x|^2+1) \sum_{\bm j\in\mathbb Z^d}
	\exp\bigg(-\frac{M_2}8\sum_{i=1}^dj_i^2\bigg) \Le \frac{C_d}{m^2}(|\bm x|^2+1).
\end{align*}
Finally, taking the expectation of this inequality with respect to $\pi(\bm x)$ yields the desired result. Since the potential is strongly convex, $\pi(\bm x)$ has finite second moments, giving:
\begin{equation*}
	\mathbb E^{\pi}\big[\mathcal Q(\bm x)\big] \Le
	\frac{C_d}{m^2} \mathbb E^\pi[|\bm x|^2+1] \Le 
	\frac{C_d}{m^3}.
\end{equation*}
This completes the proof of Lemma~\ref{lemma: paradox}.
\end{proofof}

The idea of using hypercubes to derive the $L^2$-estimate comes from Xu'an Dou.

\section{Proofs for MSE and numerical bias of SVRG-UBU}
\label{appendix: proof SVRG-UBU}

The analysis of SVRG-UBU follows the same structure as that for mini-batch SG-UBU, establishing stability through uniform-in-time moment bounds and consistency via local error estimates. The discrete Poisson equation is then applied to derive the MSE and numerical bias. A critical distinction, however, is that the SVRG-UBU solution $(\bm X_k,\bm V_k)_{k=0}^\infty$ is non-Markovian. To address this, we define the invariant distribution $\pi_h(\bm x,\bm v)$ with respect to the embedded Markovian subsequence $(\bm X_{kN/p},\bm V_{kN/p})_{k\Ge0}$. This subsequence is constructed from the anchor positions, which are updated at the end of each epoch.

Throughout this section, we let $q = N/p$ denote the number of inner iterations in each epoch. To formulate the SVRG-UBU integrator compactly, we define a solution flow for the variance-reduced gradient update, analogous to $\Phi_t^{\mathfrak B}$ and $\Phi_t^{\mathfrak B}(\theta)$. For a given anchor position $\bm x_* \in \mathbb R^d$, this flow is defined as
\begin{equation*}
	\Phi_t^{\mathfrak B}(\bm x_*,\theta) : \left\{
	\begin{aligned}
		\bm{x}_t &= \bm{x}_0, \\
		\bm{v}_t &= \bm{v}_0 - \frac{t}{M_2}
		\bigg(\frac1p \sum_{i\in\theta}
		\nabla U_i(\bm x_0) - \frac1p \sum_{i\in\theta} 
		\nabla U_i(\bm x_*) + \nabla U(\bm x_*)
		\bigg).
	\end{aligned}
	\right.
\end{equation*}
Then SVRG-UBU from Algorithm~\ref{algorithm: SVRG-UBU} can be expressed as a single update rule. Since the anchor position is recomputed every $q$ steps, the integrator takes the form:
\begin{equation}
	\bm Z_{k+1} = \big(\Phi^{\mathfrak U}_{h/2} \circ
	\Phi^{\mathfrak B}_h(\bm Y_{q[k/q]},\theta_k) \circ \Phi^{\mathfrak U}_{h/2}\big) \bm Z_k,\qquad k=0,1,\cdots,
	\label{SVRG-UBU}
\end{equation}
where $\bm Y_k$ is the position component of $\Phi_{h/2}^{\mathfrak U} \bm Z_k$. The term $\bm Y_{q[k/q]}$ represents the anchor position established at the beginning of the current epoch, which remains fixed for $q$ iterations. The random subsets $(\theta_k)_{k=0}^\infty$ of size $p$ are drawn independently at each step.

\subsection{Uniform-in-time moment bound of SVRG-UBU}

Our first step is to establish a uniform-in-time moment bound for the SVRG-UBU solution, which is a prerequisite for proving the existence of an invariant distribution for its embedded Markovian subsequence. Following the analytical approach used for SG-UBU, we begin by analyzing the corresponding BU-type integrator, SVRG-BU. Its update rule is given by
\begin{equation}
	\bm Z_{k+1} = \big(\Phi^{\mathfrak U}_{h} \circ
	\Phi^{\mathfrak B}_h(\bm X_{q[k/q]},\theta_k) \big) \bm Z_k,\qquad k=0,1,\cdots.
    \label{SVRG-BU update rule}
\end{equation}
The analysis of SVRG-BU closely parallels that of the standard SG-BU integrator. By directly adapting the proof of Lemma~\ref{lemma: SG-BU Lyapunov}, we establish an analogous Lyapunov condition for SVRG-BU, as stated below.
\begin{lemma}
	\label{lemma: SVRG-BU Lyapunov}
	Under Assumptions~\ref{asP} and \ref{asSgNp}, let the step size $h\Le \frac14$; for any initial state $\bm z_0 \in \mathbb R^{2d}$ and the anchor position $\bm x_* \in \mathbb R^d$,
	\begin{equation*}
		\mathbb E\Big[\mathcal V\big((\Phi^{\mathfrak U}_{h} \circ \Phi^{\mathfrak B}_h(\bm x_*,\theta)\big) \bm z_0)\Big] \Le \bigg(1-\frac{mh}{3M_2}\bigg) \mathcal V(\bm z_0) + \frac{Cd^4h}{m^3},
	\end{equation*}
	where $\theta$ is a random subset of $\{1,\cdots,N\}$ of size $p$, and $C$ depends only on $M_1,M_2$.
\end{lemma} 
\begin{proof}
\revised{The proof proceeds by verifying that the SVRG estimator $b(\bm x,\bm x_*,\theta)$ is unbiased and that its error carries the eighth moment required by Lemma~\ref{lemma: SG-BU Lyapunov}. The almost-sure growth bound of Assumption~\ref{asSg} is not used here.}

First, the unbiasedness property, $\mathbb{E}_\theta[b(\bm x,\bm x_*,\theta)] = \nabla U(\bm x)$, is readily verified from the definition of the estimator. To check the eighth moment condition in Assumption~\ref{asSg}, we apply the triangle inequality:
\begin{equation*}
	|b(\bm x,\bm x_*,\theta) - \nabla U(\bm x)|
	\Le
	\bigg|
	\frac1p \sum_{i\in\theta} \nabla U_i(\bm x) - 
	\nabla U(\bm x)
	\bigg| + \bigg|
	\frac1p \sum_{i\in\theta} \nabla U_i(\bm x_*) - \nabla U(\bm x_*)
	\bigg|.
\end{equation*}
Computing the eighth moment of both sides yields
\begin{align*}
	& ~~~~ \mathbb E_\theta\big[
	|b(\bm x,\bm x_*,\theta) - \nabla U(\bm x)|^8
	\big] \\
    & \Le C\, \mathbb E_\theta\Bigg[
	\bigg|
	\frac1p \sum_{i\in\theta} \nabla U_i(\bm x) - 
	\nabla U(\bm x)
	\bigg|^8 + 
	\bigg|
	\frac1p \sum_{i\in\theta} \nabla U_i(\bm x_*) - 
	\nabla U(\bm x_*)
	\bigg|^8
	\Bigg] \\
    & \Le \frac CN\sum_{i=1}^N |\nabla U_i(\bm x) - \nabla U(\bm x)|^8 + \frac CN\sum_{i=1}^N |\nabla U_i(\bm x_*) - \nabla U(\bm x_*)|^8 \Le C(\sigma^2 d)^4,
\end{align*}
where we have applied Jensen's inequality. This exactly produces the first two conditions of Assumption~\ref{asSg}
\begin{equation*}
    \Big(\mathbb E_\theta\big[|b(\bm x,\bm x_*,\theta) - \nabla U(\bm x)|^8\big]\Big)^{\frac14} \Le C\sigma^2d
\end{equation*}
for an arbitrary anchor position $\bm x_*$.
Therefore, Lemma~\ref{lemma: SG-BU Lyapunov} establishes the Lyapunov condition of the SVRG-BU integrator, regardless of the choice of anchor positions.
\end{proof}
The Lyapunov condition established for SVRG-BU in Lemma~\ref{lemma: SVRG-BU Lyapunov} provides the foundation for bounding the moments of the SVRG-UBU solution. Since the SVRG-UBU integrator differs from its BU-type counterpart only by the stable $\Phi_{h/2}^{\mathfrak U}$ mapping, their uniform-in-time moment bounds are equivalent. This leads directly to the following theorem.
\begin{theorem}
\label{theorem: SVRG-UBU moment bound}
Under Assumptions~\ref{asP} and \ref{asSgNp}, let the step size $h\Le \frac14$, and
$$
	\bm Z_0 = (\bm X_0,\bm V_0)\text{~with~}\bm X_0 = \bm 0\text{~and~}\bm V_0 \sim \mathcal N(\bm 0,M_2^{-1}\bm I_d);
$$
then the SVRG-UBU solution $(\bm X_k,\bm V_k)_{k=0}^\infty$ satisfies
\begin{equation*}
    \sup_{k\Ge 0} \mathbb E\big[|\bm Z_k|^8\big] \Le \frac{Cd^4}{m^4},
\end{equation*}
where the constant $C$ depends only on $M_1,M_2$.
\end{theorem}

By applying the same techniques from the proof of Lemma~\ref{lemma: SG-UBU invariant}, we can prove the existence of and establish a moment bound for the invariant distribution of SVRG-UBU's embedded Markovian subsequence. This subsequence, $(\bm X_{kq},\bm V_{kq})_{k=0}^\infty$, is formed by the states at the end of each epoch (recall that the anchor position is updated every $q = N/p$ iterations). The result is described in the following lemma.
\begin{lemma}
\label{lemma: SVRG-UBU invariant}
Under Assumptions~\ref{asP} and \ref{asSgNp}, let the step size $h \Le \frac14$; the SVRG-UBU subsequence $(\bm X_{kq},\bm V_{kq})_{k=0}^\infty$ has an invariant distribution $\pi_h(\bm x,\bm v)$ in $\mathbb R^{2d}$, which satisfies
\begin{equation*}
	\int_{\mathbb R^{2d}} |\bm z|^8 \pi_h(\bm z)
	\D\bm z \Le \frac{Cd^4}{m^4},
\end{equation*}
where the constant $C$ depends only on $M_1,M_2$.
\end{lemma}
\begin{proof}
To formalize the analysis, we first define a semigroup for the variance-reduced gradient update that explicitly depends on a fixed anchor position $\bm x_*\in\mathbb R^d$. Analogous to the previous definitions, we have
\begin{equation*}
	P_t^{\mathfrak B}(\bm x_*) \mu := \mathbb E_{\theta}\big[\text{distribution of }\Phi^{\mathfrak B}_t(\bm x_*,\theta) \bm z_0 \text{ with } \bm z_0 \sim \mu\big].
\end{equation*}

The Lyapunov condition established in Lemma~\ref{lemma: SVRG-BU Lyapunov}, combined with the Krylov--Bogolyubov theorem, guarantees that the SVRG-BU subsequence $(\bm X_{kq}, \bm V_{kq})_{k=0}^\infty$ admits an invariant distribution $\hat \pi_h(\bm x,\bm v)$ in $\mathbb R^{2d}$. This distribution satisfies the moment bound:
\begin{equation}
	\int_{\mathbb R^{2d}} |\bm z|^8\hat \pi_h(\bm z)
	\D\bm z \Le \frac{Cd^4}{m^4}.
    \label{proof: eighth moment bound}
\end{equation}
By definition of the SVRG-BU integrator, an epoch consists of $q = N/p$ inner iterations where the anchor point is fixed. The invariance of $\hat\pi_h$ across these epochs is expressed by the relation
\begin{equation}
	\int_{\mathbb R^{2d}}
	\Big(\big(P_h^{\mathfrak U} P_h^{\mathfrak B}(\bm x)\big)^q \delta_{\bm x,\bm v}\Big)\hat \pi_h(\bm x,\bm v)\D \bm x\D\bm v = \hat \pi_h.
    \label{proof: hat pi_h}
\end{equation}
Here, the integral signifies an expectation over the invariant distribution $\hat\pi_h$. For each state $(\bm x, \bm v)$ drawn from $\hat\pi_h$, its position component $\bm x$ becomes the anchor for an entire epoch of $q$ steps. The use of the Dirac measure $\delta_{\bm x,\bm v}$ indicates that this epoch starts from the state $(\bm x, \bm v)$ itself, ensuring the anchor position remains fixed for all subsequent $\mathfrak B$ steps.

Next, we construct the invariant distribution for the SVRG-UBU subsequence, $\pi_h(\bm x, \bm v)$, using the invariant distribution of the SVRG-BU subsequence, $\hat\pi_h$. \reviewed{Let the SVRG-BU epoch start from $(\bm x,\bm v)\sim\hat\pi_h$, set $\bm Z_1 = \big(\Phi^{\mathfrak U}_{h/2}\circ\Phi^{\mathfrak B}_h(\bm x,\theta_0)\big)(\bm x,\bm v)$, and let $(\bm Z_k)_{k=1}^q$ be the SVRG-UBU solution continued from $\bm Z_1$ with the same anchor $\bm x$ and the same Brownian motion, so that $\Phi^{\mathfrak U}_{h/2}\bm Z_k$ is the SVRG-BU state after $k$ steps. We define $\pi_h$ to be the law of $\bm Z_q$:}
\begin{equation}
	\pi_h = \int_{\mathbb R^{2d}}
	\reviewed{\Big(\big(P_{h/2}^{\mathfrak U}P_h^{\mathfrak B}(\bm x)P_{h/2}^{\mathfrak U}\big)^{q-1}
	\big(P_{h/2}^{\mathfrak U} P_h^{\mathfrak B}(\bm x)\big) \delta_{\bm x,\bm v}\Big)}\hat \pi_h(\bm x,\bm v)\D\bm x\D\bm v.
    \label{proof: pi_h}
\end{equation}
\reviewed{The rightmost factor $P_{h/2}^{\mathfrak U} P_h^{\mathfrak B}(\bm x)$ carries the SVRG-BU state $(\bm x,\bm v)$ to $\bm Z_1$, and the remaining $q-1$ factors are full SVRG-UBU steps. Note that the anchor $\bm x$ is the position component of the state drawn from $\hat\pi_h$, and stays fixed throughout the epoch.}
Then $\pi_h$ inherits the same moment bound given in \eqref{proof: eighth moment bound}.

\reviewed{Applying one further half-step of the $\mathfrak U$ flow to \eqref{proof: pi_h} recovers the SVRG-BU state after $q$ steps, whose law is $\hat\pi_h$ by \eqref{proof: hat pi_h}:
\begin{equation*}
    P_{h/2}^{\mathfrak U} \pi_h = \int_{\mathbb R^{2d}}
	\Big(\big(P_h^{\mathfrak U} P_h^{\mathfrak B}(\bm x)\big)^{q} \delta_{\bm x,\bm v}\Big)\hat \pi_h(\bm x,\bm v)\D\bm x\D\bm v = \hat\pi_h,
\end{equation*}
where the semigroup property $P_{h/2}^{\mathfrak U}P_{h/2}^{\mathfrak U} = P_h^{\mathfrak U}$ has been used to merge the half-steps.}
Therefore, substituting $\hat\pi_h = P_{h/2}^{\mathfrak U}\pi_h$ into \eqref{proof: pi_h} gives
\begin{equation*}
\pi_h = \int_{\mathbb R^{2d}}
\Big(\big(P_{h/2}^{\mathfrak U}P_h^{\mathfrak B}(\bm x)P_{h/2}^{\mathfrak U}\big)^{q-1} \big(P_{h/2}^{\mathfrak U}P_h^{\mathfrak B}(\bm x)\big) \delta_{\bm x,\bm v}\Big) (P_{h/2}^{\mathfrak U} \pi_h)(\bm x,\bm v)\D\bm x\D\bm v
\end{equation*}
showing $\pi_h$ is invariant for the SVRG-UBU subsequence $(\bm X_{kq},\bm V_{kq})_{k=0}^\infty$.
\end{proof}

\subsection{Local error analysis of SVRG-UBU}

We now turn to the local error analysis for the SVRG-UBU integrator. The discretization error is identical to that for SG-UBU, as this term is independent of the gradient estimator. Our focus therefore shifts to bounding the stochastic gradient error, which is where the effect of the variance reduction mechanism becomes apparent. The following lemma, analogous to Lemma~\ref{lemma: stochastic gradient error}, establishes the corresponding fourth moment bounds.
\begin{lemma}
	\label{lemma: SVRG-UBU stochastic gradient error}
	Under Assumptions~\ref{asP} and \ref{asSgNp}, let the step size $h\Le \frac14$, and
	$$
	\bm Z_0 = (\bm X_0,\bm V_0)\text{~with~}\bm X_0 = \bm 0\text{~and~}\bm V_0 \sim \mathcal N(\bm 0,M_2^{-1}\bm I_d);
	$$
	then the stochastic gradient error of SVRG-UBU satisfies
	\begin{align*}
		\sup_{k\Ge0}\sqrt{\mathbb E\big|\bm X_{k+1} - \bar{\bm X}_{k+1}\big|^4} &
		\Le \frac{C\sigma^2dh^4}{p}\min\bigg\{1,\frac{N^2h^2}{mp^2}\bigg\}, \\
		\sup_{k\Ge0}\sqrt{\mathbb E\big|\bm V_{k+1} - \bar{\bm V}_{k+1}\big|^4} &
		\Le \frac{C\sigma^2dh^2}{p}\min\bigg\{1,\frac{N^2h^2}{mp^2}\bigg\},
	\end{align*}
	where the constant $C$ depends only on $M_1, M_2$.
\end{lemma}

\begin{proof}
Analogous to the expression for SG-UBU, the stochastic gradient error components of the SVRG-UBU integrator are given by:
\begin{equation}
\bm X_{k+1} - \bar {\bm X}_{k+1} = -\frac{h(1-e^{-h})}{2M_2} \bm g_k(\theta_k), \qquad
\bm V_{k+1} - \bar {\bm V}_{k+1} = -\frac{he^{-h}}{M_2}
\bm g_k(\theta_k),
\label{proof: X difference}
\end{equation}
where $\bm g_k(\theta_k) \in\mathbb R^d$ is the deviation of the SVRG estimator from the true gradient $\nabla U(\bm Y_k)$. According to \eqref{SVRG gradient error}, this deviation is explicitly written as
\begin{align}
    \bm g_k(\theta_k) & = \bigg(\frac1p \sum_{i\in\theta_k} \nabla U_i(\bm Y_k)
     - \frac1p \sum_{i\in\theta_k} \nabla U_i(\bm Y_*) +
     \nabla U(\bm Y_*)\bigg) - \nabla U(\bm Y_k) \notag \\
    & = \int_0^1 \bigg(
     \frac1p \sum_{i\in\theta_k} \nabla^2 U_i - \nabla^2 U
     \bigg)\big(\lambda \bm Y_k+(1-\lambda) \bm Y_*\big)\D\lambda\,
     \revised{(\bm Y_k - \bm Y_*)},
     \label{proof: gg 0}
\end{align}
where $\bm Y_* = \bm Y_{q[k/q]}$ is the anchor position. 

To bound the fourth moment of $\bm g_k(\theta_k)$, we leverage its integral representation. For a fixed $\lambda \in [0,1]$, we define a set of vectors:
\begin{equation*}
    a_i(\lambda) = \big(\nabla^2 U_i - \nabla^2 U
    \big)\big(\lambda \bm Y_k+(1-\lambda) \bm Y_*\big)\revised{(\bm Y_k - \bm Y_*)},\qquad
    i=1,\cdots,N.
\end{equation*}
This sequence of vectors has a zero mean, namely, $\frac{1}{N}\sum_{i=1}^N a_i(\lambda) = \bm 0$. We can bound the average eighth moment of $a_i(\lambda)$ using Assumption~\ref{asSgNp}:
\begin{equation*}
    \frac1N\sum_{i=1}^N |a_i(\lambda)|^8 \Le 
    |\bm Y_k - \bm Y_*|^8 \cdot \frac{C}N \sum_{i=1}^N 
    \big\|(\nabla^2 U_i-\nabla^2 U)(\lambda \bm Y_k+(1-\lambda)\bm Y_*)\big\|^8 \Le C\sigma^8 |\bm Y_k - \bm Y_*|^8.
\end{equation*}
This confirms that the conditions of Lemma~\ref{lemma: eighth moment} are met. Applying the lemma to the mini-batch average of these vectors gives:
\begin{equation}
    \mathbb E_{\theta_k} \Bigg[
    \bigg|
    \frac1p \sum_{i\in\theta_k} a_i(\lambda)
    \bigg|^8 \Bigg]\Le \frac{C\sigma^8}{p^4} |\bm Y_k - \bm Y_*|^8.
    \label{proof: gg 1}
\end{equation}
By integrating over $\lambda \in [0,1]$ and applying Jensen's inequality, we obtain a bound on the eighth moment of $\bm g_k(\theta_k)$, which implies a bound on its fourth moment:
\begin{equation}
    \mathbb E_{\theta_k} \big[|\bm g_k(\theta_k)|^8\big]
    \Le \frac{C\sigma^8}{p^4} |\bm Y_k - \bm Y_*|^8 \quad \Longrightarrow \quad
    \mathbb E_{\theta_k} \big[|\bm g_k(\theta_k)|^4\big]
    \Le \frac{C\sigma^4}{p^2} |\bm Y_k - \bm Y_*|^4.
    \label{proof: gg 2}
\end{equation}
Taking the full expectation over the numerical solution yields:
\begin{equation}
    \mathbb E \big[|\bm g_k(\theta_k)|^4\big] 
    \Le \frac{C\sigma^4}{p^2} \mathbb E\big|\bm Y_k - \bm Y_{q[k/q]}\big|^4 \Le 
    \frac{C\sigma^4q^3}{p^2} \sum_{s=q[k/q]}^{k-1} 
    \mathbb E|\bm Y_{s+1}-\bm Y_s|^4.
    \label{proof: gg 3}
\end{equation}
The single-step deviation can be bounded using SVRG-UBU from Algorithm~\ref{algorithm: SVRG-UBU}, which implies $(\bm Y_{s+1},\bm V_{s+1}^{(1)}) = \Phi_{h}^{\mathfrak U}(\bm Y_s,\bm V_s^{(2)})$. The uniform-in-time moment bounds then give
\begin{equation*}
    \mathbb E|\bm Y_{s+1} - \bm Y_s|^4 \Le Ch^4 \big(
    \mathbb E|\bm Y_s|^4 + \mathbb E|\bm V_s^{(2)}|^4 + d^2
    \big) \Le \frac{Cd^2h^4}{m^2}.
\end{equation*}
Substituting this into \eqref{proof: gg 3} and noting that the sum has at most $q$ terms, we arrive at the final bound for the stochastic gradient deviation:
\begin{equation}
    \sqrt{\mathbb E \big[|\bm g_k(\theta_k)|^4\big]} \Le 
    \frac{C\sigma^2d q^2 h^2}{mp}.
    \label{proof: gg 4}
\end{equation}

Alternatively, a simpler bound on the stochastic gradient deviation $\bm g_k(\theta_k)$ that is independent of the distance between the sampling point $\bm Y_k$ and the anchor point $\bm Y_*$ can be obtained via the triangle inequality:
\begin{align*}
    & ~~~~ \mathbb E_{\theta_k}\big[
    |\bm g_k(\theta_k)|^8
    \big] \Le C\Bigg(\bigg|\frac1p \sum_{i\in \theta_k}
    \nabla U_i(\bm Y_k) - \nabla U(\bm Y_k)
    \bigg|^8 + \bigg|\frac1p \sum_{i\in \theta_k}
    \nabla U_i(\bm Y_*) - \nabla U(\bm Y_*)
    \bigg|^8\Bigg) \\
    & \Le \frac{C}{p^4} \bigg(\frac1N\sum_{i=1}^N \big|
    \nabla U_i(\bm Y_k) - \nabla U(\bm Y_k)\big|^8+
    \frac1N\sum_{i=1}^N \big|
    \nabla U_i(\bm Y_*) - \nabla U(\bm Y_*)\big|^8\bigg) \Le \frac{C\sigma^8d^4}{p^4}.
\end{align*}
Then we arrive at a uniform bound on the fourth moment:
\begin{equation}
    \sqrt{\mathbb E\big[
    |\bm g_k(\theta_k)|^4
    \big]}\Le \frac{C\sigma^2d}{p}.
    \label{proof: gg 5}
\end{equation}
Combining this uniform bound with the distance-dependent bound from \eqref{proof: gg 4} yields a tighter composite estimate using the $\min$ function:
\begin{equation}
    \sqrt{\mathbb E\big[
    |\bm g_k(\theta_k)|^4
    \big]} \Le 
    \frac{C\sigma^2d}{p} \min\bigg\{1,\frac{q^2h^2}{m}\bigg\} = 
    \frac{C\sigma^2d}{p} \min\bigg\{1,\frac{N^2h^2}{mp^2}\bigg\}.
    \label{proof: gg 6}
\end{equation}
Finally, substituting the comprehensive bound from \eqref{proof: gg 6} into the local error expressions in \eqref{proof: X difference} yields the result and completes the proof.
\end{proof}

\subsection{MSE of SVRG-UBU}
The uniform-in-time moment bound from Theorem~\ref{theorem: SVRG-UBU moment bound} and the local stochastic gradient error estimate from Lemma~\ref{lemma: SVRG-UBU stochastic gradient error} provide the necessary stability and consistency results for the SVRG-UBU integrator. Consequently, the MSE bound in Theorem~\ref{theorem: SVRG-UBU MSE} follows directly by applying the analytical procedure detailed in the proof of Theorem~\ref{theorem: SG-UBU MSE}.

\revised{Two points deserve mention. The anchor $\bm Y_*$ is $\mathcal G_k$-measurable, so the martingale property $\mathbb E[R_{k,1}|\mathcal G_k] = 0$ and the resulting orthogonality carry over verbatim. Moreover, Lemma~\ref{lemma: discretization error} and Lemma~\ref{lemma: displacement} are stated under Assumption~\ref{asSg}, while their proofs use only Assumption~\ref{asP}, the uniform-in-time eighth moment, and the growth bound in Assumption~\ref{asSgNp}, all of which SVRG-UBU satisfies.}

\subsection{Numerical bias of SVRG-UBU}

\begin{proofof}[Theorem~\ref{theorem: SVRG-UBU numerical bias}]
Let $q = N/p$ be the length of one epoch, and let $\pi_h(\bm x,\bm v)$ be the invariant distribution of the SVRG-UBU subsequence $(\bm X_{kq},\bm V_{kq})_{k=0}^\infty$. Let the initial state $\bm Z_0 \sim \pi_h$, then the state after one full epoch is also distributed according to $\pi_h$, i.e., $\bm Z_q \sim \pi_h$.

Using this property with the discrete Poisson equation \eqref{discrete Poisson eq}, we can express the bias of the initial state distribution as follows:
\begin{align}
    \mathbb E^{\pi_h}[f(\bm Z_0)] - \pi(f) &= \mathbb E^{\pi_h}
    \bigg[\bigg(\frac{1-e^{h\mathcal L}}{h} \phi_h\bigg)(\bm Z_0)
    \bigg] = \mathbb E^{\pi_h}\bigg[
    \frac{\phi_h(\bm Z_0) - \phi_h(\bm Z_0(h))}{h}
    \bigg] \notag \\
    & = \mathbb E^{\pi_h}\bigg[
    \frac{\phi_h(\bm Z_q) - \phi_h(\bm Z_0(h))}{h}
    \bigg].
    \label{SVRG bias 0}
\end{align}
To connect this expression to the bias of the intermediate steps within the epoch, we decompose the term $\phi_h(\bm Z_q) - \phi_h(\bm Z_0(h))$ into a sum over the $q$ steps:
\begin{align*}
	\mathbb E^{\pi_h}\bigg[
    \frac{\phi_h(\bm Z_q) - \phi_h(\bm Z_0(h))}{h}
    \bigg] & = \sum_{k=0}^{q-1}
	\mathbb E^{\pi_h} \bigg[\frac{\phi_h(\bm Z_{k+1})-\phi_h(\bm Z_k(h))}{h}\bigg] + \sum_{k=1}^{q-1}
	\mathbb E^{\pi_h} \bigg[\frac{\phi_h(\bm Z_k(h))-\phi_h(\bm Z_k)}{h}\bigg] \\
	& = \sum_{k=0}^{q-1}
	\mathbb E^{\pi_h} \bigg[\frac{\phi_h(\bm Z_{k+1})-\phi_h(\bm Z_k(h))}{h}\bigg]  + \sum_{k=1}^{q-1} \mathbb E^{\pi_h}
	\bigg[\frac{e^{h\mathcal L}\phi_h(\bm Z_k) - \phi_h(\bm Z_k)}{h}\bigg] \\
	& = \sum_{k=0}^{q-1}
	\mathbb E^{\pi_h} \bigg[\frac{\phi_h(\bm Z_{k+1})-\phi_h(\bm Z_k(h))}{h}\bigg]  - \sum_{k=1}^{q-1}
	\big(\mathbb E[f(\bm Z_k)] - \pi(f)\big).
\end{align*}
By combining this decomposition with \eqref{SVRG bias 0} and rearranging, we obtain an expression for the average bias over one epoch:
\begin{equation}
	\frac1q \sum_{k=0}^{q-1} \mathbb E[f(\bm Z_k)] - \pi(f) =
	\frac1q \sum_{k=0}^{q-1} \mathbb E^{\pi_h} \bigg[\frac{\phi_h(\bm Z_{k+1})-\phi_h(\bm Z_k(h))}{h}\bigg].
	\label{SVRG bias 1}
\end{equation}
By the definition of the averaged distribution $\tilde \pi_h$, the left-hand side is precisely the numerical bias $\tilde \pi_h(f) - \pi(f)$. This yields the final expression
\begin{equation}
    \tilde \pi_h(f) - \pi(f) = 
    \frac1q \sum_{k=0}^{q-1} \mathbb E^{\pi_h} \bigg[\frac{\phi_h(\bm Z_{k+1})-\phi_h(\bm Z_k(h))}{h}\bigg].
    \label{SVRG bias 2}
\end{equation}

Following the same decomposition strategy used in the proof of Theorem~\ref{theorem: SG-UBU numerical bias}, we separate each local error term on the right-hand side of \eqref{SVRG bias 2} into its corresponding stochastic gradient error component, $R_k$, and discretization error component, $S_k$. The expected values of these components satisfy bounds analogous to those derived for the standard SG-UBU case:
\begin{equation*}
	\bigg|\mathbb E[R_k] - \frac{h}{2M_2^2} \mathbb E_{\theta_k}^{\pi_h}
	\Big[\big(\bm g(\bm Y_k,\bm Y_0,\theta_k)\big)^\top \bm M_k \big(\bm g(\bm Y_k,\bm Y_0,\theta_k)\big)\Big]\bigg|  \Le \frac{Cdh^2}{m^2}, \quad
    \big|\mathbb E[S_k]\big| \Le 
    \frac{Cdh^2}{m^3}.
\end{equation*}
By substituting these estimates into \eqref{SVRG bias 2} and summing the error bounds over the epoch from $k=0$ to $q-1$, we arrive at the explicit characterization of the numerical bias:
\begin{equation}
		\Bigg|\tilde \pi_h(f) - \pi(f) -
		\frac{h}{2M_2^2q} \mathbb E_{\theta_0,\cdots,\theta_{q-1}}^{\pi_h}\bigg[
        \sum_{k=0}^{q-1}
		\big(\bm g(\bm Y_k,\bm Y_0,\theta_k)\big)^\top \bm M_k \big(\bm g(\bm Y_k,\bm Y_0,\theta_k)\big) \bigg]\Bigg| \Le
    \frac{Cdh^2}{m^3}.
    \label{SVRG bias 11}
\end{equation}

Next, since the numerical solution $(\bm Y_k)_{k=0}^{q-1}$ has bounded eighth moments, we can apply the bound from \eqref{proof: gg 6} to obtain the following estimate on the second moment:
\begin{equation}
    \mathbb E\Big[\big|\bm g(\bm Y_k,\bm Y_0,\theta_k)\big|^2 \Big]
    \Le \frac{C\sigma^2d}{p} \min\bigg\{1,\frac{N^2h^2}{mp^2}\bigg\}.
    \label{proof: gg 11}
\end{equation}
Finally, combining \eqref{SVRG bias 11} and \eqref{proof: gg 11} yields the inequality
\begin{equation*}
    |\tilde \pi_h(f) - \pi(f)| \Le
    C\bigg(\frac{\sigma^2dh}{m^2p} \min
		\bigg\{1,\frac{N^2h^2}{mp^2}\bigg\} +
    \frac{dh^2}{m^3}\bigg),
\end{equation*}
thus completing the proof of Theorem~\ref{theorem: SVRG-UBU numerical bias}.
\end{proofof}

\section{Proofs for MSE and numerical bias of SAGA-UBU}
\label{appendix: proof SAGA-UBU}

The MSE analysis for SAGA-UBU proceeds by establishing its stability and consistency. Stability is addressed by deriving a modified Lyapunov condition for the corresponding SAGA-BU integrator, while consistency is analyzed by bounding the moments of the stochastic gradient error. A key challenge is the non-Markovian nature of the SAGA-UBU solution $(\bm X_k, \bm V_k)_{k=0}^\infty$, which arises from its dependence on the evolving table of historical positions $\{\bm\Phi_i\}_{i=1}^N$. The position-velocity marginal $(\bm X_k,\bm V_k)_{k=0}^\infty$ is therefore not itself a Markov chain, so no invariant distribution of that marginal chain is available. Consequently, the numerical bias, $\tilde \pi_h(f) - \pi(f)$, is measured with respect to a limit distribution $\tilde \pi_h(\bm x,\bm v)$, defined as the subsequential limit of the time-averaged distributions of the SAGA-UBU trajectory.

To formulate a compact representation for the SAGA-UBU integrator, we first define a solution flow that incorporates the SAGA gradient estimator. For a given table of historical positions $\{\bm\Phi_i\}_{i=1}^N$ and a single randomly drawn index $\theta \in \{1,\dots,N\}$ (corresponding to the batch size $p=1$ case), this flow is defined as:
\begin{equation*}
	\Phi_t^{\mathfrak B}(\{\bm\Phi_i\}_{i=1}^N,\theta) : \left\{
	\begin{aligned}
		\bm{x}_t &= \bm{x}_0, \\
		\bm{v}_t &= \bm{v}_0 - \frac{t}{M_2}
		\bigg(\nabla U_\theta(\bm x_0) - \nabla U_\theta(\bm \Phi_\theta) + \frac1N \sum_{i=1}^N \nabla U_i(\bm \Phi_i)
		\bigg).
	\end{aligned}
	\right.
\end{equation*}
Then the SAGA-UBU integrator can be expressed as a two-part update rule for each iteration $k$. The first equation describes the evolution of the state $\bm Z_k$ using the UBU splitting scheme, while the second equation describes the update to the historical position table for the sampled index $\theta_k$:
\begin{equation}
\left\{
\begin{aligned}
	\bm Z_{k+1} & = \big(\Phi^{\mathfrak U}_{h/2} \circ
	\Phi^{\mathfrak B}_h(\{\bm\Phi_i\}_{i=1}^N,\theta_k) \circ \Phi^{\mathfrak U}_{h/2}\big) \bm Z_k, \\
    \bm\Phi_{\theta_k} & = \bm Y_k,
\end{aligned}
\right.
\qquad k=0,1,\cdots,
	\label{SAGA-UBU}
\end{equation}
where the historical positions $\{\bm\Phi_i\}_{i=1}^N$ are initialized at $\bm Y_0$, and $(\theta_k)_{k=0}^\infty$ are selected independently and uniformly from $\{1,\cdots,N\}$ at each step.

Specifically, at step $k$ of SAGA-UBU, the historical position $\bm \Phi_i$ for each component $i$ is set to the intermediate state $\bm Y_{k_i}$, where $k_i$ is the index of the most recent iteration prior to $k$ in which the gradient $\nabla U_i(\bm x)$ was evaluated. This relationship is defined formally as
\begin{equation}
	\bm{\Phi}_i = \bm{Y}_{k_i}, \qquad
	k_i = \max\Big\{0 \Le l \Le k-1: \nabla U_i(\bm{x}) \text{ was computed at step } l\Big\}.
	\label{Phi definition}
\end{equation}
By convention, if the gradient for a component $i$ has not been computed in any of the steps from $1$ to $k-1$, its historical index is taken to be the initial step, $k_i=0$.

\subsection{Uniform-in-time moment bound of SAGA-UBU}

Following our analytical strategy, we begin by examining the SAGA-BU integrator, which is the BU-type counterpart to SAGA-UBU. Its update rule, similar to \eqref{SVRG-BU update rule}, is given by
\begin{equation}
\left\{
\begin{aligned}
	\bm Z_{k+1} & = \big(\Phi^{\mathfrak U}_{h} \circ
	\Phi^{\mathfrak B}_h(\{\bm\Phi_i\}_{i=1}^N,\theta_k)\big) \bm Z_k, \\
    \bm\Phi_{\theta_k} & = \bm X_k,
\end{aligned}
\right.
\qquad k=0,1,\cdots,
	\label{SAGA-BU update rule}
\end{equation}
where the table of historical positions $\{\bm\Phi_i\}_{i=1}^N$ is initialized with $\bm X_0$. At each step $k$, the historical position for a given component $i$ is determined by
\begin{equation}
	\bm{\Phi}_i = \bm{X}_{k_i}, \qquad
	k_i = \max\Big\{0 \Le l \Le k-1: \nabla U_i(\bm{x}) \text{ was computed at step } l\Big\}.
	\label{Phi definition 2}
\end{equation}
To establish a uniform-in-time moment bound, we derive a Lyapunov condition for this integrator. Unlike the condition for SVRG-BU, the non-Markovian nature of SAGA introduces a dependency on the maximum moment over the entire history of the process. This modified Lyapunov condition is stated in the following lemma.
\begin{lemma}
	\label{lemma: SAGA-BU Lyapunov}
	Under Assumptions~\ref{asP} and \ref{asSgNp}, let the step size $h\Le \frac14$; then the SAGA-BU solution $(\bm X_k,\bm V_k)_{k=0}^\infty$ with the batch size $1$ satisfies 
	\begin{equation*}
		\mathbb E\big[\mathcal V(\bm Z_{k+1})\big] \Le
		\bigg(1-\frac{mh}{3M_2}\bigg) \mathbb E\big[\mathcal V(\bm Z_k)\big] + \frac{Cd^4h}{m^3} + \frac{Ch^5}{m^3} \max_{0\Le l\Le k} \mathbb E\big[\mathcal V(\bm Z_l)\big],
	\end{equation*}
	where the constant $C$ depends only on $M_1,M_2$.
\end{lemma} 
\begin{proof}
The Lyapunov condition for the SAGA-BU integrator can be validated by adapting the proof for the full-gradient case from Lemma~\ref{lemma: FG-BU Lyapunov}. The core of the argument is to analyze the stochastic gradient error, which is the difference between the SAGA-BU and FG-BU updates, namely, $\bm Z_{k+1}$ and $\bar{\bm Z}_{k+1}$. This error can be written as
\begin{equation}
\bm X_{k+1} - \bar {\bm X}_{k+1} = -\frac{h(1-e^{-2h})}{2M_2} \bm g_k(\theta_k), \qquad
\bm V_{k+1} - \bar {\bm V}_{k+1} = -\frac{he^{-2h}}{M_2}
\bm g_k(\theta_k),
\end{equation}
where the deviation term $\bm g_k(\theta_k) \in\mathbb R^d$ is explicitly defined as
\begin{equation}
	\bm g_k(\theta_k) = \nabla U_{\theta_k}(\bm X_k) - \nabla U_{\theta_k}(\bm\Phi_{\theta_k}) + \frac1N \sum_{i=1}^N \nabla U_i(\bm\Phi_i) - \frac1N \sum_{i=1}^N \nabla U_i(\bm X_k).
\end{equation}
The analysis then proceeds by bounding the moments of this deviation term.

To bound the moments of $\bm g_k(\theta_k) \in\mathbb R^d$, we first note that $\bm g_k(\theta_k)$ is unbiased with respect to $\theta_k$ for a fixed history. Applying the triangle inequality gives
\begin{equation}
		|\bm g_k(\theta_k)| \Le \bigg|\nabla U_{\theta_k}(\bm X_k)- \frac1N \sum_{i=1}^N \nabla U_i(\bm X_k)\bigg| + |\nabla U_{\theta_k} (\bm \Phi_{\theta_k})| + \frac1N \sum_{i=1}^N |\nabla U_i(\bm \Phi_i)|.
\label{SAGA: gk}
\end{equation}
Taking the eighth moment and the full expectation over the numerical solution $(\bm X_k,\bm V_k)_{k=0}^\infty$ and the random index $(\theta_k)_{k=0}^\infty$ in \eqref{SAGA: gk}, we obtain
\begin{align*}
    \mathbb E|\bm g_k(\theta_k)|^8 & \Le C\bigg( 
    \mathbb E \bigg|\nabla U_{\theta_k}(\bm X_k)- \frac1N \sum_{i=1}^N \nabla U_i(\bm X_k)\bigg|^8 + 
    \mathbb E|\nabla U_{\theta_k} (\bm \Phi_{\theta_k})|^8 + \frac1N \sum_{i=1}^N \mathbb E|\nabla U_i(\bm \Phi_i)|^8 \bigg) \\
    & \Le C\bigg((\sigma^2d)^4 + \frac2N \sum_{i=1}^N 
    \mathbb E|\nabla U_i(\bm \Phi_i)|^8
    \bigg) \Le C\bigg(d^4 + \frac1N \sum_{i=1}^N \mathbb E\big(|\bm \Phi_i|^2+d\big)^4\bigg) \\
    & \Le C\bigg(\frac1N \sum_{i=1}^N \mathbb E|\bm\Phi_i|^8 + d^4\bigg).
\end{align*}
This chain of inequalities provides the key moment bound for the deviation term:
\begin{equation}
    \mathbb E|\bm g_k(\theta_k)|^8 \Le C\bigg(\frac1N \sum_{i=1}^N \mathbb E|\bm\Phi_i|^8 + d^4\bigg).
    \label{SAGA: moment bound 1}
\end{equation}

To establish the Lyapunov condition, we follow the analysis in the proof for Lemma~\ref{lemma: SG-BU Lyapunov}. We first define the scaled stochastic gradient error at step $k$ as $\bm\Delta_{k} = h^{-1}(\bm Z_{k+1} - \bar{\bm Z}_{k+1}) \in \mathbb R^{2d}$. To manage the dependency on the algorithm's history, we also introduce a constant $C_k$ that tracks the maximum normalized moment up to step $k$:
\begin{equation*}
	C_k = \frac1{d^4} \max_{0\Le l\Le k} \mathbb E\big[\mathcal V(\bm Z_l)\big] = \frac1{d^4} \max_{0\Le l\Le k} \mathbb E\big[(\bm Z_l^\top \bm S \bm Z_l)^4\big], \qquad k=0,1,\cdots.
\end{equation*}
The error term $\bm\Delta_{k}$ is unbiased. \reviewed{Since $\bm X_l$ is determined by $\theta_0,\cdots,\theta_{l-1}$ while $\bm 1_{\{k_i=l\}}$ depends only on $\theta_l,\cdots,\theta_{k-1}$, the random variables $\bm X_l$ and $\bm 1_{\{k_i=l\}}$ are independent, and
\begin{equation*}
	\frac1N \sum_{i=1}^N \mathbb E|\bm\Phi_i|^8
	= \frac1N \sum_{i=1}^N \sum_{l=0}^{k-1}\mathbb E|\bm X_l|^8\, \mathbb P(k_i=l)
	\Le \max_{0\Le l\Le k-1} \mathbb E|\bm X_l|^8 \Le Cd^4(C_k+1).
\end{equation*}
Consequently the moment bound on $\bm g_k(\theta_k)$ from \eqref{SAGA: moment bound 1} yields a bound on $\bm\Delta_k$ in terms of the historical maximum moment $C_k$:}
\begin{equation}
	\mathbb E|\bm\Delta_k|^8 \Le Cd^4(C_k+1).
    \label{SAGA: moment bound 2}
\end{equation}
We now expand the eighth moment of $\bm Z_{k+1}$ using its decomposition $\bm Z_{k+1} = \bar{\bm Z}_{k+1} + h\bm\Delta_k$. This expansion, analogous to the one in the proof of Lemma~\ref{lemma: SG-BU Lyapunov}, yields several terms:
	\begin{align*}
		\mathbb E\big|\sqrt{\bm S} \bm Z_{k+1}\big|^8 & =
		\mathbb E\big|\sqrt{\bm S} \bar{\bm Z}_{k+1} + h \sqrt{\bm S} \bm\Delta_{k}\big|^8 \\
		& \revised{\Le}\, \mathbb E\big|\sqrt{\bm S} \bar{\bm Z}_{k+1}\big|^8 +
		28h^2\mathbb E\Big(\big|\sqrt{\bm S} \bar{\bm Z}_{k+1}\big|^6\big|\sqrt{\bm S}\bm \Delta_{k}\big|^2\Big) \\
		& \hspace{1cm} + \revised{56h^3\mathbb E\Big(\big|\sqrt{\bm S} \bar{\bm Z}_{k+1}\big|^5\big|\sqrt{\bm S}\bm \Delta_{k}\big|^3\Big)} +
		70h^4\mathbb E\Big(\big|\sqrt{\bm S} \bar{\bm Z}_{k+1}\big|^4\big|\sqrt{\bm S}\bm \Delta_{k}\big|^4\Big) \\
		& \hspace{1cm} + \revised{56h^5\mathbb E\Big(\big|\sqrt{\bm S} \bar{\bm Z}_{k+1}\big|^3\big|\sqrt{\bm S}\bm \Delta_{k}\big|^5\Big)} +
		28h^6\mathbb E\Big(\big|\sqrt{\bm S} \bar{\bm Z}_{k+1}\big|^2\big|\sqrt{\bm S}\bm \Delta_{k}\big|^6\Big) \\
		& \hspace{1cm} + \revised{8h^7\mathbb E\Big(\big|\sqrt{\bm S} \bar{\bm Z}_{k+1}\big|\big|\sqrt{\bm S}\bm \Delta_{k}\big|^7\Big)} +
		h^8 \mathbb E\big|\sqrt{\bm S} \bm \Delta_{k}\big|^8 \\
		& \Le \mathbb E\big|\sqrt{\bm S} \bar{\bm Z}_{k+1}\big|^8 + \frac{mh}{12M_2}
		\mathbb E\big|\sqrt{\bm S} \bar{\bm Z}_{k+1}\big|^8 + \frac{Ch^5}{m^3} \mathbb E|\bm\Delta_{k}|^8,
	\end{align*}
where we have utilized the binomial theorem. \revised{The term linear in $\bm\Delta_{k}$ is dropped from the binomial expansion because $\mathbb E[\bm\Delta_{k}|\mathcal F_k]=\bm 0$, and the remaining terms are absorbed into the first and last ones by Young's inequality.}
Applying the Lyapunov condition for the FG-BU integrator in Lemma~\ref{lemma: FG-BU Lyapunov} to the $\mathbb E|\sqrt{\bm S} \bar{\bm Z}_{k+1}|^8$ terms and substituting the bound \eqref{SAGA: moment bound 2} for the moment of $\bm\Delta_k$, we get
	\begin{align*}
		\mathbb E\big|\sqrt{\bm S} \bm Z_{k+1}\big|^8 & \Le \bigg(1+\frac{mh}{12M_2}\bigg)
		\bigg( \bigg(1-\frac{mh}{2M_2}\bigg)
		\mathbb E\big|\sqrt{\bm S} \bm Z_k\big|^8 + \frac{Cd^4h}{m^3} \bigg) + \frac{Cd^4h^5}{m^3}(C_k+1) \\
		& \Le \bigg(1-\frac{mh}{3M_2}\bigg)
		\mathbb E\big|\sqrt{\bm S} \bm Z_k\big|^8 +
		\frac{Cd^4h}{m^3} + \frac{Cd^4h^5}{m^3} C_k.
	\end{align*}
Expressing this inequality in terms of the Lyapunov function $\mathcal V(\bm x,\bm v)$, we arrive at the desired recursive relationship
	\begin{equation*}
		\mathbb E\big[\mathcal V(\bm Z_{k+1})\big] \Le
		\bigg(1-\frac{mh}{3M_2}\bigg) \mathbb E\big[\mathcal V(\bm Z_k)\big] + \frac{Cd^4h}{m^3} + \frac{Ch^5}{m^3} \max_{0\Le l\Le k} \mathbb E\big[\mathcal V(\bm Z_l)\big].
	\end{equation*}
This completes the proof of the Lyapunov condition for the SAGA-BU integrator.
\end{proof}

The recursive Lyapunov condition for SAGA-BU from Lemma~\ref{lemma: SAGA-BU Lyapunov} is sufficient to establish a uniform-in-time moment bound for the full SAGA-UBU integrator. This result requires a sufficiently small step size $h$ to ensure that the contractive dynamics dominate the historical dependency inherent in the SAGA method.
\begin{theorem}
\label{theorem: SAGA-UBU moment bound}
Under Assumptions~\ref{asP} and \ref{asSgNp}, let the step size $h\Le c_0 m$, and
$$
	\bm Z_0 = (\bm X_0,\bm V_0)\text{~with~}\bm X_0 = \bm 0\text{~and~}\bm V_0 \sim \mathcal N(\bm 0,M_2^{-1}\bm I_d);
$$
then the SAGA-UBU solution $(\bm X_k,\bm V_k)_{k=0}^\infty$ satisfies
\begin{equation*}
    \sup_{k\Ge 0} \mathbb E\big[|\bm Z_k|^8\big] \Le \frac{Cd^4}{m^4},
\end{equation*}
where the constants $c_0$ and $C$ depend only on $M_1,M_2$.
\end{theorem}
\begin{proof}
To prove the uniform-in-time moment bound for the SAGA-BU integrator, we analyze the recursive relationship for the maximum normalized moment, defined as
\begin{equation*}
	C_k = \frac1{d^4} \max_{0\Le l\Le k} \mathbb E\big[\mathcal V(\bm Z_l)\big],\qquad
	k = 0,1,\cdots.
\end{equation*}
Writing $a_k = \frac1{d^4}\mathbb E\big[\mathcal V(\bm Z_k)\big]$, so that $C_k = \max_{0\Le l\Le k} a_l$, the Lyapunov condition from Lemma~\ref{lemma: SAGA-BU Lyapunov} divided by $d^4$ provides the following inequality:
\begin{equation}
	a_{k+1} \Le \bigg(1-\frac{mh}{3M_2}\bigg) a_k + \frac{Ch}{m^3} + \frac{Ch^5}{m^3} C_k
	\Le \bigg(1-\frac{mh}{3M_2}\bigg) C_k + \frac{Ch}{m^3} + \frac{Ch^5}{m^3} C_k.
	\label{proof: SAGA Ck}
\end{equation} 
The challenge lies in the final term, $\frac{Ch^5}{m^3} C_k$, which reflects the integrator's dependency on its entire history. To ensure stability, we must choose a step size $h$ small enough for the contractive term to dominate this historical dependency. By setting a threshold \revised{
\begin{equation*}
	c_0 = \min\big\{(6CM_2)^{-1/4},(4M_2)^{-1}\big\},
\end{equation*}
}a step size $h \Le c_0 m$ guarantees that \reviewed{$\frac{Ch^5}{m^3} \Le \frac{mh}{6M_2}$}\revised{, and also that $h\Le\frac{m}{4M_2}\Le\frac14$, as required by Lemma~\ref{lemma: SAGA-BU Lyapunov}}. Writing $\gamma = \frac{m}{3M_2}$ and $\beta = \frac{C}{m^3}$, the inequality \eqref{proof: SAGA Ck} simplifies to
\begin{equation*}
	a_{k+1} \Le \bigg(1-\frac{\gamma h}{2}\bigg) C_k + \beta h.
\end{equation*}
\revised{Next we prove $C_k \Le \max\big\{C_0, \frac{2\beta}{\gamma}+\beta h\big\}$ by induction on $k$. If $C_k \Ge \frac{2\beta}{\gamma}$, then
\begin{equation*}
	a_{k+1} \Le C_k - \frac{\gamma h}{2} C_k + \beta h \Le C_k - \beta h + \beta h = C_k,
\end{equation*}
so that $C_{k+1} = \max\{C_k,a_{k+1}\} = C_k$ and the bound is inherited. If instead $C_k < \frac{2\beta}{\gamma}$, then $a_{k+1} \Le \frac{2\beta}{\gamma} + \beta h$, and again $C_{k+1}\Le\frac{2\beta}{\gamma}+\beta h$.}
Since $\frac{\beta}{\gamma} = \frac{3CM_2}{m^4}$ and $\beta h \Le \frac{\beta}{\gamma}$, the sequence $(C_k)_{k=0}^\infty$ is uniformly bounded:
\begin{equation*}
	\sup_{k \Ge 0} C_k \Le \max\bigg\{C_0, \revised{\frac{C}{m^4}}\bigg\} \Le \revised{\frac{C}{m^4}},
\end{equation*}
\revised{and therefore $\sup_{k\Ge0} \mathbb E\big[\mathcal V(\bm Z_k)\big] \Le d^4 \sup_{k\Ge0} C_k \Le \frac{Cd^4}{m^4}$.}

This establishes the desired uniform-in-time moment bound for the SAGA-BU integrator, \revised{and the same argument extends to SAGA-UBU by conjugation. Writing $\tilde{\bm Z}_k = \Phi^{\mathfrak U}_{h/2} \bm Z_k$ for the SAGA-UBU solution \eqref{SAGA-UBU}, the semigroup property $\Phi^{\mathfrak U}_{h/2}\circ\Phi^{\mathfrak U}_{h/2} = \Phi^{\mathfrak U}_h$ gives
\begin{equation*}
	\tilde{\bm Z}_{k+1} = \Phi^{\mathfrak U}_{h/2}\bm Z_{k+1}
	= \big(\Phi^{\mathfrak U}_{h} \circ \Phi^{\mathfrak B}_h(\{\bm\Phi_i\}_{i=1}^N,\theta_k)\big) \tilde{\bm Z}_k,
\end{equation*}
while the position component of $\tilde{\bm Z}_k$ is exactly $\bm Y_k$, which is the state stored in the table by \eqref{SAGA-UBU}. Hence $(\tilde{\bm Z}_k)_{k=0}^\infty$ is a SAGA-BU solution in the sense of \eqref{SAGA-BU update rule}, and the bound just proved applies to it. Since $\bm Z_{k+1} = \Phi^{\mathfrak U}_{h/2}\big(\Phi^{\mathfrak B}_h(\{\bm\Phi_i\}_{i=1}^N,\theta_k)\tilde{\bm Z}_k\big)$ is obtained from $\tilde{\bm Z}_k$ by one velocity update $-\frac h{M_2}\bm b_k$ followed by the integrator flow $\Phi^{\mathfrak U}_{h/2}$, and $\mathbb E|\bm b_k|^8$ is controlled by the same factorization $\mathbb E|\bm\Phi_i|^8 = \sum_l \mathbb E|\bm Y_l|^8\,\mathbb P(k_i=l) \Le \max_l \mathbb E|\bm Y_l|^8$ used above, we conclude
\begin{equation*}
	\sup_{k\Ge0}\mathbb E|\bm Z_k|^8 \Le C\big(\sup_{k\Ge0}\mathbb E|\tilde{\bm Z}_k|^8 + d^4\big) \Le \frac{Cd^4}{m^4},
\end{equation*}
completing the proof of Theorem~\ref{theorem: SAGA-UBU moment bound}.}
\end{proof}

Due to the non-Markovian nature of the SAGA-UBU solution, a standard invariant distribution cannot be defined. Instead, the numerical bias is analyzed with respect to $\tilde \pi_h(\bm x,\bm v)$, which represents the subsequential weak limit of the time-averaged distributions along the trajectory, as established in the following lemma.
\begin{lemma}
\label{lemma: SAGA-UBU invariant}
Under Assumptions~\ref{asP} and \ref{asSgNp}, let the step size $h \Le c_0m$; define $\tilde{\pi}_h^K$ as the averaged distribution of SAGA-UBU over the first $K$ steps:
	\begin{equation*}
		\tilde{\pi}_h^K(A) = \frac{1}{K}
		\sum_{k=0}^{K-1} \mathbb{P} \big(
		\bm{Z}_k \in A \mid \bm{X}_0 = \bm{0},\ \bm{V}_0 \sim \mathcal{N}(\bm{0}, M_2^{-1}\bm{I}_d)
		\big).
	\end{equation*}
	Then $\{\tilde{\pi}_h^K\}_{K=1}^\infty$ admits a weakly convergent subsequence with limit $\tilde{\pi}_h$ satisfying
	\begin{equation*}
		\int_{\mathbb{R}^{2d}} 
		|\bm z|^8  
		\D\tilde{\pi}_h(\bm{z})
		\leqslant \frac{Cd^4}{m^4},
	\end{equation*}
	where the constants $c_0,C$ depend only on $M_1,M_2$.

    Moreover, for any Lipschitz continuous test function $g(\bm{x}, \bm{v})$ on $\mathbb{R}^{2d}$, $\tilde{\pi}_h(g)$ is the subsequential limit of $\{\tilde{\pi}_h^K(g)\}_{K=1}^\infty$ in $\mathbb R$.
\end{lemma}

\begin{proof}
To prove the existence of a subsequential weak limit $\tilde{\pi}_h$, we introduce a sequence of weighted measures. For each $K \in \mathbb N$, define the positive measure $\tilde \mu_h^K$ on $\mathbb R^{2d}$ via the Radon--Nikodym derivative with respect to the averaged distribution $\tilde \pi_h^K$:
	\begin{equation*}
		\frac{\mathrm{d}\tilde \mu_h^K}{\mathrm{d}\tilde \pi_h^K}(\bm x,\bm v) = |\bm x|^2 + |\bm v|^2 + d.
	\end{equation*}
The uniform-in-time moment of SAGA-UBU in Theorem~\ref{theorem: SAGA-UBU moment bound} implies uniform bounds on the total mass and higher moments of these weighted measures:
	\begin{equation*}
		\int_{\mathbb R^{2d}} \frac{1}{|\bm x|^2 + |\bm v|^2 + d}  \mathrm{d}\tilde \mu_h^K(\bm x,\bm v) = \int_{\mathbb R^{2d}} \mathrm{d}\tilde \pi_h^K(\bm x,\bm v) = 1,
	\end{equation*}
and
    \begin{equation*}
		\int_{\mathbb R^{2d}} (|\bm x|^2 + |\bm v|^2 + d)^3  \mathrm{d}\tilde \mu_h^K(\bm x,\bm v) = \int_{\mathbb R^{2d}} (|\bm x|^2 + |\bm v|^2 + d)^4 \mathrm{d}\tilde \pi_h^K(\bm x,\bm v) \leqslant \frac{Cd^4}{m^4}.
	\end{equation*}
These uniform bounds ensure that the sequence of measures $\{\tilde \mu_h^K\}_{K=1}^\infty$ is tight. By Prokhorov's theorem, there exists a subsequence $\{\tilde \mu_h^{K_j}\}_{j=1}^\infty$ that converges weakly to a limit measure $\tilde \mu_h$. This limit measure inherits the same moment bounds:
	\begin{equation*}
		\int_{\mathbb R^{2d}} \frac{1}{|\bm x|^2 + |\bm v|^2 + d}  \mathrm{d}\tilde \mu_h(\bm x,\bm v) = 1, \quad
		\int_{\mathbb R^{2d}} (|\bm x|^2 + |\bm v|^2 + d)^3  \mathrm{d}\tilde \mu_h(\bm x,\bm v) \leqslant \frac{Cd^4}{m^4}.
	\end{equation*}
We define the probability measure $\tilde \pi_h$ via its Radon--Nikodym derivative with respect to $\tilde \mu_h$:
	\begin{equation*}
		\frac{\mathrm{d}\tilde \pi_h}{\mathrm{d}\tilde \mu_h}(\bm x,\bm v) = \frac{1}{|\bm x|^2 + |\bm v|^2 + d}.
	\end{equation*}
By construction, $\tilde \pi_h$ is the weak limit of the subsequence $\{\tilde \pi_h^{K_j}\}$. Therefore, for any test function $g(\bm x,\bm v)$ with linear growth, the following holds
	\begin{equation*}
		\lim_{j\rightarrow\infty} \int_{\mathbb R^{2d}} g(\bm x,\bm v)  \mathrm{d}\tilde \pi_h^{K_j}(\bm x,\bm v) = 
		\int_{\mathbb R^{2d}} g(\bm x,\bm v)  \mathrm{d}\tilde \pi_h(\bm x,\bm v).
	\end{equation*}
This establishes $\tilde{\pi}_h$ as a subsequential weak limit possessing the stated moment bound and convergence property, completing the proof.
\end{proof}

\subsection{Local error analysis of SAGA-UBU}

A key step in the local error analysis for the SAGA-UBU integrator is to bound the distance between the current intermediate position $\bm Y_k$ and the historical positions $\{\bm\Phi_i\}_{i=1}^N$. The magnitude of this distance directly determines the variance, and thus the accuracy, of the SAGA gradient estimator. The following lemma establishes the required moment bound.
\begin{lemma}
\label{lemma: SAGA distance}
Under Assumptions~\ref{asP} and \ref{asSgNp}, let the step size $h\Le c_0 m$, and
$$
	\bm Z_0 = (\bm X_0,\bm V_0)\text{~with~}\bm X_0 = \bm 0\text{~and~}\bm V_0 \sim \mathcal N(\bm 0,M_2^{-1}\bm I_d);
$$
for each $k\Ge 1$, let $\bm Y_k$ be the position component of $ \Phi^{\mathfrak U}_{h/2}\bm Z_k$ defined in \eqref{Y_k expression}, and $\{\bm\Phi_i\}_{i=1}^N$ be the historical positions of SAGA-UBU in \eqref{Phi definition}, then
\begin{equation*}
	\frac1N \sum_{i=1}^N \mathbb E\big[|\bm Y_k - \bm \Phi_i|^4\big]
	\Le \frac{CN^4d^2h^4}{m^2},
\end{equation*}
where the constants $c_0$ and $C$ depend only on $M_1,M_2$.
\end{lemma}
\begin{proof}
The proof begins by expressing the average fourth moment distance between the current position $\bm Y_k$ and the historical positions $\{\bm \Phi_i\}_{i=1}^N$. For each component $i$, the historical position is $\bm \Phi_i = \bm Y_{k_i}$, where $k_i$ is the last step at which the gradient $\nabla U_i$ was evaluated. By repeatedly applying Hölder's inequality, we can bound this distance by 
	\begin{align*}
		\frac1N \sum_{i=1}^N \mathbb E|\bm Y_k - \bm \Phi_i|^4
		& = \frac1N \sum_{i=1}^N \mathbb E|\bm Y_k - \bm Y_{k_i}|^4 \\
		& \Le \frac1N \sum_{i=1}^N \mathbb E \bigg( \sum_{l=k_i}^{k-1} |\bm Y_{l+1} - \bm Y_l| \bigg)^4 \\
		& \Le \frac1N \sum_{i=1}^N \mathbb E\bigg[(k-k_i)^3 \sum_{l=k_i}^{k-1} |\bm Y_{l+1} - \bm Y_l|^4\bigg] \\
		& = \frac1N \sum_{i=1}^N \sum_{l=0}^{k-1} \mathbb E\Big[(k-k_i)^3 \bm 1_{\{k_i\Le l\}} |\bm Y_{l+1} - \bm Y_l|^4\Big].
	\end{align*}
\reviewed{Unlike the estimate in the proof of Lemma~\ref{lemma: SAGA-UBU stochastic gradient error}, the summand here involves $\bm Y_{l+1}-\bm Y_l$ for $l\Ge k_i$, which is not independent of $\bm 1_{\{k_i\Le l\}}$; the expectation therefore does not factor, and the following split cannot be avoided.} Applying Cauchy's inequality separates the expectation, isolating the term related to the random indices from the term involving the dynamics:
	\begin{equation}
		\frac1N \sum_{i=1}^N \mathbb E|\bm Y_k - \bm \Phi_i|^4
		\Le \frac1N \sum_{i=1}^N \sum_{l=0}^{k-1}
		\sqrt{\mathbb E\big[(k-k_i)^6 \bm 1_{\{k_i\Le l\}}\big]}
		\sqrt{\mathbb E|\bm Y_{l+1} - \bm Y_l|^8}.
		\label{proof: SAGA distance bound}
	\end{equation}
The eighth moment of the single-step deviation, $\bm Y_{l+1} - \bm Y_l$, can be bounded using an argument similar to that in the proof of Lemma~\ref{lemma: displacement}, yielding
	\begin{equation*}
		\sqrt{\mathbb E|\bm Y_{l+1} - \bm Y_l|^8} \Le \frac{Cd^2h^4}{m^2},\qquad l=0,1,\cdots.
	\end{equation*}
Substituting this bound into \eqref{proof: SAGA distance bound} simplifies the expression. Since the random indices $\{k_i\}_{i=1}^N$ are identically distributed, the average over $i$ can be removed:
	\begin{align*}
		\frac1N \sum_{i=1}^N \mathbb E|\bm Y_k - \bm \Phi_i|^4
		& \Le \frac{Cd^2h^4}{Nm^2} \sum_{i=1}^N \sum_{l=0}^{k-1}
		\sqrt{\mathbb E\big[(k-k_i)^6 \bm 1_{\{k_i\Le l\}}\big]} \\
		& = \frac{Cd^2h^4}{m^2} \sum_{l=0}^{k-1}
		\sqrt{\mathbb E\big[(k-k_1)^6 \bm 1_{\{k_1\Le l\}}\big]}.
	\end{align*}
Thus, the proof of Lemma~\ref{lemma: SAGA distance} is reduced to establishing the following purely probabilistic inequality concerning the distribution of the random index $k_1$:
\begin{equation}
	\sum_{l=0}^{k-1} \sqrt{\mathbb E\big[(k-k_1)^6 \bm 1_{\{k_1\Le l\}}\big]} \Le CN^4.
	\label{proof: SAGA pure inequality}
\end{equation}

To establish the probabilistic inequality \eqref{proof: SAGA pure inequality}, we formally define the random variable $k_1$. It signifies the index of the last iteration before step $k$ where the gradient $\nabla U_1(\bm x)$ was computed. Given the sequence of randomly chosen indices $\theta_1, \dots, \theta_{k-1}$, $k_1$ is defined as
	\begin{equation*}
		k_1 = \left\{
		\begin{aligned}
			& \max\{l:\theta_l=1\}, && \text{if $\exists\, l\in\{1,\cdots,k-1\}$ with $\theta_l=1$}, \\
			& 0, && \text{otherwise}.
		\end{aligned}
		\right.
	\end{equation*}
Due to the i.i.d. nature of the index selection, a time-reversal symmetry argument shows that the distribution of the ``age'' of the gradient, $k-k_1$, is identical to the distribution of the ``waiting time'' for the first selection of index 1, which we denote by $k_1'$:
	\begin{equation*}
		k_1' = \left\{
		\begin{aligned}
			& \min\{l:\theta_l=1\}, && \text{if $\exists\, l\in\{1,\cdots,k-1\}$ with $\theta_l=1$}, \\
			& k, && \text{otherwise}.
		\end{aligned}
		\right.
	\end{equation*}
\revised{The map $l\mapsto k-l$ is a bijection of $\{0,\cdots,k-1\}$ onto $\{1,\cdots,k\}$, and $\{k_1\Le l\} = \{k-k_1\Ge k-l\}$. Combining this reindexing with the identity in law $k-k_1 \overset{\mathrm d}{=} k_1'$ transforms \eqref{proof: SAGA pure inequality} into the equivalent statement}
	\begin{equation*}
		\sum_{l=1}^k \sqrt{\mathbb E\big[(k_1')^6\bm 1_{\{k_1'\Ge l\}}\big]} \Le CN^4.
	\end{equation*}
Expanding the expectation inside the square root, we only need to prove
	\begin{equation}
		\sum_{l=1}^k \sqrt{\sum_{s=l}^k s^6 \cdot \mathbb P(k_1'=s)}
		\Le CN^4.
		\label{proof: SAGA pure inequality 2}
	\end{equation}

The random variable $k_1'$ follows a truncated geometric distribution. As a consequence, the probability $\mathbb P(k_1'=s)$ is given by the standard geometric probability:
	\begin{equation*}
		\mathbb P(k_1' = s) =
		\bigg(1-\frac1N\bigg)^{s-1} \frac1N, \qquad s = 1,\cdots,k-1.
	\end{equation*}
The event $k_1'=k$ occurs if index 1 is not selected in any of the first $k-1$ steps:
	\begin{equation*}
		\mathbb P(k_1' = k) =
		\bigg(1-\frac1N\bigg)^{k-1}.
	\end{equation*}
Conveniently, this probability can be expressed as the tail sum of the corresponding infinite geometric distribution: $\mathbb P(k_1'=k) = \sum_{s=k}^\infty (1-\frac{1}{N})^{s-1} \frac{1}{N}$.

Substituting this distribution into \eqref{proof: SAGA pure inequality 2} allows us to replace the finite sum up to $k$ with an infinite sum, effectively removing the dependency on $k$ and reducing the problem to proving the following algebraic inequality:
	\begin{equation}
		\sum_{l=1}^\infty \sqrt{\sum_{s=l}^\infty s^6 \bigg(1-\frac1N\bigg)^{s-1}\frac1N} \Le CN^4.
		\label{proof: SAGA pure inequality 3}
	\end{equation}
To establish \eqref{proof: SAGA pure inequality 3}, we bound the terms in the inner sum. \revised{Assume $N\Ge2$; the case $N=1$ is trivial, since then $k_1=k-1$ almost surely and the left-hand side of \eqref{proof: SAGA pure inequality} equals $1$.} Let $j = \lfloor s/N \rfloor$. For $s$ such that $jN \Le s < (j+1)N$, \revised{we have $s^6 \Le (j+1)^6N^6$, and combining $(1-1/N)^{jN} \Le e^{-j}$ with $(1-1/N)^{-1} = \frac{N}{N-1} \Le 2$ gives}
	\begin{equation*}
		s^6 \bigg(1-\frac1N\bigg)^{s-1} \Le (j+1)^6N^6 \bigg(1-\frac1N\bigg)^{jN-1} \Le \revised{2(j+1)^6 N^6 e^{-j}}.
	\end{equation*}
\revised{Each block $\{s: jN\Le s<(j+1)N\}$ contains $N$ values of $s$, so the factor $\frac1N$ carried by $\mathbb P(k_1'=s)$ cancels against the block size:
	\begin{equation*}
		\sum_{s=l}^\infty s^6\bigg(1-\frac1N\bigg)^{s-1}\frac1N
		\Le \frac1N \sum_{j\Ge\lfloor l/N\rfloor} N\cdot 2(j+1)^6N^6e^{-j}
		= 2N^6 \sum_{j\Ge\lfloor l/N\rfloor} (j+1)^6 e^{-j}.
	\end{equation*}}
The inequality \eqref{proof: SAGA pure inequality 3} then simplifies to
	\begin{equation}
		\sum_{l=1}^\infty \sqrt{\revised{2N^6\sum_{j=\lfloor l/N \rfloor}^\infty (j+1)^6 e^{-j}}} \Le CN^4.
		\label{proof: SAGA pure inequality 4}
	\end{equation}
\revised{Factoring $N^3$ out of the square root, and grouping the outer sum into blocks of $N$ consecutive values of $l$ sharing the same $i = \lfloor l/N\rfloor$, contributes a further factor $N$; since $N^3\cdot N = N^4$,} we arrive at the final simplified inequality, which is independent of $N$:
	\begin{equation}
	\sum_{i=0}^\infty \sqrt{\sum_{j=i}^\infty \revised{(j+1)^6}e^{-j}} \Le C.
		\label{proof: SAGA pure inequality 5}
	\end{equation}
This final inequality \eqref{proof: SAGA pure inequality 5} holds because the exponential term $e^{-j}$ decays faster than any polynomial in $j$. Specifically, using the bound $e^{-j} \Le C(j+1)^{-10}$, we have:
	\begin{equation*}
		\sum_{j=i}^\infty \revised{(j+1)^6}e^{-j} \Le C \sum_{j=i}^\infty \frac{\revised{(j+1)^6}}{(j+1)^{10}} = C \sum_{j=i}^\infty \frac1{\revised{(j+1)^4}} \Le \frac{C}{\revised{(i+1)^3}}.
	\end{equation*}
Therefore, the outer sum converges:
	\begin{equation*}
		\sum_{i=0}^\infty \sqrt{\sum_{j=i}^\infty \revised{(j+1)^6}e^{-j}}
		\Le C \sum_{\revised{i=0}}^\infty \frac1{\revised{(i+1)^{3/2}}} < +\infty.
	\end{equation*}
This confirms \eqref{proof: SAGA pure inequality 5}, which in turn proves \eqref{proof: SAGA pure inequality} and thus establishes Lemma~\ref{lemma: SAGA distance}.
\end{proof}

Building on the previous results, we now establish bounds for the stochastic gradient error of the SAGA-UBU integrator.
\begin{lemma}
	\label{lemma: SAGA-UBU stochastic gradient error}
	Under Assumptions~\ref{asP} and \ref{asSgNp}, let the step size $h\Le c_0 m$, and
	$$
	\bm Z_0 = (\bm X_0,\bm V_0)\text{~with~}\bm X_0 = \bm 0\text{~and~}\bm V_0 \sim \mathcal N(\bm 0,M_2^{-1}\bm I_d);
	$$
    then the stochastic gradient error of SAGA-UBU with the batch size $1$ satisfies
	\begin{align*}
		\sup_{k\Ge0}\sqrt{\mathbb E\big|\bm X_{k+1} - \bar{\bm X}_{k+1}\big|^4} &
		\Le \frac{Cdh^4}{m}\min\big\{1,N^2h^2\big\}, \\
		\sup_{k\Ge0}\sqrt{\mathbb E\big|\bm V_{k+1} - \bar{\bm V}_{k+1}\big|^4} &
		\Le \frac{Cdh^2}{m}\min\big\{1,N^2h^2\big\},
	\end{align*}
    where $c_0$ depends only on $(M_i)_{i=1}^2$, and $C$  depends only on $(M_i)_{i=1}^3$.
\end{lemma}
\begin{proof}
The SAGA stochastic gradient deviation at step $k$, denoted by $\bm g_k(\theta_k)$, is defined as
	\begin{equation}
		\bm g_k(\theta_k) = \nabla U_{\theta_k}(\bm Y_k) - \nabla U_{\theta_k}(\bm\Phi_{\theta_k}) + \frac1N \sum_{i=1}^N
		\nabla U_i(\bm\Phi_i) - \frac1N \sum_{i=1}^N \nabla U_i(\bm Y_k).
        \label{proof: SAGA sg 0}
	\end{equation}
As established in \eqref{proof: X difference}, the stochastic gradient error $\bm Z_{k+1} - \bar{\bm Z}_{k+1}$ is directly proportional to this deviation term $\bm g_k(\theta_k)$. Therefore, bounding the moments of $\bm g_k(\theta_k)$ is sufficient to control the stochastic gradient error. The specific target bound required is
	\begin{equation}
		\sqrt{\mathbb E|\bm g_k(\theta_k)|^4} \Le \frac{C d}{m}\min\big\{1,N^2h^2\big\}.
		\label{proof: SAGA sg 1}
	\end{equation}

We establish the required bound \eqref{proof: SAGA sg 1} by deriving two separate estimates for $\mathbb E|\bm g_k(\theta_k)|^4$ and then taking the minimum.

First, a uniform bound independent of $h$ can be obtained using an argument similar to the one leading to \eqref{SAGA: moment bound 1}. We start with the triangle inequality \eqref{SAGA: gk}:
\begin{equation*}
		|\bm g_k(\theta_k)| \Le \bigg|\nabla U_{\theta_k}(\bm Y_k)- \frac1N \sum_{i=1}^N \nabla U_i(\bm Y_k)\bigg| + |\nabla U_{\theta_k} (\bm \Phi_{\theta_k})| + \frac1N \sum_{i=1}^N |\nabla U_i(\bm \Phi_i)|.
\end{equation*}
Using Assumption~\ref{asSgNp} and Theorem~\ref{theorem: SAGA-UBU moment bound} (since $\bm \Phi_i = \bm Y_{k_i}$ are previous states), we obtain
	\begin{align*}
		\mathbb E|\bm g_k(\theta_k)|^4
		\Le C \bigg( d^2 + \frac1N \sum_{i=1}^N \mathbb E\big(|\bm \Phi_i|^2+d\big)^2 \bigg) \Le C \bigg( d^2 + \frac1N \sum_{i=1}^N \mathbb E|\bm Y_{k_i}|^4 \bigg).
	\end{align*}
	\reviewed{The historical index $k_i$ takes values in $\{0,\cdots,k-1\}$, and the events $\{k_i=l\}_{l=0}^{k-1}$ partition the sample space, thus}
	\begin{align*}
		\reviewed{\mathbb E\big|\bm Y_{k_i}\big|^4}
		& \reviewed{{}= \sum_{l=0}^{k-1} \mathbb E\big(|\bm Y_l|^4\, \bm 1_{\{k_i=l\}}\big)
		= \sum_{l=0}^{k-1} \mathbb E|\bm Y_l|^4\, \mathbb P(k_i=l)} \\
		& \reviewed{{}\Le \Big(\max_{0\Le l\Le k-1} \mathbb E|\bm Y_l|^4\Big) \sum_{l=0}^{k-1} \mathbb P(k_i=l)
		= \max_{0\Le l\Le k-1} \mathbb E|\bm Y_l|^4.}
	\end{align*}
	\reviewed{The factorization uses the independence of $\bm Y_l$ and $\bm 1_{\{k_i=l\}}$: by \eqref{Phi definition},}
	\begin{equation*}
		\reviewed{\{k_i = l\} = \{\theta_l = i\} \cap \bigcap_{m=l+1}^{k-1} \{\theta_m \ne i\},}
	\end{equation*}
	\reviewed{so $\bm 1_{\{k_i=l\}}$ depends only on $\theta_l,\cdots,\theta_{k-1}$, whereas Algorithm~\ref{algorithm: SAGA-UBU} evaluates $\bm Y_l$ before $\theta_l$ is drawn. The random indices are independent of each other and of the Brownian motion.}

	\reviewed{Theorem~\ref{theorem: SAGA-UBU moment bound} controls that maximum: since $\mathcal V(\bm z) = (\bm z^\top\bm S\bm z)^4$ with $\bm S$ positive definite, we have $\mathbb E|\bm Y_l|^8 \Le C\,\mathbb E[\mathcal V(\bm Z_l)] \Le \frac{Cd^4}{m^4}$ uniformly in $l$, and Jensen's inequality gives $\mathbb E|\bm Y_l|^4 \Le \big(\mathbb E|\bm Y_l|^8\big)^{\frac12} \Le \frac{Cd^2}{m^2}$. Therefore}
	\begin{equation*}
		\reviewed{\mathbb E|\bm g_k(\theta_k)|^4 \Le C \bigg(d^2 + \frac{d^2}{m^2}\bigg) \Le \frac{Cd^2}{m^2},}
	\end{equation*}
	providing the first part of the minimum bound
\begin{equation}
\sqrt{\mathbb E|\bm g_k(\theta_k)|^4} \Le \frac{Cd}{m}.
\label{proof: SAGA local 2}
\end{equation}

Starting from the definition of $\bm g_k(\theta_k)$, we apply the triangle inequality to obtain
	\begin{align*}
		|\bm g_k(\theta_k)|
		& \Le \big|\nabla U_{\theta_k}(\bm Y_k) - \nabla U_{\theta_k}(\bm \Phi_{\theta_k})\big| +
		\frac1N \sum_{i=1}^N \big|\nabla U_i(\bm Y_k) - \nabla U_i(\bm\Phi_i)\big| \\
		& \Le C |\bm Y_k - \bm \Phi_{\theta_k}| + \frac CN \sum_{i=1}^N |\bm Y_k - \bm \Phi_i|.
	\end{align*}
Taking the fourth moment and the full expectation over the trajectory and $\theta_k$ gives
	\begin{align*}
		\mathbb E|\bm g_k(\theta_k)|^4
		& \Le C\,\mathbb E|\bm Y_k-\bm\Phi_{\theta_k}|^4 + C \,\mathbb E \Bigg[ \bigg(\frac{1}{N} \sum_{i=1}^N |\bm Y_k - \bm \Phi_i| \bigg)^4 \Bigg]  \\
        & \Le C\,\mathbb E|\bm Y_k-\bm\Phi_{\theta_k}|^4 + \frac{C}{N} \sum_{i=1}^N \mathbb E|\bm Y_k - \bm \Phi_i|^4 = \frac{C}{N} \sum_{i=1}^N \mathbb E|\bm Y_k - \bm \Phi_i|^4. 
	\end{align*}
Applying the bound from Lemma~\ref{lemma: SAGA distance} yields the $h$-dependent estimate:
	\begin{equation}
		\sqrt{\mathbb E|\bm g_k(\theta_k)|^4}
		 \Le \frac{CN^2dh^2}{m}.
		\label{proof: SAGA local 3}
	\end{equation}
Combining the uniform bound implied by \eqref{proof: SAGA local 2} and the $h$-dependent bound from \eqref{proof: SAGA local 3}, we obtain the desired composite bound \eqref{proof: SAGA sg 1}.
\end{proof}

\subsection{MSE of SAGA-UBU}
The MSE bound for SAGA-UBU stated in Theorem~\ref{theorem: SAGA-UBU MSE} is derived using the same analytical framework employed for SG-UBU and SVRG-UBU. The proof combines the uniform-in-time moment bound from Theorem~\ref{theorem: SAGA-UBU moment bound} and the local error estimates from Lemma~\ref{lemma: SAGA-UBU stochastic gradient error}.
\begin{proofof}[Theorem~\ref{theorem: SAGA-UBU MSE}]
For SAGA-UBU, we note that the stochastic gradient error bound from Lemma~\ref{lemma: SAGA-UBU stochastic gradient error} exhibits a stronger dependence on the parameter $m$. Therefore, the contribution of this error to the total MSE must be bounded carefully.

Following the proof strategy for SG-UBU (Theorem~\ref{theorem: SG-UBU MSE}), we decompose the stochastic gradient error term $R_k$ as:
\begin{align*}
		R_{k,1} & = \frac1h (\bm Z_{k+1} - \bar {\bm Z}_{k+1}) ^\top \nabla \phi_h(\bar{\bm Z}_{k+1}), \\
		R_{k,2} & = \frac1h (\bm Z_{k+1} - \bar {\bm Z}_{k+1})^\top \int_0^1 \Big(\nabla \phi_h\big(\lambda \bm Z_{k+1} + (1-\lambda) \bar{\bm Z}_{k+1}\big)-\nabla\phi_h(\bar {\bm Z}_{k+1})\Big)\D\lambda.
\end{align*}
This decomposition leads to the following moment bounds:
\begin{align*}
\mathbb E[R_{k,1}^2] & \Le \frac{C}{m^2h^2}
\mathbb E|\bm Z_{k+1} - \bar{\bm Z}_{k+1}|^2 \Le \frac{Cd}{m^3}\min\big\{1,N^2h^2\big\} \Le \frac{Cd}{m^3}, \\
\mathbb E[R_{k,2}^2] & \Le \frac{C}{m^4h^2}
\mathbb E|\bm Z_{k+1} - \bar{\bm Z}_{k+1}|^4 \Le \frac{Cd^2h^2}{m^6}\min\big\{1,N^4h^4\big\}.
\end{align*}
Then the contribution of the stochastic gradient error $\sum_{k=0}^{K-1} R_k$ to the MSE is bounded by
\begin{align*}
\frac1{K^2} \mathbb E\Bigg[
\bigg(\sum_{k=0}^{K-1} R_k\bigg)^2
\Bigg] & \Le \frac{C}{K^2}\sum_{k=0}^{K-1}
\mathbb E[R_{k,1}^2] + \frac{C}{K}
\sum_{k=0}^{K-1} \mathbb E[R_{k,2}^2] \\
& \Le C\bigg(\frac{d}{m^3K} + \frac{d^2h^2}{m^6}\min\big\{1,N^4h^4\big\}\bigg).
\end{align*}
The analysis of the discretization error and martingale components remains unchanged.
Therefore, the total MSE is bounded by
\begin{equation*}
		\mathbb E \Bigg[\bigg(\frac1K\sum_{k=0}^{K-1} f(\bm X_k) - \pi(f)\bigg)^2\Bigg] \Le
		C\bigg(\frac{d}{m^3Kh} + \frac{d^2h^2}{m^6} \min\big\{1,N^4h^4\big\}+ \reviewed{\frac{d^2h^4}{m^5}\bigg(1+\frac{dh}{m^2}\bigg)} \bigg),
\end{equation*}
completing the proof of Theorem~\ref{theorem: SAGA-UBU MSE}.
\end{proofof}

\subsection{Numerical bias of SAGA-UBU}

\begin{proofof}[Theorem~\ref{theorem: SAGA-UBU numerical bias}]
We begin with the time average error decomposition in \eqref{explicit time average}:
	\begin{equation*}
		\frac{1}{K} \sum_{k=0}^{K-1} f(\bm{X}_k) - \pi(f) = 
		\frac{\phi_h(\bm{Z}_0) - \phi_h(\bm{Z}_K)}{Kh} + 
		\frac{1}{K}\sum_{k=0}^{K-1} (R_k + S_k + T_k),
	\end{equation*}
where $R_k$, $S_k$ and $T_k$ represent the stochastic gradient error, discretization error, and martingale difference term, respectively.
To analyze the numerical bias, we take the expectation of this expression. Since the martingale term satisfies $\mathbb{E}[T_k]=0$, the expectation becomes
	\begin{equation}
		\mathbb{E}\bigg[\frac{1}{K} \sum_{k=0}^{K-1} f(\bm{X}_k)\bigg] - \pi(f) = 
		\frac{\mathbb{E}[\phi_h(\bm{Z}_0) - \phi_h(\bm{Z}_K)]}{Kh} + 
		\frac{1}{K} \sum_{k=0}^{K-1} \big( \mathbb{E}[R_k] + \mathbb{E}[S_k] \big).
		\label{proof: SAGA bias 1}
	\end{equation}
The subsequent analysis focuses on bounding the expectations of the stochastic error term $\mathbb E[R_k]$ and the discretization error term $\mathbb E[S_k]$.

For the stochastic gradient error $R_k$, we use its integral representation
	\begin{equation*}
		R_k = \frac{\phi_h(\bm{Z}_{k+1}) - \phi_h(\bar{\bm{Z}}_{k+1})}{h} = \frac{1}{h} 
		(\bm{Z}_{k+1} - \bar{\bm{Z}}_{k+1})^\top
		\int_0^1 \nabla \phi_h\big(\lambda \bm{Z}_{k+1} +
		(1-\lambda) \bar{\bm{Z}}_{k+1}
		\big) \mathrm{d}\lambda.
	\end{equation*}
Following the decomposition strategy used in the proof of Theorem~\ref{theorem: SG-UBU MSE}, we split $R_k$ into a martingale difference term $R_k^1$ and a higher-order remainder $R_k^2$:
	\begin{align*}
		R_k^1 & = \frac{1}{h}(\bm{Z}_{k+1} - \bar{\bm{Z}}_{k+1})^\top
		\nabla \phi_h(\bar{\bm{Z}}_{k+1}), \\
		R_k^2 & = \frac{1}{h} (\bm{Z}_{k+1} - \bar{\bm{Z}}_{k+1})^\top 
		\int_0^1 \Big(\nabla \phi_h\big(\lambda \bm{Z}_{k+1} + (1-\lambda) \bar{\bm{Z}}_{k+1}\big) - \nabla \phi_h(\bar{\bm{Z}}_{k+1}) \Big) \mathrm{d}\lambda.
	\end{align*}
The martingale property ensures $\mathbb{E}[R_k^1] = 0$. For the remainder $R_k^2$, we apply the Hessian bound from Theorem~\ref{theorem: phi h estimate} and the stochastic gradient error bounds from Lemma~\ref{lemma: SAGA-UBU stochastic gradient error} to obtain
	\begin{equation*}
		\mathbb{E}|R_k^2| \leqslant \frac{C}{m^2 h} \mathbb{E} \big| \bm{Z}_{k+1} - \bar{\bm{Z}}_{k+1} \big|^2 \leqslant \frac{Cdh}{m^3}\min\big\{1,N^2h^2\big\}.
	\end{equation*}
Since $\mathbb E[R_k] = \mathbb E[R_k^2]$, the average expected stochastic gradient error is bounded by:
	\begin{equation}
		\bigg| \frac{1}{K}\sum_{k=0}^{K-1} 
		\mathbb{E}[R_k] \bigg|
        \leqslant \frac{Cdh}{m^3}\min\big\{1,N^2h^2\big\}.
		\label{proof: SAGA bias 2}
	\end{equation}
For the discretization error $S_k$, \revised{the estimate \eqref{proof: S_0 estimate} established in the proof of Theorem~\ref{theorem: SG-UBU numerical bias} applies verbatim: it involves $\bm Z_k$ only through the uniform-in-time moment bounds, which Theorem~\ref{theorem: SAGA-UBU moment bound} supplies at the same order for SAGA-UBU. Hence}
	\begin{equation}
		\bigg|\frac{1}{K} \sum_{k=0}^{K-1} \mathbb{E}[S_k]\bigg| \leqslant 
    \frac{Cdh^2}{m^3}.
		\label{proof: SAGA bias 3}
	\end{equation}

Finally, we need to bound the boundary term $\mathbb{E}[\phi_h(\bm{Z}_0) - \phi_h(\bm{Z}_K)] / (Kh)$ in \eqref{proof: SAGA bias 1}. Using the Lipschitz property of $\phi_h$ from Theorem~\ref{theorem: phi h estimate} and an argument similar to Lemma~\ref{lemma: displacement}, we can estimate the displacement moment:
	\begin{equation}
		\mathbb E\big|\phi_h(\bm Z_0) - \phi_h(\bm Z_K)\big| \leqslant \frac{C}{m} \mathbb E|\bm Z_0 - \bm Z_K| \leqslant C\sqrt{\frac{dKh}{m^3}}.
		\label{proof: SAGA bias 4}
	\end{equation}
Substituting the bounds for the expected stochastic error \eqref{proof: SAGA bias 2}, the discretization error \eqref{proof: SAGA bias 3}, and the boundary term \eqref{proof: SAGA bias 4} into the expression \eqref{proof: SAGA bias 1} yields an estimate for the time average bias over $K$ steps:
	\begin{equation*}
		\Bigg|\mathbb{E}\bigg[\frac{1}{K} \sum_{k=0}^{K-1} f(\bm{X}_k)\bigg] - \pi(f)  \Bigg| \leqslant C\Bigg( \sqrt{\frac{d}{m^3Kh}} + 
		\frac{dh}{m^3}\min\big\{1,N^2h^2\big\}+
    \frac{dh^2}{m^3}\Bigg).
	\end{equation*}
As established in Lemma~\ref{lemma: SAGA-UBU invariant}, the time-averaged distribution $\tilde \pi_h^K$ converges weakly along a subsequence as $K\rightarrow\infty$ to the limit distribution $\tilde{\pi}_h$. Taking the subsequential limit $K\rightarrow\infty$, the boundary term vanishes, resulting in the final bound for the numerical bias:
	\begin{equation*}
		|\tilde{\pi}_h(f) - \pi(f)| \leqslant C\bigg(\frac{dh}{m^3}\min\big\{1,N^2h^2\big\} + 
    \frac{dh^2}{m^3}\bigg).
	\end{equation*}
This bound exactly matches the result stated in Theorem~\ref{theorem: SAGA-UBU numerical bias}.
\end{proofof}

\section{Numerical calculation of leading-order coefficient}
\label{appendix: leading-order coefficient}

In this section, we present the computational method for the leading-order coefficient \eqref{h coefficient} in the numerical bias of SG-UBU. This calculation requires computing the quantity
\begin{equation*}
	C_0 = \mathbb E^{\pi} \Big[
	\big(b(\bm x_0,\theta) - \nabla U(\bm x_0)\big)^\top
	\nabla_{\bm v\bm v}^2 \phi_0(\bm x_0,\bm v_0)
	\big(b(\bm x_0,\theta) - \nabla U(\bm x_0)\big)
	].
\end{equation*}
We define the noise outer-product matrix $\bm E(\bm x_0) \in \mathbb R^{d\times d}$ as
\begin{equation*}
	\bm E(\bm x_0) = \mathbb E_\theta\Big[\big(b(\bm x_0,\theta) - \nabla U(\bm x_0)\big) \big(b(\bm x_0,\theta) - \nabla U(\bm x_0)\big)^\top\Big],
\end{equation*}
which allows $C_0$ to be written equivalently using the trace:
\begin{equation*}
	C_0 = \mathbb E^{\pi} \Big[\Tr(\bm E(\bm x_0) \nabla_{\bm v\bm v}^2 \phi_0(\bm x_0,\bm v_0)) \Big].
\end{equation*}
Since the continuous Poisson solution $\phi_0(\bm x,\bm v)$ is independent of the specific stochastic gradient model, we can compute the Hessian matrix $\nabla_{\bm v\bm v}^2 \phi_0(\bm x,\bm v)$ and apply this result to both the Gaussian noise and finite-sum cases.

Using the Kolmogorov solution $u(\bm x,\bm v,t)$ from \eqref{function u}, the Hessian matrix $\nabla_{\bm v\bm v}^2 \phi_0(\bm x,\bm v)$ is
\begin{equation*}
	\nabla_{\bm v\bm v}^2 \phi_0(\bm x_0,\bm v_0) =
	\int_0^\infty \mathbb E^{(\bm x_0,\bm v_0)}\Big[
	(D_{\bm v\bm v} \bm x_t)^\top
	\nabla f(\bm x_t) + (D_{\bm v}\bm x_t)^\top \nabla^2 f(\bm x_t) D_{\bm v}\bm x_t
	\Big],
\end{equation*}
where $(\bm x_t,\bm v_t)_{t\Ge0}$ solves the exact underdamped Langevin dynamics \eqref{ULD} with the initial state $(\bm x_0,\bm v_0)$, and $D_{\bm v} \bm x_t$ and $D_{\bm v\bm v} \bm x_t$ are the first and second variation processes, respectively.

Here, the second variation process can be derived by taking derivatives with respect to $\bm v$ in the first variation process \eqref{first variation}, and it is governed by the coupled ODE system:
\begin{equation}
	\left\{
	\begin{aligned}
		\frac{\D}{\D t} D_{\bm v\bm v} \bm x_t & = D_{\bm v\bm v} \bm v_t, \\
		\frac{\D}{\D t} D_{\bm v\bm v} \bm v_t & = -\frac{1}{M_2}\Big(\nabla^3 U(\bm x_t) \langle D_{\bm v} \bm x_t, D_{\bm v} \bm x_t \rangle + \nabla^2 U(\bm x_t) D_{\bm v\bm v} \bm x_t\Big) - 2D_{\bm v\bm v} \bm v_t,
	\end{aligned}
	\right.
	\label{second variation 1}
\end{equation}
with the initial configuration $D_{\bm v\bm v} \bm x_0 = D_{\bm v\bm v} \bm v_0 = \bm O_{d\times d \times d}$, which are zero third-order tensors. The tensor quadratic form $\bm B\langle\bm A,\bm A\rangle \in \mathbb R^{d\times d\times d}$ is defined for $\bm B \in \mathbb R^{d\times d\times d}$ and $\bm A \in \mathbb R^{d\times d}$ by
\begin{equation*}
	\big(\bm B \langle\bm A,
	\bm A \rangle\big)_{kmn} :=  \sum_{i,j=1}^d B_{ijk} A_{im} A_{jn}, \qquad
	k,m,n = 1,\cdots,d.
\end{equation*}

For the discrete-time setting, let $(\bm X_k,\bm V_k)_{k=0}^\infty$ be the FG-UBU solution with step size $h>0$ starting from $(\bm x_0,\bm v_0)$. Let $(D_{\bm v} \bm X_k)_{k=0}^\infty$ and $(D_{\bm v\bm v} \bm X_k)_{k=0}^\infty$ be the corresponding first and second variation processes, both solved using an exponential integrator. The Hessian matrix $\nabla_{\bm v\bm v}^2 \phi_0(\bm x_0,\bm v_0)$ can then be approximated by the sum
\begin{equation*}
	\nabla_{\bm v\bm v}^2 \phi_0(\bm x_0,\bm v_0) \approx h \sum_{k=0}^\infty \mathbb E^{(\bm x_0,\bm v_0)} \Big[
	(D_{\bm v\bm v} \bm X_k)^\top
	\nabla f(\bm X_k) + (D_{\bm v}\bm X_k)^\top \nabla^2 f(\bm X_k) D_{\bm v}\bm X_k
	\Big].
\end{equation*}
Note that the second variation $D_{\bm v\bm v} \bm x_t$ is only guaranteed to decay exponentially almost surely when the potential function $U(\bm x)$ is globally convex. For non-convex potentials, this computation of $\nabla_{\bm v\bm v}^2 \phi_0(\bm x_0,\bm v_0)$ may be unstable.

\end{document}